\documentclass[10pt, a4paper]{article}
\usepackage{amsmath,amsthm,stackengine}
\usepackage{amsfonts}
\usepackage{mathrsfs}
\usepackage{mathtools}
\usepackage{hyperref}
\usepackage{amssymb}
\usepackage{epsfig}
\usepackage{dsfont}
\usepackage{stmaryrd}
\usepackage{mathtools}
\usepackage[cal=cm]{mathalfa}
\usepackage{nicefrac}
\usepackage{float}
\restylefloat{table}
\restylefloat{figure}
\usepackage{graphicx}
\usepackage{bm}
\usepackage[nottoc]{tocbibind}
\usepackage[top=1in,bottom=1in,left=1.25in,right=1.25in]{geometry}
\usepackage{enumitem}
\usepackage{pgfplots}
\usepackage{pgfplotstable}
\usepackage{lipsum}
\usepackage{siunitx}
\usepackage[belowskip=-20pt]{caption} 
\usepackage{tikz}
\usetikzlibrary{arrows,shapes.geometric}
\usetikzlibrary{patterns}
\usepackage{pst-eucl}
\usepackage{multido,fp}
\pgfplotsset{width=7cm,compat=1.8}
\usepackage{caption}
\usepackage{subcaption}
\usepackage{fancyhdr}
\hypersetup{hidelinks}

\newtheorem{proposition}{Proposition}[section]
\newtheorem{theorem}[proposition]{Theorem}
\newtheorem{lemma}[proposition]{Lemma}
\newtheorem{example}[proposition]{Example}
\newtheorem*{lemma*}{Lemma}

\newtheorem*{theorem*}{Theorem}
\theoremstyle{definition}
\newtheorem{definition}[proposition]{Definition}

\newtheorem{remark}{Remark}
\newcommand{\T}{\mathcal{M}}
\newcommand{\e}{\mathcal{E}}
\newcommand{\tT}{\mathcal{T}}
\newcommand{\eE}{\widehat{\mathcal{E}}}
\newcommand{\pid}{G}
\newcommand{\p}{\mathcal{P}}
\newcommand{\bn}{\textbf{n}}
\newcommand{\bt}{\boldsymbol{\tau}}

\newcommand{\nc}{\mathrm{nc}}

\newcommand{\mn}{{\cal{M}}_{\bn\bn}}

\newcommand{\mt}{{\cal{M}}_{\bn\bt}}
\newcommand{\h}{\mathrm{HCT}}
\newcommand{\pw}{\mathrm{pw}}
\newcommand{\PF}{\mathrm{PF}}

\newcommand{\M}{\mathrm{M}}
\newcommand{\ms}{h_{\mathrm{max}}}
\newcommand{\vertiii}[1]{{\left\vert\kern-0.25ex\left\vert\kern-0.25ex\left\vert #1 
		\right\vert\kern-0.25ex\right\vert\kern-0.25ex\right\vert}}
\def\Xint#1{\mathchoice
	{\XXint\displaystyle\textstyle{#1}}%
	{\XXint\textstyle\scriptstyle{#1}}%
	{\XXint\scriptstyle\scriptscriptstyle{#1}}%
	{\XXint\scriptscriptstyle\scriptscriptstyle{#1}}%
	\!\int}
\def\XXint#1#2#3{{\setbox0=\hbox{$#1{#2#3}{\int}$ }
		\vcenter{\hbox{$#2#3$ }}\kern-.6\wd0}}

\def\dashint{\Xint-}
\title{Nonconforming virtual elements for the biharmonic equation with  Morley degrees of freedom on polygonal meshes}
\author{Carsten Carstensen\thanks{ Department of Mathematics, Humboldt-Universit\"{a}t zu Berlin, 10099 Berlin, Germany.
		Email: cc@math.hu-berlin.de},  Rekha Khot\thanks{Department of Mathematics, Indian Institute of Technology Bombay, Powai, Mumbai, 400076. Email: rekhamp@math.iitb.ac.in, akp@math.iitb.ac.in} \; and Amiya K. Pani\footnotemark[2]}
\date{}	
\begin{document}
	\maketitle
		\begin{abstract}
			The lowest-order nonconforming virtual element  extends the Morley triangular element  to polygons for the approximation of the weak solution $u\in V:=H^2_0(\Omega)$ to the biharmonic equation. The abstract framework allows (even a mixture of) two examples of the local discrete spaces $V_h(P)$ and a smoother allows rough source terms $F\in V^*=H^{-2}(\Omega)$.   The \textit{a priori} and \textit{a posteriori} error analysis in this paper circumvents any trace of second derivatives by some computable conforming companion operator $J:V_h\to V$  from the nonconforming virtual element space $V_h$. The operator $J$ is a right-inverse of the interpolation operator  and  leads to  optimal error estimates in  piecewise Sobolev norms without any additional regularity assumptions on $u\in V$.  As a smoother the companion operator modifies the discrete right-hand side  and then allows a quasi-best approximation. An explicit residual-based \textit{a posteriori} error estimator is reliable and efficient up to data oscillations. Numerical examples   display the predicted empirical convergence rates  for uniform and optimal convergence rates for adaptive mesh-refinement.
\end{abstract}
	 \par  \textit{a priori}, \textit{a posteriori}, adaptive  mesh-refinement, companion  operator, smoother
	\\
	\numberwithin{equation}{section}
	\numberwithin{figure}{section}
	\section{Introduction}
	The popular nonconforming Morley finite element method (FEM) for fourth-order problems allows a generalization from triangular domains to polygons  in the class of nonconforming virtual element methods (ncVEM).  Two  ncVEM have been introduced in \cite{antonietti2018fully,chen2020nonconforming,zhao2018morley} for $H^3$ regular  solutions, while a medius analysis  in \cite{huang2021medius}  allows minimal regularity.   %
	In comparison to the existing literature on ncVEM for biharmonic problems, this paper presents an abstract framework and identifies two hypotheses \ref{H1}-\ref{H2} for a unified stability and \textit{a priori} error analysis of at least  two different ncVEM  each with an individual  parameter $r=-1,0,1,2$ ($r=-1$ for original VE spaces and $r=0,1,2$ for enhanced VE spaces \cite{ahmad2013equivalent}) and even a mixture of those.    This paper adds a new analysis with a computable conforming companion that allows a quasi-best approximation
	\begin{align}|u-Gu_h|_{2,\pw}\lesssim \min_{v_h\in V_h}|u-Gv_h|_{2,\pw} \label{quasi}
	\end{align} with the local Galerkin projection $G$ onto piecewise quadratics and a general source function with a smoother for the first time in ncVEM and completes the \textit{a priori} error convergence analysis  in piecewise energy and weaker Sobolev norms.   The lower-order estimates are available in the  literature for enhanced VE spaces, e.g., Zhao \textit{et al.} discuss piecewise $H^1$ error estimate in \cite{zhao2018morley} for an enhanced VE space ($r=0$) and this paper proves it also for original VE spaces ($r=-1$).
	The design of companion operators started in \cite{carstensen2015adaptive} for second-order  and  in \cite{carstensen2014guaranteed,carstensen2018prove,veeser2019quasi} for fourth-order problems. It is related to enrichments in  multigrid methods \cite{brenner1993nonconforming} and to reliable  \textit{a posteriori} error control \cite{carstensen2013computational}. Its role as a smoother in ncVEM generalizes \cite{carstensen2021,carstensen2021lower,veeser2019quasi} for the Morley FEM.
	 The first paper \cite{carstensen2021priori}  on an \textit{a posteriori} error analysis for  ncVEM  is restricted to second-order  problems and includes many references on an \textit{a posteriori} error analysis for the conforming VEM. This is the first paper on an \textit{a posteriori} virtual element error control for fourth-order problems  with reliable and efficient error estimators and a suggested adaptive mesh-refining algorithm. The presented \textit{a posteriori} error analysis also covers conforming VEM \cite{brezzi2013virtual}.

	\bigskip
	\noindent\textit{Main results}. This paper contributes to the understanding of the ncVEM for a class of examples that includes the two known  examples of discrete VE spaces for  fourth-order problems
	\begin{itemize}
		\item  a computable conforming companion operator, 
		\item   \textit{a priori} error estimates in piecewise $H^1$ and $H^2$ norms,
		\item quasi-best approximation for a smoother for any source term $F\in H^{-2}(\Omega)$,
	    \item reliable and (up to data oscillations)  efficient  \textit{a posteriori} error control, 
	    \item adaptive  mesh-refinement algorithm with improved empirical convergence rates. 
	\end{itemize}
	The results are displayed for 2D and the lowest-order case only  corresponding to the Morley degrees of freedom, but the arguments allow a higher dimension and higher degrees.
	
	\medskip
	\noindent\textit{Outline and organisation of the paper}. Section 2 describes the admissible partitions of the domain $\Omega$ into polygonal domains  and defines the local and global Morley degrees of freedom. Subsection 2.2 establishes a Poincar\'e-Friedrichs inequality on polygons and Subsection~2.4 recall the local Galerkin projection $G$ and establish associated error estimates. Section~3 explains an abstract framework with hypotheses \ref{H1}-\ref{H2} and  presents two affirmative examples of virtual elements $V_h(P)$. The interpolation operator is defined in Subsection~3.3 and  its error estimates follow in Subsection~3.4. Section~4 designs a computable conforming companion operator, which is a right-inverse of the interpolation operator, and provides the fundamental approximation error estimates. Subsection~5.1 introduces the natural stabilization and Subsection~5.2  the discrete problem with two choices  for the right-hand side.  Subsection~5.3 provides  the \textit{a priori} error estimates in piecewise $H^1$ and $H^2$ norms  in the best-approximation form up to data oscillations; a smoother in the right-hand side eliminates the  oscillations.  Section~6 developes an explicit residual-based reliable and (up to data oscillations) efficient  error analysis for ncVEM that also applies for the conforming VEM. 	 The stabilization term is efficient with respect to the sum of the error $u-\pid u_h$ and  $u-u_h$ in their piecewise $H^2$ seminorms. Subsection~7.1 suggests an adaptive mesh-refinement algorithm. Numerical results support the theoretical predictions in Subsections~7.2-7.3 and provide striking numerical evidence of  optimal empirical convergence rates for adaptive mesh-refining. Supplement material accompanies this paper with details on the local virtual element spaces  and is referred to as Appendix A, B, and C throughout this paper. This contains partly established or routine results that are somehow standard but seemingly not available in the literature in this form.
	
	\medskip
	\noindent\textit{Notation}. Standard notation on Lebesgue and Sobolev spaces  and norms applies throughout this paper, e.g.,  $\|
	\cdot\|_{s,\cal{D}}$ (resp. seminorm $|\cdot|_{s,\cal{D}}$) for $s\geq0$ denotes norm on the Sobolev space  $H^s(\mathcal{D}):=H^s(\text{int}(\cal{D}))$  of order $s\in\mathbb{R}$ defined in the interior $\text{int}(\mathcal{D})$ of a  domain $\mathcal{D}$, while $(\cdot,\cdot)_{L^2({\cal D})}$ and $\|\cdot\|_{L^2({\cal D})}$  denote the $L^2$ scalar product and $L^2$ norm in  ${\cal D}$.  Let $|\mathcal{D}|$ denote the area of a domain $\mathcal{D}$, and $\dashint_{\mathcal{D}}\bullet\,dx:=|\mathcal{D}|^{-1}\int_{\mathcal{D}}\bullet\,dx$  denote the integral mean  on $\mathcal{D}$.  Define the Sobolev space $V=H^2_0(\Omega):=\{v\in H^2(\Omega):v=v_\bn=0\}$ for the  derivative $v_\bn=\nabla v\cdot\bn$ in the direction of outward unit normal $\bn$ along the boundary $\partial\mathcal{D}$. The vector space $C^r(\mathcal{D})$ is the set of $C^r$-continuous functions  defined on a domain $\mathcal{D}$ for $r\in\mathbb{N}_0$.  Let $\p_k({\cal D})$ denote  the set of polynomials of degree at most $k\in\mathbb{N}_0$ defined on a domain ${\cal D }$ and  $\p_k({\T})$ denote the set of piecewise polynomials on an admissible partition $\T\in\mathbb{M}$ (defined in Subsection~2.1). The piecewise seminorm and norm  in $H^s(\T)$ for $s\in\mathbb{R}$ (see the definition of $|\cdot|_{s,P}=|\cdot|_{H^s(P)}$ in, e.g.,  \cite[Chapter~14]{7}) read
	$|\cdot|_{s,\text{pw}}:=\big(\sum_{P\in\T}| \cdot|_{s,P}^2\big)^{1/2}$ and $\|\cdot\|_{s,\text{pw}}:=\big(\sum_{P\in\T}\| \cdot\|_{s,P}^2\big)^{1/2}$. Let $\Pi_k$  denote the $L^2$ projection on $\p_k(\T)$ for $k\in\mathbb{N}_0$. The oscillation of $f\in L^2(\Omega)$   reads  
	\begin{align*}
	\mathrm{osc}_2(f,{\T}):=\Big(\sum_{P\in{\T}}\mathrm{osc}_2^2(f,P)\Big)^{1/2}\quad\text{for}\quad\mathrm{osc}^2_2(f,P):= \|h_P^2(1-\Pi_2)f\|_{L^2(P)}^2.
	\end{align*} 
\par	Let $\mathbb{S}$ be the set of $2\times 2$ symmetric matrices in $\mathbb{R}^{2\times 2}$ and let $\delta_{jk}$ denote the Kronecker delta ($\delta_{jk}=0$ if $j\neq k$ and $\delta_{jj}=1$). Let $\alpha:=(\alpha_1,\alpha_2)$ denote a multi-index with $\alpha_j\in \mathbb{N}_0$ for $j=1,2$ and $|\alpha|:=\alpha_1+\alpha_2.$
	The  outward normal and tangential   derivatives of first and higher orders  are written as subscripts $\bn,\bt, {\bn\bn} , {\bt\bt}, {\bn\bt\bt}$ etc. for the exterior  unit normal vector $\bn$ and the tangential vector $\bt$ along the boundary $\partial P$ of the (polygonal Lipschitz) domain $P\in\T\in\mathbb{M}$ (from Subsection~ 2.1).  An inequality $A\lesssim B$ abbreviates $A\leq CB$ for   a generic constant $C$,  that exclusively depends on the domain $\Omega$ and on the mesh-parameter $\rho$ (from (M2) below). 
	\section{Virtual element method}
	\subsection{Admissible partitions}
	Let $\mathbb{M}$ be a family of decompositions of $\overline{\Omega}$ into polygonal domains satisfying the two mesh conditions (M1)-(M2) with a universal positive constant $\rho$.

\medskip
		\noindent(M1) \textit{Admissibility}.  Any two distinct polygonal domains $P$ and $P'$ in $\T\in\mathbb{M}$ are disjoint or share   a  finite number of edges  and vertices.
		
		\medskip
		\noindent(M2) \textit{Mesh regularity}.   Every polygonal domain $P$ of diameter $h_P$ is star-shaped with respect to every point of a ball of radius greater than equal to $\rho h_P$ and every edge $E$ of $P$ has a length $h_E$ greater than equal to $\rho h_P$.\\
\par Here and throughout this paper,  $h_\T|_P:=h_P:=\text{diam}(P)$ denotes the piecewise constant mesh-size $h_\T\in\p_0(\T)$  and $\ms:=\max_{P\in\T}h_P$ denotes the maximum diameter over all $P\in\T\in\mathbb{M}$. Let $\mathcal{V}(P)$ (resp. $\mathcal{V}$) denote the set of vertices  of $P$ (resp. of $\T$) and let $\e(P)$ (resp. $\e$) denote the set of edges  of $P$ (resp. of $\T$). Denote the interior and boundary edges of $\T$ by $\e(\Omega)$ and $\e(\partial\Omega)$.  Let $|\mathcal{V}|$ (resp. $|\e|$) denote the number of vertices (resp. edges) and $N:=|\mathcal{V}|+|\e|$. 

\medskip
\noindent The \textit{standard notation  of the polygonal domain $P$} with $N_P$ edges and $N_P$ vertices is depicted in Figure~2.1.a. Note that $3\leq N_P\leq M(\rho)$ for a global number $M(\rho)$ that exclusively depends on $\rho$ \cite{da2014mimetic}. We enumerate the vertices $\mathcal{V}(P):=\{z_1,\dots,z_{N_P}\}$ and edges $\e(P):=\{E(1),\dots,E(N_P)\}$ consecutively, i.e., $E(j)=\text{conv}\{z_j,z_{j+1}\}$ for $j=1,\dots,N_P$ with $z_{N_P+1}:=z_1$ and enumerate $z_1,\dots,z_{N_P}$ counterclockwise along the boundary $\partial P$.  (M2) implies that  each polygonal domain $P\in\T$  can be divided into triangles $T(j):=\text{conv}\{z_0,z_j,z_{j+1}\}$  for all $j=1,\dots,N_P$ and for the midpoint $z_0$ of the ball from (M2) in Figure~2.1.b. It is known \cite{14} that the resulting sub-triangulation $\tT|_P:=\tT(P):=\cup_{j=1}^{N_P}T(j)$ of $P\in\T$ is uniformly shape-regular; i.e, the minimum angle in each triangle $T\in\tT(P), P\in\T\in\mathbb{M}$, is bounded below by some positive constant $w_0>0$ that exclusively depends on $\rho$. Let  $\widehat{\mathcal{V}}$ (resp. $\widehat{\mathcal{V}}(P)$) denote the set of vertices and   $\eE$ (resp. $\eE(P)$) denote the set of edges  in  $\tT$ (resp. $\tT(P)$). 

	\begin{figure}[H]
	\centering
	\begin{subfigure}{.33\textwidth}
		\centering
		\includegraphics[width=1\linewidth]{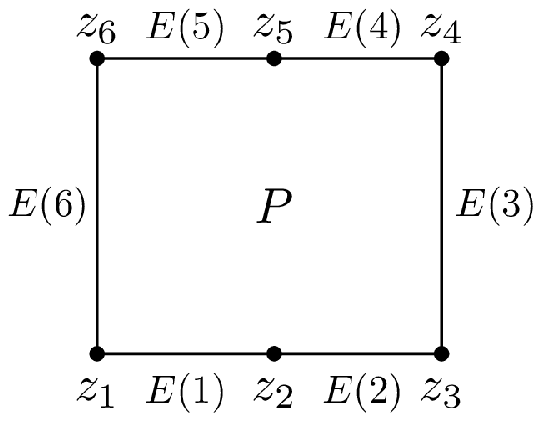}               
		\caption{}
		\vspace{0.5cm}
		\label{fig1}
	\end{subfigure}%
	\begin{subfigure}{.33\textwidth}
		\centering
		\includegraphics[width=1\linewidth]{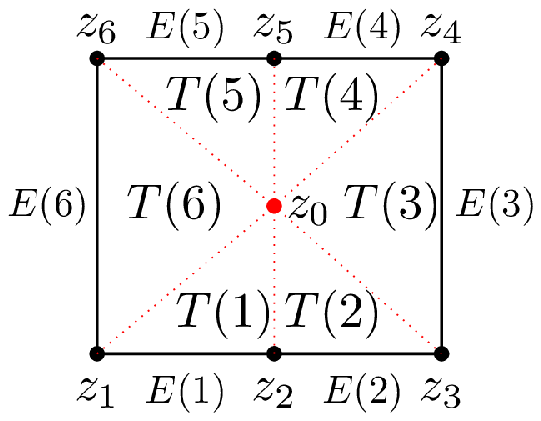}
		\caption{}
		\vspace{0.5cm}
		\label{fig2}
	\end{subfigure}
\begin{subfigure}{.33\textwidth}
	\centering
	\includegraphics[width=1\linewidth]{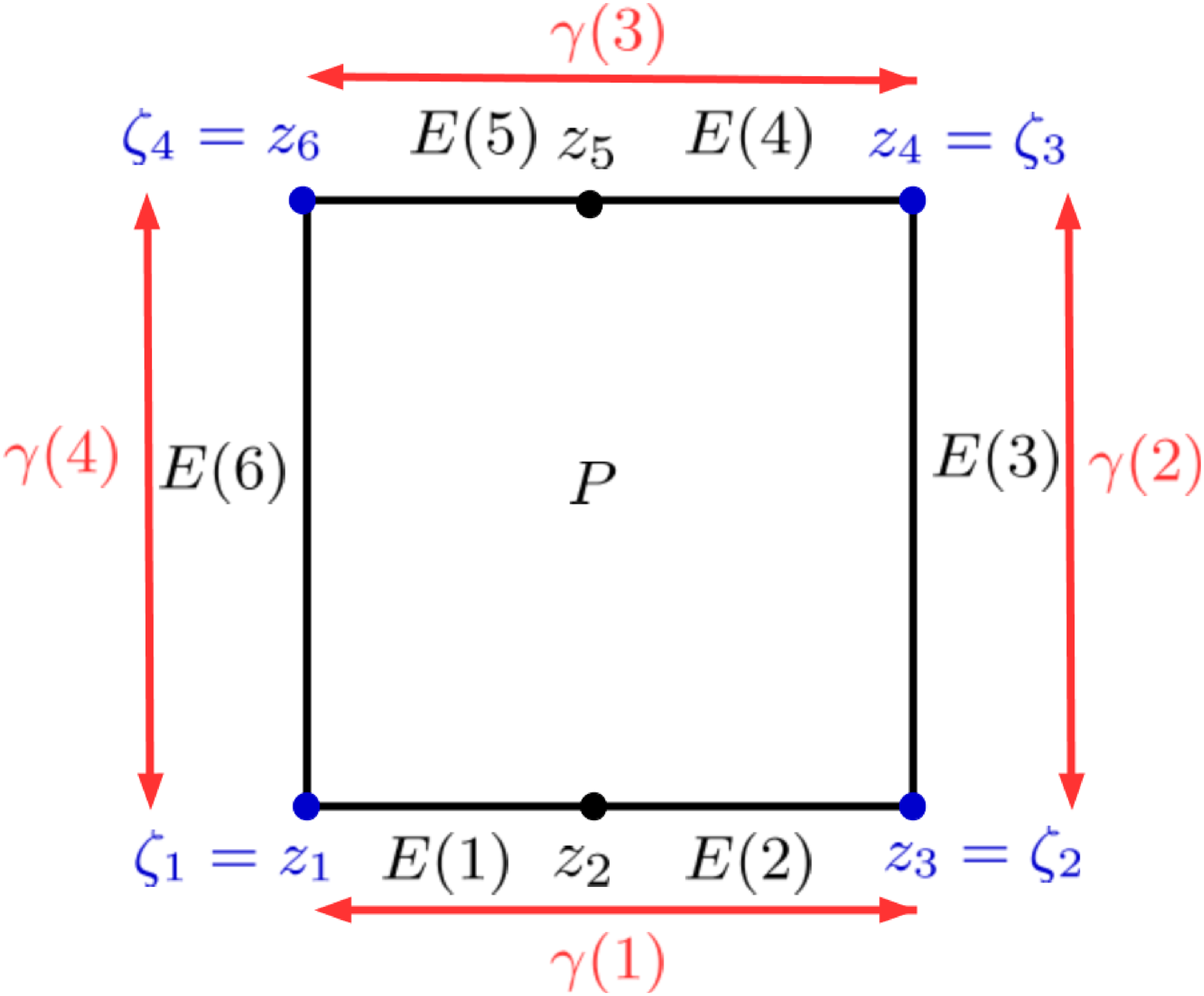}
	\caption{}
	\vspace{0.5cm}
	\label{fig3}
\end{subfigure}
	\caption{(a) Hexagon $P$ with vertices $z_1,\dots,z_6$ and edges $E(1), \dots, E(6)$  (b) its sub-triangulation $\tT(P)$ and (c)  $P$ with corner points $\zeta_1,\dots,\zeta_4$ and sides $\gamma(1),\dots,\gamma(4)$.}
	\label{fig}
\end{figure}
With a counterclockwise orientation along the polygonal boundary $\partial P$, assign the piecewise constant tangential unit vector $\bt_P$ and the outer normal unit vector $\bn_P$.  Define the local degrees of freedom (dofs) $\text{dof}_1(v),\dots,\text{dof}_{2N_P}(v)$,  for $v\in  H^2(P)$, by
\begin{align}
\text{dof}_j(v) = \begin{cases} \text{$v(z_j)\hspace{1.85cm}$ for $j=1,\dots,N_P,$}\\
\text{$\int_{E(k)} v_\bn\,ds\hspace{1cm}$  for  $j=N_P+1,\dots,2N_P$ and $k=j-N_P$}
\end{cases}\label{dof}
\end{align}
with the vertices $z_j$  and the edges $E(k)\in\e(P)$ of the polygonal domain $P\in\T$ in Figure~2.1.a. Given a polygonal domain $P\in\T$, the dofs in \eqref{dof} are collected in the linear map $\text{Dof}:=(\text{dof}_1,\dots,\text{dof}_{2N_P}):H^2(P)\to\mathbb{R}^{2N_P}$.

\medskip
  Fix the orientation of a unit normal vector $\bn_E$  to  each edge $E\in\e$. The sign of the jump $[\bullet]_E$ across an edge $E\in \e$  follows from the subsequent convention of the orientation of a unit normal vector $\bn_E$ along an edge $E$: Label the two neighbouring polygons $P_{\pm}$ sharing the interior edge $E=\partial P_{+}\cap\partial P_{-}$ such that  $\bn_{P_+}|_E=\bn_E$  and $\bn_{P_-}|_E=-\bn_E$. This defines the sign in the jump $[\bullet]_E:=\bullet|_{P_{+}}-\bullet|_{P_{-}}$ across $E\in\e(\Omega)$.  For  a boundary edge $E\in\e(\partial\Omega)$, set  $\bn_{\Omega}|_E= \bn_E$ and $[\bullet]_E:=\bullet|_E$. For $
 v_\nc\in V_\nc:=\{v\in H^2(\T):v\; \text{is continuous at interior vertices  and zero at boundary}$ $\text{vertices of $\T$, and} \int_E[v_{\bn}]_E\,ds=0\quad \text{for all}\;E\in\e\}$, the global Morley dofs read
\begin{align}
\text{dof}_j(v_\nc) = \begin{cases} \text{$v_\nc(z_j)\hspace{1.85cm}$ for $j=1,\dots,|\mathcal{V}|,$}\\
\text{$\int_{E(k)} (v_\nc)_\bn\,ds\hspace{0.75cm}$  for  $j=|\mathcal{V}|+1,\dots,N$ and $k=j-|\mathcal{V}|.$}
\end{cases}\label{gdof}
\end{align}
The global degrees of freedom $\text{dof}_j(v_\nc)$ from \eqref{gdof} coincide with the local degrees of freedom $\text{Dof}(v_\nc|_P)$ from \eqref{dof} for each polygonal domain $P\in\T$ up to a (known) change of signs of $\bn_E$ and $\bn_P|_E$ for an interior edge $E\in\e(\Omega)$.
\begin{lemma}\label{eq}
	$|\cdot|_{2,\pw}$  defines a norm  on $V_\nc$ equivalent to $\|\cdot\|_{2,\pw}$. 
\end{lemma}
\begin{proof}
	This is shown in \cite[Lemma~3.1]{antonietti2018fully} and in \cite[Lemma~5.1]{zhao2018morley}   based on the Poincar\'e-Friedrichs inequality for piecewise $H^2$ functions in \cite{brennerh2}.
\end{proof}
\subsection{Poincar\'e-Friedrichs inequality}
This subsection provides a Poincar\'e-Friedrichs inequality for a polygonal domain $P\in\T\in\mathbb{M}$ with  explicit constants that exclusively depend on $\rho$ from Subsection~2.1.
\begin{theorem}[Poincar\'e-Friedrichs inequality]\label{PF} 
	There exists a positive constant $C_\PF$ (that  exclusively depends on $\rho$)  such that\\
	$(a)$ $
	\|f\|_{L^2(P)}\leq C_\PF h_P|f|_{1,P}$
	holds for any $f\in H^1(P)$ with $0\in\mathrm{conv} \Big\{\int_{E(1)}f\,ds,\dots,\int_{E(N_P)} f\,ds\Big\},$\\
	$(b)$
	$
	\sum_{m=0}^1h_P^{m-2}|f|_{m,P}\leq C_\PF |f|_{2,P}
	$
	holds for any $f\in H^2(P)$ with $0\in \mathrm{conv}\{f(z_1),\dots,f(z_{N_P})\}$ and $ 0\in\mathrm{conv}\left\{\int_{E(1)} \frac{\partial f}{\partial x_j}\,ds,\dots,\int_{E(N_P)} \frac{\partial f}{\partial x_j}\,ds\right\}$ for $j=1$ and $j=2$.
\end{theorem}
\begin{proof}
	The proof of $(a)$ is included in  \cite[Lemma~2.1]{carstensen2021priori}. The proof of $(b)$ considers the sub-triangulation $\tT(P)$ of the polygonal domain $P$ from Figure~2.1.b. Define the linear interpolation $I_1f\in S_1(\tT(P)):=\p_1(\tT(P))\cap C^0(P)$ with $(f-I_1f)(z_k)=0$ for all $k=0,\dots,N_P$. The triangle inequality shows
	\begin{align}
	\|f\|_{L^2(P)}\leq \|f-I_1f\|_{L^2(P)}+\|I_1f\|_{L^2(P)}. \label{2.1}
	\end{align}
	The Bramble Hilbert lemma \cite{14, ciarlet2002finite}   leads to  a positive constant $C_{\text{BH}}$ (that exclusively depends on the shape of the triangles $T(1),\dots,T(N_P)$ and so merely on $\rho$) in the error estimate
	\begin{align}
	\|f-I_1f\|_{L^2(P)}\leq C_{\text{BH}}h_P^2|f|_{2,P}.\label{2.2a}
	\end{align}
	(Explicit formulas for $C_{\text{BH}}$ in terms of the maximal angle in a triangle can be found in \cite{carstensen2012explicit}.) Let $\varphi_0,\dots,\varphi_{N_P}$ be the nodal basis functions of $S_1(\tT(P))$ with $\varphi_k(z_{\ell})=\delta_{k\ell}$ for $k,\ell=0,\dots,N_P$; whence $I_1f=\sum_{k=0}^{N_P}f(z_k)\varphi_{k}$. The  local mass matrix with entries $\int_T\lambda_j\lambda_k\,dx=(1+\delta_{jk})|T|/12$ for $j,k=1,2,3$ and the barycentric coordinates $\lambda_1,\lambda_2,\lambda_3$ in a triangle $T\in\tT(P)$ has the eigenvalues $|T|/12$ (twice) and $|T|/3$.   Rayleigh quotients with the local mass matrix reveal \begin{align}\|I_1f\|^2_{L^2(P)}&=\sum_{T\in\tT(P)}\|\sum_{z\in\mathcal{V}(T)}f(z)\varphi_z\|^2_{L^2(T)}\leq \sum_{T\in\tT(P)}\frac{|T|}{3}\sum_{z\in\mathcal{V}(T)}f(z)^2\nonumber\\&\leq\frac{1}{3}\sum_{j=0}^{N_P}f(z_j)^2\sum_{\substack{T\in\tT(P)\\z_j\in\mathcal{V}(T)}}|T|\leq \frac{|P|}{3}  \sum_{k=0}^{N_P}f(z_k)^2. \label{2.3a}\end{align} 
	Since $0\in\text{conv}\{f(z_1),\dots,f(z_{N_P})\}$, there exists convex coefficients $0\leq \mu_1,\dots,\mu_{N_P}\leq 1$ with $\sum_{\ell=1}^{N_P}\mu_{\ell}=1$ and $\sum_{\ell=1}^{N_P}\mu_{\ell} f(z_{\ell})=0$. This implies $\min_{j=1}^{N_P}f(z_j)\leq 0\leq \max_{j=1}^{N_P}f(z_j)$ and so $x=\{f(z_1),\dots,f(z_{N_P})\}$ in \cite[Lemma~4.2]{8} guarantees
	\begin{align}
	\sum_{k=0}^{N_P}f(z_k)^2\leq  f(z_0)^2+M\sum_{j=1}^{N_P}(f(z_j))-f(z_{j+1}))^2\label{2.6}
	\end{align}
	for the constant $M := (2(1-\cos(\pi/N_P)))^{-1}$ (that exclusively depends on $M(\rho)\geq N_P$ and so on $\rho$). The underlying inequality $\sum_{j=1}^{N_P}x_j^2\leq M \sum_{j=1}^{N_P}(x_{j+1}-x_j)^2$ with $x_{N_P+1}=x_1$ follows for $x_\ell:=\min\{x_1,\dots,x_{N_P}\}\leq 0\leq\max\{x_1,\dots,x_{N_P}\}=:x_m$ for some indices $\ell,m\in\{1,\dots,N_P\}$ immediately from $\max\{|x_1|,\dots,|x_{N_P}|\}\leq |x_\ell|+|x_m|=|x_\ell-x_m|\leq \sum_{j=1}^{N_P}|x_{j+1}-x_j|\leq N_P^{1/2}(\sum_{j=1}^{N_P}|x_{j+1}-x_j|^2)^{1/2}$ with  triangle and Cauchy-Schwarz inequalities. The optimal constants in \cite{8} require little matrix analysis.  The above coefficients $\mu_\ell$ and a Cauchy-Schwarz inequality show that \begin{align}f(z_0)^2=\left(\sum_{\ell=1}^{N_P}\mu_{\ell}(f(z_0)-f(z_\ell))\right)^2\leq N_P\sum_{\ell=1}^{N_P}(f(z_0)-f(z_{\ell}))^2.\label{2.7}\end{align} The combination \eqref{2.6}-\eqref{2.7} results in 
	\begin{align}
	\sum_{k=0}^{N_P}f(z_k)^2\leq N_P\sum_{\ell=1}^{N_P}(f(z_0))-f(z_{\ell}))^2+M\sum_{j=1}^{N_P}(f(z_j))-f(z_{j+1}))^2.\label{2.3}
	\end{align}
	For any edge $E=\text{conv}\{a,b\}\in\e(\tT(P))$ with vertices $a,b\in\mathcal{V}(E)\subset\mathcal{V}(\tT(P))$ and an aligned triangle $ T(E)=\text{conv}\{z_0,E\}\supset E$, the tangential derivative shows $|f(a)-f(b)|=\bigg|\int_E f_{\bt}\,ds\bigg|$. The trace identity  $\dashint_{E} f\,ds =\dashint_{T(E)}f\,dx+\frac{1}{2}\dashint_{T(E)}(x-z_0)\cdot\nabla f(x)\,dx$ \cite[Lemma~2.6]{8}  implies that
	\begin{align*}
	2h_E^{-1}|T(E)|\;|f(a)-f(b)|\leq 2\bigg|\int_{T(E)}\nabla f(x)\cdot\bt_E\,dx\bigg|+\bigg|\int_{T(E)}(x-z_0)\cdot D^2f(x)\bt_E\;dx\bigg|.
	\end{align*}
	A Cauchy-Schwarz inequality provides
	\begin{align}
	|f(a)-f(b)|\leq h_E|T(E)|^{-1/2}(|f|_{1,T(E)}+h_{T(E)}|f|_{2,T(E)}).\label{2.9}
	\end{align}
	Recall $\omega_0$ from Subsection~2.1 and  note that $\sin(\omega_0)h_{T(E)}^2\leq 4|T(E)|$. This and $h_E\leq h_{T(E)}\leq h_P$ for all $E\in\e(P)$ lead  in \eqref{2.9} to
	\begin{align*}
	|f(a)-f(b)|^2\leq \frac{8}{\sin(\omega_0)}(|f|^2_{1,T(E)}+h_{P}^2|f|^2_{2,T(E)}).
	\end{align*}
	This applies to the triangles $T(E)=T(j)\in\tT(P)$ with the edges $\text{conv}\{a,b\}= \text{conv}\{z_j,z_{j+1}\}\;$\\ and $ \text{conv}\{z_0,z_j\}$  for $j=1,\dots, N_P$ from Figure~2.1.b.  Hence  \eqref{2.3} shows that
	\begin{align*}
	(N_P+M)^{-1}\frac{\sin{\omega_0}}{8}\sum_{k=0}^{N_P}f(z_k)^2\leq \sum_{E\in\e(P)}(|f|^2_{1,T(E)}+h_P^2|f|^2_{2,T(E)})=|f|^2_{1,P}+h_P^2|f|^2_{2,P}.
	\end{align*} 
	This and \eqref{2.3a} result in $
	\|I_1f\|_{L^2(P)}\leq C(\rho) (h_P|f|_{1,P}+h_P^2|f|_{2,P})$
	with  $|P|\leq \pi h_P^2$ from \cite[Lemma~1.12]{di2019hybrid},  $N_P\leq M(\rho)$ from Subsection~2.1, and  $C(\rho)^2:=\frac{8\pi(M(\rho)+M)}{3\sin(\omega_0)}$. The part $(a)$ of the lemma applies to $\partial f/\partial x_j$ for $j=1$ and for $j=2$, and so controls  the term  $|f|_{1,P}\leq C_\PF h_P|f|_{2,P}$. Consequently,
	$\|I_1f\|_{L^2(P)}\leq C(\rho)(1+C_\PF) h_P^2|f|_{2,P}.$ This and \eqref{2.1}-\eqref{2.2a} show
	$
	\|f\|_{L^2(P)}\leq (C_{\text{BH}}+C(\rho)(1+C_\PF)) h_P^2|f|_{2,P}
	$
	and conclude the proof of $(b)$ with a  re-labelled  constant $C_\PF$.
\end{proof}
\noindent Recall $\text{Dof}:H^2(P)\to\mathbb{R}^{2N_P}$ for a polygonal domain $P\in\T\in\mathbb{M}$ from \eqref{dof}.
\begin{lemma}\label{lem2.3}
	Any $v\in H^2(P)$ with $\mathrm{Dof}(v)=0$ satisfies
	$h_P^{-2}\|v\|_{L^2(P)}+h_P^{-1}|v|_{1,P}\leq C_\PF |v|_{2,P}.$
\end{lemma}
\begin{proof}
	If $\bn:=(\bn_x,\bn_y)$ is an outward unit  normal  to an edge $E$, then the unit tangential vector along $E$ is $\bt:=(\bt_x,\bt_y)=(-\bn_y,\bn_x)$. This leads to the split $\nabla v= v_\bn \bn+v_{\bt} \bt$. An integration along $E(k)$ of the tangential derivative   implies for all $k=1,\dots,N_P$ that
	\begin{align}
	\int_{E(k)}\nabla v\,ds=\Big(\int_{E(k)}v_\bn\,ds\Big) \bn_{E(k)}+(v(z_{k+1})-v(z_k))\bt_{E(k)}.\label{2.20s}
	\end{align}
	This and the re-summation $\sum_{k=1}^{N_P}(v(z_{k+1})-v(z_k))\bt_{E(k)}=\sum_{k=1}^{N_P}(\bt_{E(k-1)}-\bt_{E(k)})v(z_k)$ lead  to
	\begin{align}
	\int_{E(k)}\nabla v\,ds=\text{dof}_{N_P+k}(v)\bn_{E(k)}+(\bt_{E(k-1)}-\bt_{E(k)})\text{dof}_k(v).\label{s}
	\end{align}If $\text{Dof}(v)=0$, then $\int_{E(k)} \nabla v\,ds=0$ for all $k=1,\dots,N_P$ from \eqref{s}. This and $v(z_j)=\text{dof}_j(v)=0$ for all $j=1,\dots,N_P$ allow the Poincar\'e-Friedrichs inequality in Theorem~\ref{PF}.b. 
\end{proof}

\subsection{Biharmonic model problem}
Define the scalar product $a(\cdot,\cdot):V\times V\to\mathbb{R}$ for $u,v\in V=H^2_0(\Omega)$  by
\begin{align}
a(u,v):= \int_\Omega D^2u:D^2v\,dx\quad\text{with}\;D^2u:=\begin{pmatrix}
u_{11} &u_{12}\\u_{12}& u_{22}
\end{pmatrix}\label{1.2}
\end{align}
with $D^2u:D^2v := u_{11}v_{11}+2 u_{12} v_{12}+u_{22}v_{22}$. The subscripts $\alpha,\beta=1,2$ abbreviate the second-order partial derivatives $u_{\alpha\beta}:=\frac{\partial^2 u}{\partial x_\alpha\partial x_\beta}$. The bilinear form $a_\pw$ and the differential operator $D^2_\pw$ denote the corresponding piecewise versions (with respect to a partition $\T$ or $\tT$ suppressed in the notation). The local contribution $a^P(\cdot,\cdot)$ is the semi-scalar product \[a^P(u,v):=\int_P D^2u:D^2v\,dx\quad\text{for}\; u,v \in H^2(P).\] 
\noindent The scalar product $a(\cdot,\cdot)$   induces the energy norm  $|v|_{2,\Omega}:= a(v,v)^{1/2}$  equivalent to the Sobolev norm $\|\cdot\|_{2,\Omega}$ owing to the Friedrichs inequality \cite[Sec.~10.6]{7} and $(V,a(\cdot,\cdot))$ is a Hilbert space. Given any $F\in V^*=H^{-2}(\Omega)$, the Riesz representation is the weak solution $u\in V$ to
\begin{align}
a(u,v)=F(v)\quad\text{for all}\; v\in V.\label{3}
\end{align}

\medskip 
\noindent\textit{Elliptic regularity}. For the fixed polygonal bounded Lipschitz domain $\Omega$, there exist positive constants $\sigma_{\text{reg}}>1/2$ and $C_{\text{reg}}$ \cite{bacuta2002shift,blum1980boundary,carstensen2021lower} such that $F\in H^{-s}(\Omega)$  and  $2-\sigma\leq s \leq 2$ for $\sigma:=\min\{\sigma_{\text{reg}},1\}$ imply $u\in V\cap H^{4-s}$ and
\begin{align}
\|u\|_{4-s,\Omega}\leq C_{\text{reg}}\|F\|_{-s,\Omega}.\label{2.14}
\end{align}

\subsection{Galerkin projection}
The $H^2$ elliptic projection operator $G:H^2(P)\to\p_2(P)$ is defined, for any $v\in H^2(P)$, by $G v\in\p_2(P)$ and
\begin{align}
a^P(G v, \chi) = a^P(v,\chi)\quad\text{for all}\;\chi\in\p_2(P)\label{pid1}
\end{align}
with the  additional conditions (i.e., three equations to fix the affine contribution)
\begin{align}
\frac{1}{N_P}\sum_{j=1}^{N_P}G v(z_j)=\frac{1}{N_P}\sum_{j=1}^{N_P}v(z_j)\quad
\text{and}\quad\int_{\partial P}\nabla\pid v\,ds =\int_{\partial P}\nabla v\,ds.\label{pid2}
\end{align}
Equation \eqref{pid1} determines $\pid v\in\p_2(P)$ up to affine functions and the additional three equations in \eqref{pid2} define  $\pid v\in\p_2(P)$  uniquely for  $v\in H^2(P)$.  The linear operator $\pid:H^2(P)\to \p_2(P)$ is a projection onto $\p_2(P)$. An integration by parts and \eqref{pid2} imply, for all $v\in H^2(P)$, that \begin{align}\Pi_0D^2v=\frac{1}{|P|}\int_{\partial P}\nabla v\,ds=\frac{1}{|P|}\int_{\partial P}\nabla \pid v\,ds=\Pi_0D^2\pid v=D^2\pid v.\label{2.5a}
\end{align}
\begin{lemma}[approximation error of $G$]\label{lem2.1}
	Any $v\in H^2(P)$  with $Gv\in\p_2(P)$ from \eqref{pid1}-\eqref{pid2} satisfies  
	$C_\PF^{-1}\sum_{m=0}^1h_P^{m-2}|v-Gv|_{m,P}\leq |v-Gv|_{2,P}\leq  |v|_{2,P}$ and there exists a positive constant $C_{\mathrm{apx}}$ (that exclusively depends on $\rho$) such that  $|v-Gv|_{2,P}\leq C_{\mathrm{apx}}h_P^{s}|v|_{2+s,P}$
	for  $v\in H^{2+s}(P)$ and $0<s\leq 1$.
\end{lemma}
\begin{proof}
	The condition \eqref{pid2} implies that $\sum_{j=1}^{N_P}(v-Gv)(z_j)=0$ and $\int_{\partial P}\nabla(v-Gv)\,ds=0$. Hence the Poincar\'e-Friedrichs inequality in Theorem~\ref{PF}.b proves that $C_\PF^{-1}\sum_{m=0}^1h_P^{m-2}|v-Gv|_{m,P}\leq |v-Gv|_{2,P}$. The Pythagoras identity $|v-Gv|^2_{2,P}+|Gv|^2_{2,P}=|v|^2_{2,P}$ from \eqref{pid1} provides $|v-Gv|_{2,P}\leq |v|_{2,P}$. 	The definition of $G$ in \eqref{pid1} shows that $|v-Gv|_{2,P}\leq |v-\chi|_{2,P}$ for any $\chi\in\p_2(P)$. The Bramble-Hilbert lemma \cite[Thm.~6.1]{10} concludes the proof.
\end{proof}
\noindent Recall $\text{Dof}:H^2(P)\to\mathbb{R}^{2N_P}$ for a polygonal domain $P\in\T\in\mathbb{M}$ from \eqref{dof}.
\begin{lemma}[boundedness of Dof]\label{lem}	There exists a positive constant $C_d$ (that exclusively depends on $\rho$) such that any $v\in H^2(P)$ satisfies
	$|\mathrm{Dof}(v-Gv)|_{\ell^2}\leq C_dh_P|v-Gv|_{2,P}.$
\end{lemma}
\begin{proof}
	The scaled Sobolev inequality from \cite[Sec.~2.1.3]{14}  leads for $w:=v-Gv$  to
	\begin{align}
	|w(z_0)|\leq \|w\|_{L^\infty(P)}\leq C_S \sum_{m=0}^2h_P^{m-1}|w|_{m,P}\label{2.16}
	\end{align}
	with a positive constant $C_S$ (that exclusively depends on $\rho$). A Cauchy-Schwarz inequality and the trace inequality   $\|w\|^2_{L^2(E)}\leq C_T(h_E^{-1}\|w\|^2_{L^2(P)}+h_E\|\nabla w\|^2_{L^2(P)})$ (e.g., from \cite[p.~554]{14})  for any edge $E\in\e(P)$  result in
	\begin{align}
	\Big|\int_{E}w_\bn\,ds\Big|\leq h_{E}^{1/2}\|w_\bn\|_{L^2(\partial P)}\leq C_T(|w|_{1,P}+h_E|w|_{2,P}).\label{2.17}
	\end{align}
	The combination of     \eqref{2.16}-\eqref{2.17} and Lemma~\ref{lem2.1}   imply that
	\begin{align}
	|w(z_0)|\leq C_S(1+C_\PF)h_P|w|_{2,P}\quad\text{and}\quad\Big|\int_{E}w_\bn\,ds\Big|\leq C_T(1+C_\PF)h_P|w|_{2,P}.\label{2.15a}
	\end{align}
	This concludes the proof of the lemma with $C_d:=(C_S+C_T)(1+C_\PF)$.
\end{proof}
The following lemma estimates   $|Gv|_{m,P}$ for $m=0,1,2$ and  the Galerkin projection $G$  in terms of dofs of $v$ for any $v\in H^2(P)$.
\begin{lemma}[$G$ as a function of Dof]\label{lem2.2}
	The projection operator $Gv$  is  computable in terms of the degrees of freedom $\mathrm{Dof}(v)\in\mathbb{R}^{2N_P}$ for any $v\in H^2(P)$ and	 
	$\sum_{m=0}^2h_P^{m-1}|Gv|_{m,P}\leq C_g|\mathrm{Dof}(v)|_{\ell^2}$
	holds with  a positive constant $C_g$ (that exclusively depends on $\rho$).
\end{lemma}
\begin{proof}
	 Antonietti \textit{et al.} discuss the proof \cite[Lemma~3.3]{antonietti2018fully} of the computability of the projection operator $G$ in terms of the dofs from \eqref{dof}. Appendix~A provides  details of this first part and the proof of the estimates of $|Gv|_{m,P}$ for $m=0,1,2$.
\end{proof}

\section{Abstract framework and fundamental estimates}
 
\subsection{Hypotheses}
Given any polygonal domain  $P\in\T\in\mathbb{M}$, recall the geometry and the local degrees of freedom from  Subsection~2.1 and merely suppose the hypotheses \ref{H1}-\ref{H2} throughout this paper.
\begin{enumerate}[label={(\bfseries H\arabic*)}]
\item\label{H1} The vector space $V_h(P)$ is of dimension $2N_P$, satisfies $\displaystyle \p_2(P)\subseteq V_h(P)\subset H^2(P)$, and the triplet $(P,V_h(P),(\text{dof}_1,\dots,\text{dof}_{2N_P}))$ is a finite element in the sense of Ciarlet.
\end{enumerate}
\noindent The unique existence of  a nodal basis $\psi_1,\dots,\psi_{2N_P}$ of $V_h(P)$ with $\text{dof}_k(\psi_j)=\delta_{jk}$ for all $j,k=1,\dots,2N_P$ is a consequence for  any finite element in the sense of Ciarlet \cite[Chapter~3]{7}.
\begin{enumerate}[label={(\bfseries H2)}]
\item\label{H2} The aforementioned nodal basis functions $\psi_1,\dots,\psi_{2N_P}$   satisfy
$h_P\Big(\sum_{j=1}^{2N_P}|\psi_j|^2_{2,P}\Big)^{1/2}\leq C_{\text{stab}}$
for a positive constant $C_{\text{stab}}$ (that exclusively depends on $\rho$).
\end{enumerate}
Notice that \ref{H1}-\ref{H2} also imply the uniform stability of the discrete problem for the natural stabilization term $s_h$ in \eqref{ex} below. The road map of the proofs in the two examples below is outlined in the seminal contribution \cite{chen2020nonconforming} for a related virtual element space $V_h(P)$.  Appendix~B and C independently provide  details for the two examples below. 
\subsection{Examples of the discrete space $V_h(P)$}
This subsection presents two  examples \cite{antonietti2018fully,zhao2018morley} of the local discrete space $V_h(P)$ with \ref{H1}-\ref{H2}.  Recall that $\p_r(P)$ is the vector space of polynomials of degree $\leq r$ regarded as functions in $P$, and fix the parameter $r=-1,0,1,2$ with the convention $\p_{-1}(P)=\{0\}$.

\subsubsection{First example of $V_h(P)$ from \cite{antonietti2018fully, chen2020nonconforming}} 
The discrete space in \cite[Sec.~3]{antonietti2018fully} and in \cite[Sec.~2.3]{chen2020nonconforming}  solves the biharmonic problem with   boundary conditions for $\p_0(\e(P)):=\{q\in L^\infty(\partial P):\forall E\in\e(P)\quad q|_E\in\p_0(E)\}$,
\begin{align}
\widehat{V}_h(P)&:=\begin{rcases}\begin{dcases} v\in H^2(P):  \exists\; f\in \p_r(P)\; \exists\;g\in\p_0(\e(P))\;\exists\; a_1,\dots,a_{N_P}\in\mathbb{R}\\\quad \forall\;w\in H^2(P)\quad a^P(v,w)=(f, w)_{L^2(P)}+(g, w_\bn)_{L^2(\partial P)}+\sum_{j=1}^{N_P} a_j w(z_j)\end{dcases}\end{rcases},\label{vhp}\\
V_h(P)&:=\{v\in\widehat{V}_h(P):v-Gv\perp\p_r(P)\quad\text{in}\;L^2(P)\}.\label{hvhp}
\end{align}
\begin{proposition}The discrete  space $V_h(P)$ from \eqref{vhp}-\eqref{hvhp}   satisfies  \ref{H1}-\ref{H2}.
\end{proposition}
\begin{proof} 
The arguments in \cite[Lemma~3.4-3.5]{chen2020nonconforming} and in \cite[Appendix~A]{chen2020nonconforming} can be adopted    for the proof of \ref{H1} and  of \ref{H2} for $r=-1$. Appendix~B presents a simpler proof that also covers $r=0,1,2$. 	
\end{proof}

\subsubsection{Second example of $V_h(P)$ generalizes \cite{zhao2018morley}}
Recall the set $\e(P)=\{E(1),\dots,E(N_P)\}$ of edges   and the set  $\mathcal{V}(P)=\{z_1,\dots,z_{N_P}\}$ of vertices  along the polygon $\partial P=E(1)\cup\dots \cup E(N_P)$ as in Subsection~2.1. The following generalization  of $V_h(P)$ in \cite[Sec.~4]{zhao2018morley} requires further notation with corners as depicted in Figure~2.1.c. The boundary of  the polygon $\partial P:=\cup_{j=1}^{J}\text{conv}\{\zeta_j,\zeta_{j+1}\}$  is also a polygon of the corner points $\zeta_1,\dots,\zeta_{J}\subset \{z_1,\dots,z_{N_P}\}$ with  indices $\zeta_j=z_{k(j)}$, $\zeta_{j+1}=z_{k(j)+m(j)}$, $k(j+1)=k(j)+m(j)$, and $k(J+1):=k(1)$. By definition,  the interior angle at a corner $\zeta_j$ is different from $0,\pi,2\pi$, while it is equal to $\pi$ at all other vertices $z_j\in\mathcal{V}(P)\setminus\{\zeta_1,\dots,\zeta_J\}$. Given the one-dimensional side $\overline{\gamma(j)}:=\text{conv}\{\zeta_j,\zeta_{j+1}\}= E(k(j))\cup\dots\cup E(k(j)+m(j))$, consider the $(m(j)+2)$-dimensional quadratic $C^1$ spline space 
\[S(j):=\p_2(\e(\gamma(j)))\cap C^1(\gamma(j))\quad\text{for}\;j=1,\dots,J.\]
The vertices $z_{k(j)},\dots,z_{k(j)+m(j)}$ on $\gamma(j)$ lead to a partition of $\gamma(j)$, written as $\e(\gamma(j))$, and the subset of functions in $S(j)$ that vanish at all those vertices form a one-dimensional subspace span$\{\psi_j\}$ of $S(j)$. This is elementary to verify and Appendix~C exploits pictures and norms of $\psi_j$. It turns out that  two conditions on the sign $\psi_j|_{E(k(j))}\geq 0$ and on the scaling $\|\psi_j\|_{L^\infty(\gamma(j))}=1$ determine $\psi_j$ uniquely. So  $\psi_j\in S(j)$ is fixed by the geometry of $P$ in Figure~2.1.c. The second class of VEs is generalized through a  linear functional $\Lambda_j:S(j)\to\mathbb{R}$ with the normalization $\Lambda_j(\psi_j)=1$ and the boundedness $\|\Lambda_j\|\leq C_\Lambda$ of the operator norm $\|\Lambda_j\|:=\sup\{\Lambda_j(f):f\in S(j), \|f\|_{L^\infty(\gamma(j))}=1\}$ of $\Lambda_j$, provided $S(j)$ is endowed with the maximum norm. We suppose that the upper bound    $C_\Lambda$  exclusively depends on $\rho$.   Abbreviate $W:=H^1_0(P)\cap H^2(P)$ and define
\begin{align}
\widehat{W}_h(P)&:=\begin{cases}\begin{rcases}
w\in H^2(P)&: w|_{\partial P}\in C^0(\partial P),\quad\forall{j=1,\dots,J}\quad w|_{\gamma(j)}\in S(j),\;\text{and}\\&\exists f\in\p_r(P)\quad\exists g\in\p_0(\e(P))\quad\forall \phi\in W \\&a^P(w,\phi)=(f, \phi)_{L^2(P)}+(g,\phi_{\bn})_{L^2(\partial P)}\end{rcases},\end{cases}\label{whp}
\\
W_h(P)&:=\begin{cases}\begin{rcases}w\in\widehat{W}_h(P):&\forall j=1,\dots,J\quad \Lambda_j((w-Gw)|_{\gamma(j)})=0\;\text{and}\\& w-Gw\perp\p_r(P)\quad\text{in}\;L^2(P)\end{rcases}.\end{cases}\label{hwhp}
\end{align}
\begin{proposition}The discrete  space  $W_h(P)$ from \eqref{whp}-\eqref{hwhp}  satisfies  \ref{H1}-\ref{H2}.
\end{proposition}
\begin{proof} 
	The proof of \ref{H1} for  $\Lambda_j$ from Example~\ref{ex2} below and for $r=-1$ is given in \cite[Lemma~4.1]{zhao2018morley}.  A proof of \ref{H2} with clear dependence of the constant $C_{\text{stab}}$ on the mesh regularity parameter $\rho$ seems missing in the literature.  Appendix~C contains further details for the general case  and the proof of \ref{H2} with explicit constant.
\end{proof} 

 \begin{example}[Example of $\Lambda_j$]\label{ex1}  Given the first edge $E(k(j))$  in $\gamma(j)$,  define $\Lambda_j(v):=\frac{3}{2}\dashint_{E(k(j))}v\,ds$ for $v\in L^1(\gamma(j))$.  Then $\Lambda_j(\psi_j)=1$ and $\|\Lambda_j\|\leq C_\Lambda=3/2$. 
\end{example}
\begin{example}[Comparison with \cite{zhao2018morley}]\label{ex2}
	Zhao \textit{et al.}  consider merely convex polygons  in \cite{zhao2018morley} and we interpret that this implies that the sides are the edges ($\zeta_j=z_j$ for all $j=1,\dots,J$ and $J=N_P$; all interior angles in $P$ are different from $\pi$).  Then their choice  $\Lambda_j(\bullet)=\dashint_{E(j)}\bullet\,ds$ coincides with Example~\ref{ex1} and recovers $V_h(P)$ in  \cite{zhao2018morley}     for certain geometries.
\end{example}

We continue the more  general discussion and point out that, for each $P\in\T$, $V_h(P)$ can even  be selected from either \eqref{vhp}-\eqref{hvhp} or \eqref{whp}-\eqref{hwhp} and a mixture is allowed in the abstract framework at hand. In particular, for different polygons $P\in\T$, the space $V_h(P)$ and the parameter $r$ could be different.

\subsection{Interpolation}
Given $P\in\T\in\mathbb{M}$, recall from \ref{H1}-\ref{H2} the nodal basis $(\psi_1,\dots,\psi_{2N_P})$   of $V_h(P)$ with dof$_j(\psi_k)=\delta_{jk}$ for $j,k=1,\dots,2N_P$.

\begin{definition}[local interpolation operator]\label{l}
  	Define  $I_h^P:H^2(P)\to V_h(P)$ by
	\begin{align}
	I_h^P v =\sum_{j=1}^{2N_P} \text{dof}_j(v)\psi_j\quad\text{for all}\;v\in H^2(P).\label{2.18}
	\end{align}
\end{definition}
\noindent Recall the notation $V_\nc$ from Subsection~2.1 for $\T\in\mathbb{M}$.
\begin{definition}[global discrete space]\label{gd} The (nonconforming) global virtual element space  is  the  collection of all local spaces $V_h(P)$ for $P\in\T$ with well-defined global dofs from \eqref{gdof}, namely
\begin{align*}
V_h:=\begin{rcases}\begin{dcases}v_h\in V_\nc: &\forall \; P\in\T\quad v_h|_P\in V_h(P)\end{dcases}\end{rcases}.
\end{align*}
\end{definition}
\begin{definition}[global interpolation operator]\label{gi}
	 Define the global interpolation operator $I_h:V_\nc\to V_h$ by $(I_hv_\nc)|_P:=I_h^P(v_\nc|_P)$ for all $P\in\T$. (The global interpolation $I_h$ is well-defined  because the dofs from \eqref{dof} are uniquely defined for any $v_\nc\in V_\nc$.)
\end{definition}
\subsection{Interpolation error estimates}
The main results about $I_h$ are summarized as follows.
\begin{theorem}[interpolation]\label{2.10}
	There exists a positive constant $C_{\mathrm{I}}$ (that  exclusively depends on $\rho$) such that any $v\in H^2(P)$ and its interpolation $I_h^Pv\in V_h(P)$ from Definition~\ref{l} satisfy
	\begin{enumerate}[label=$\left(\alph*\right)$]
		\item \label{ip}
		$\displaystyle
		D^2_\pw(v-I_h^Pv)\perp \p_0(P;\mathbb{S})\quad \text{in} \;L^2(P;\mathbb{S}),
		$
		\item $GI_h^Pv=Gv$ and $\displaystyle  \sum_{m=0}^2h_P^{m-2}|v-G I_h^Pv|_{m,P}\leq (1+C_\PF)|v-\pid v|_{2,P},$
		\item $\displaystyle
		|I_h^Pv|_{2,P}\leq C_{\mathrm{Ib}}|v|_{2,P},
		$
		\item \label{25.2}
		$\displaystyle
		\sum_{m=0}^2h_P^{m-2}|v-I_h^Pv|_{m,P}\leq C_\mathrm{I}|v-\pid v|_{2,P}\leq C_{\mathrm{I}}C_{\mathrm{apx}}h_P^s|v|_{2+s,P}$
	\end{enumerate}
	provided $v\in H^{2+s}(P)$ for  $0<s\leq 1$ in the last estimate in $(d)$ with $C_{\mathrm{apx}}$ from Lemma~\ref{lem2.1}.
\end{theorem}
\begin{proof}[Proof of $(a)$]
	Since   $v$ and $I_h^Pv$ coincide at the  vertices from Definition~\ref{l} of $I_h^P$, their tangential derivatives satisfy
	\begin{align}
	\int_E (I_h^Pv)_{\bt}\,ds =I_h^Pv(z_2)-I_h^Pv(z_1) =v(z_2)-v(z_1)=\int_E v_{\bt}\,ds\label{t}
	\end{align}
	for the vertices $z_1,z_2$ of an edge $E$  directed from $z_1$ to $z_2=z_1+h_E\bt_E$. An integration by parts proves the first step and the split $\nabla v=v_{\bn}\bn+v_{\bt}\bt$ proves the second step in  \begin{align}\int_P D^2v\,dx&=\sum_{E\in\e(P)}\left(\int_E \nabla v\,ds\right)\otimes \bn_P|_E =\sum_{E\in\e(P)}\left(\int_E (v_{\bn}\bn+v_{\bt}\bt)\,ds\right)\otimes\bn_P|_E.\label{2.15}\end{align}
	The combination of \eqref{t}-\eqref{2.15}   and $\int_E (I_h^Pv)_\bn\,ds=\int_E(v)_\bn\,ds$ from  Definition~\ref{l} shows 
	\begin{align}
	\int_P D^2v\,dx=\sum_{E\in\e(P)}\left(\int_E ((I_h^Pv)_{\bn}\bn+(I_h^Pv)_{\bt}\bt)\,ds\right)\otimes\bn_P|_E =\int_PD^2I_h^Pv\,dx\label{2.27}
	\end{align}
	with the split for $\nabla I_h^Pv=(I_h^Pv)_{\bn}\bn+(I_h^Pv)_{\bt}\bt$ and an integration by parts in the last step. 
\end{proof}
\begin{proof}[Proof of $(b)$]
	Since the dofs of $I_h^Pv$ and $v$ coincide by Definition~\ref{l} and $G$ is uniquely determined by  Lemma~\ref{lem2.2}, $GI_h^Pv=Gv$ for  $v\in H^2(P)$. Hence Lemma~\ref{lem2.1} concludes the proof.
\end{proof}
\begin{proof}[Proof of $(c)$]
	Since $I_h^Pv-Gv\in V_h(P)$ for any $v\in H^2(P)$, the nodal basis functions $\psi_j\in V_h(P)$ for $j=1,\dots,2N_P$ and  the dofs from \eqref{dof} with $\text{Dof}(I_h^Pv)=\text{Dof}(v)$ lead to
	\begin{align*}
	I_h^Pv-Gv=\sum_{j=1}^{2N_P}\text{dof}_j(I_h^Pv-Gv)\psi_j=\sum_{j=1}^{2N_P}\text{dof}_j(v-Gv)\psi_j.
	\end{align*}
	 The Cauchy-Schwarz inequality 
	$
	|I_h^Pv-Gv|_{2,P}\leq h_P^{-1}|\text{Dof}(v-Gv)|_{\ell^2}h_P\Big(\sum_{j=1}^{2N_P}|\psi_j|_{2,P}^2\Big)^{1/2},
	$
	 Lemma \ref{lem2.1}-\ref{lem}, and \ref{H2} result in 
	\begin{align}
	|I_h^Pv-Gv|_{2,P}\leq C_dC_{\text{stab}}|v|_{2,P}.\label{2}
	\end{align} 
	Since $|Gv|_{2,P}\leq |v|_{2,P}$ (from \eqref{pid1}), \eqref{2} and the triangle inequality  $|I_h^Pv|_{2,P}\leq |I_h^Pv-Gv|_{2,P}+|Gv|_{2,P}$ conclude the proof with $C_{\text{Ib}}:=1+C_dC_{\text{stab}}$.
\end{proof}
\begin{proof}[Proof of $(d)$]
	Substitute $w:=v-I_h^Pv\in H^2(P)$ for $v$ in $(b)$ and observe $I_h^Pw=0$, so $GI_h^Pw=0,$ to derive 
	\begin{align}
	(1+C_\PF)^{-1}\sum_{m=0}^2h_P^{m-2}|w|_{m,P}\leq |w-\pid w|_{2,P}\leq |v-\pid v|_{2,P}+|I_h^Pv-\pid I_h^Pv|_{2,P}\label{3.10}
	\end{align}
	with a triangle inequality in the last step.  Let $g:=v-GI_h^Pv\in H^2(P)$ with $I_h^Pg=I_h^Pv-\pid I_h^Pv$ from $I_h^P(GI_h^Pv)=GI_h^Pv$. Then $(b)$-$(c)$ result in $|I_h^Pg|_{2,P}\leq C_{\text{Ib}}|g|_{2,P}\leq C_{\text{Ib}}|v-\pid v|_{2,P}$. This and \eqref{3.10} conclude the proof of $(d)$ with $C_{\text{I}}:=(1+C_\PF)(1+C_{\text{Ib}})$.
\end{proof}
\section{Conforming companion}
Recall  that  $\tT$ is a shape-regular sub-triangulation of the polygonal mesh $\T\in\mathbb{M}$ of the domain $\Omega$  from  Subsection~2.1 with the set of edges $\eE\supset\e$. 
\subsection{Morley interpolation}
The Morley finite element space $\M(\tT)$ is defined  as
\begin{align*}
\M(\tT):=\begin{rcases}\begin{dcases}
v_\M\in\p_2(\tT):&v_\M\; \text{is continuous at  interior vertices  and zero at boundary}\\& \text{vertices of $\tT$, and $(v_\M)_\bn$ is continuous at  midpoints of interior}\\&\text{edges and zero at midpoints of boundary edges of $\eE$}
\end{dcases}
\end{rcases}.
\end{align*}
The Morley interpolation operator $I_\M:V_h\to \M(\tT)$ is (uniquely) defined by
\begin{align}
I_\M v_h(z) = v_h(z)\quad\text{at}\;z\in\widehat{\mathcal{V}}\;\text{and}\;
\int_E(I_\M v_h)_{\bn}\,ds = \int_E (v_h)_\bn\,ds\quad\text{on}\;E\in\eE.\label{m}
\end{align}
A well-known consequence of \eqref{m} \cite[Lemma~3.1]{carstensen2021lower} (follows as the proof of Theorem~\ref{2.10}.a)  is
\begin{align}
D^2_{\pw}(v_h-I_\M v_h)\perp\p_0(\tT;\mathbb{S})\quad\text{in}\; L^2(\Omega;\mathbb{S}).\label{m1}
\end{align}
The interpolation error estimate from \cite[Theorem~3]{carstensen2014guaranteed} (for  functions in $H^2(T)$) implies
\begin{align}
\frac{1}{2}\sum_{m=0}^2|h_{\tT}^{m-2}(v_h-I_\M v_h)|_{m,\pw}\leq |v_h-Gv_h|_{2,\pw}.\label{2.20}
\end{align}
\begin{lemma}[conforming companion for Morley \cite{gallistl2015morley}]\label{J'}
	There exists a linear map $J':\M(\tT)\to H^2_0(\Omega)$ and a  constant $C_{J'}$ (that exclusively depends on $\rho$) such that   any $v_\M\in \M(\tT)$ satisfies \\
		$(a)$ $ J'v_\M(z) = v_\M(z)$ at all vertices $z\in\widehat{\mathcal{V}}$,
		$(b)$ $ \dashint_E (J'v_\M)_{\bn}\,ds = \dashint_E (v_\M)_{\bn}\,ds$ on all edges $E\in{\eE}$,\\
		$(c)$ $ \displaystyle D^2_{\pw}(v_\M-J'v_\M)\perp \p_0(\tT;\mathbb{S})$ in $L^2(\Omega;\mathbb{S})$,\\
		$(d)$  $\displaystyle \sum_{m=0}^2|h_{\tT}^{m-2}(v_\M-J'v_\M)|_{m,\pw}\leq C_{J'}\min_{v\in V}|v_\M-v|_{2,\pw}$.\qed
\end{lemma}
\subsection{Computable Morley interpolation}
The virtual element function $v_h\in V_h$ is given implicitly  such that the computation of $I_\M v_h$ is possible in principle, but too costly. The aim in this section is the analysis of a  function $v_\M\in \M(\tT)$ that is computable in terms of dofs of $v_h\in V_h$. Given any $v_h\in V_h$ set
\begin{align}
v_\M(z) :=\begin{cases}v_h(z)\;\text{at}\;z\in\mathcal{V},\\
Gv_h(z)\;\text{at}\;z\in\widehat{\mathcal{V}}\setminus \mathcal{V}\end{cases}\;\text{and}\quad
\int_E( v_\M)_\bn\,ds:=\begin{cases}\int_E(v_h)_\bn\,ds\;\text{on}\;E\in\e,\\
\int_E(Gv_h)_\bn\,ds\;\text{on}\;E\in\eE\setminus\e. \end{cases}\label{newm}
\end{align}
Notice that $v_\M\in\M(\tT)$ is well-defined and computable in terms of the dofs of $v_h\in V_h$ (for $z\in\mathcal{V}, E\in\e$) and  from $Gv_h\in \p_2(\tT)$ with Lemma~\ref{lem2.2} (for $z\in\widehat{\mathcal{V}}\setminus\mathcal{V}, E\in\eE\setminus\e$).  The following estimates control the difference between  $I_\M v_h$ and $v_\M$ in  $\M(\tT)$ for  $v_h\in V_h$.
\begin{lemma}[key]\label{M}
	There exists  a positive constant $C_\M$ (that exclusively depends on $\rho$) with
	\begin{align*}
	\sum_{m=0}^{2}|h_\tT^{m-2}(I_\M v_h-v_\M)|_{m,\pw}\leq C_\M|v_h-Gv_h|_{2,\pw}\quad\text{for any}\; v_h\in V_h\;\text{and}\;v_\M\in \M(\tT)\;\text{with}\; \eqref{newm}.
	\end{align*}
\end{lemma}
\begin{proof}
	Recall the sub-triangulation $\tT(P)$ of $P$ with  mid-point $z_0$ and the set of edges $\eE(P)$ (resp. $\e(P)$)  in $\tT(P)$ (resp. $P$) from Subsection~2.1 for a polygonal domain $P\in\T$. Let $\psi_{z}$ and  $\psi_E$ be the nodal basis functions of $\M(\tT)$ for $z\in\widehat{\mathcal{V}}$ and $E\in\eE$, and write  $I_\M v_h- v_\M\in \M(\tT)$ as
	\begin{align}
	(I_\M v_h- v_\M)|_P = (v_h-G v_h)(z_0)\psi_{z_0}+\sum_{E\in\eE(P)\setminus\e (P)}\left(\dashint_E(v_h-Gv_h)_\bn\,ds\right)\psi_E.\label{5.7}
	\end{align}
	The  basis functions $\psi_z$ and $\psi_E$ are written explicitly, e.g., in  \cite[Sec.~6]{carstensen2014discrete} and scale like
	\begin{align}
	|\psi_{z_0}|_{m,T}\approx h_T^{1-m}\quad\text{and}\quad|\psi_E|_{m,T}\approx h_T^{2-m}\quad\text{for}\;m=0,1,2\label{5.10}
	\end{align}
	in a triangle $T\in\tT(P)$ with $E\subset\partial T$. Note that the integral means of normal derivatives of $\psi_E$ over edges enter \eqref{5.7} and \textit{not} the dofs for the ncVEM  from \eqref{dof}. Lemma~\ref{lem} with integral means along edges $E\in\eE(P)\setminus\e(P)$ in \eqref{2.15a} and
	 $h_{P}\leq \rho^{-1}h_E$  from (M2) imply that
	\begin{align}
	h_P^{-1}|(v_h-Gv_h)(z_0)|+\Big|\dashint_{E}(v_h-Gv_h)_\bn\,ds\Big|\lesssim |v_h-Gv_h|_{2,P}.\label{3.5n}
	\end{align}   The combination \eqref{5.7}-\eqref{3.5n}  concludes the proof.
\end{proof}
\subsection{Computable companion for VEM}
The new tool in this paper is the conforming companion operator $J$ established in Theorem~\ref{J}. The estimate in part (c) of which is sharp in the sense that the reverse inequality also holds (with a multiplicative constant $2$ from a triangle inequality).
\begin{theorem}[companion operator for VEM]\label{J}
		There exist a linear map $J:V_h\to H^2_0(\Omega)$, which is a right-inverse to the interpolation operator $I_h$, and a universal constant  $C_{J}$ (that exclusively depends on $\rho$) such that any $v_h\in V_h$ satisfies  $(a)$-$(c)$.\\
		$(a)$  $ D^2_{\pw}(v_h-Jv_h)\perp \p_0(\T;\mathbb{S}) \quad\text{in}\;L^2(\Omega;\mathbb{S}),\qquad$
		$(b)$ $Gv_h-Jv_h\perp\p_2(\T)\quad\text{in}\;L^2(\Omega)$,\\
		$(c)$  $\displaystyle \sum_{m=0}^2(|h_\T^{m-2}(Gv_h-Jv_h)|_{m,\pw}+|h_\T^{m-2}(v_h-Jv_h)|_{m,\pw}) \leq C_{J}\big(|v_h-Gv_h|_{2,\pw}+\min_{v\in V}|v_h-v|_{2,\pw}\big).$
\end{theorem}
\begin{proof}
	\noindent\textit{Construction of $J$}.
	For a given $v_h\in V_h$, a composite function $J' v_\M$ belongs to $V=H^2_0(\Omega)$ and we need a modification to achieve $(b)$. For each $T\in\tT(P)$ from Subsection~2.1,   the cubic bubble-function $b_T:=27\lambda_1\lambda_2\lambda_3\in H_0^{1}(T)$ for the barycentric co-ordinates $\lambda_1, \lambda_2, \lambda_3\in\p_1(T)$ of $T$, leads to $b_T^2\in H^2_0(T)$ with $0\leq b_T^2\leq 1$ and $\dashint_T b_T^2\,dx=81/280$. Hence $\dashint_{P}b_P\,dx=1$ follows for 
	\begin{align}
	b_P:=\frac{280}{81}\sum_{T\in\tT(P)}b_T^2\in H_0^{2}(P)\subset H^2_0(\Omega)\quad\text{for}\;P\in\T.\label{bubble}
	\end{align} 
	 Let $v_P\in\p_2(P)$  be the Riesz representation of the linear functional $\p_2(P)\to\mathbb{R}$,
	$
	w_2\mapsto (Gv_h-J'v_\M,w_2)_{L^2(P)}
	$
	in the Hilbert space $\p_2(P)$ endowed with the weighted scalar product $(b_P\bullet,\bullet)_{L^2(P)}$. In other words, given $v_h\in V_h$, $v_\T\in\p_2(\T)$ with $v_\T|_P:=v_P$ satisfies
	\begin{align}(b_Pv_\T,w_2)_{L^2(\Omega)}=(Gv_h-J' v_\M,w_2)_{L^2(\Omega)}\quad\text{for all}\;w_2\in\p_2(\T);\label{*}\end{align} whence  $\Pi_2(Gv_h-J' v_\M) = \Pi_2(b_Pv_P)$. Given $v_h\in V_h$ with $v_\T\in\p_2(\T)$ define
	\begin{align}
	Jv_h:= J' v_\M +\sum_{P\in\T} b_Pv_\T \in V.\label{J1} 
	\end{align}  


\begin{proof}[Proof of $(a)$]
	 The definitions of $J$ and $b_P\in H^2_0(P)$ imply for any vertex $z\in\mathcal{V}$ and for any edge $E\in\e$ that $Jv_h(z) = J'v_\M(z)$ and  $\dashint_E (Jv_h)_{\bn}\,ds = \dashint_E(J' v_\M)_{\bn}\,ds.$ This,  Lemma~\ref{J'}.a-b as $\mathcal{V}\subset\widehat{\mathcal{V}}$ and $\e\subset\eE$, and \eqref{newm}  lead to
	\begin{align}
	Jv_h(z)& = J'v_\M(z)=v_\M(z)=v_h(z),\label{a0}\\
	\dashint_E (Jv_h)_{\bn}\,ds&=\dashint_E (J' v_\M)_{\bn}\,ds=\dashint_E ( v_\M)_{\bn}\,ds=\dashint_E (v_h)_{\bn}\,ds. \label{a1}
	\end{align}
	The proof of $(a)$ follows from  \eqref{a0}-\eqref{a1} as in the proof of Theorem~\ref{2.10}.a.
\end{proof}
\begin{proof}[Proof of $(b)$]
	The definition of $J$ in \eqref{J1} directly shows $\Pi_2(Gv_h)=\Pi_2(Jv_h)$. 
\end{proof}
\begin{proof}[Proof of $(c)$]
	Abbreviate $\overline{v}:=N_P^{-1}\sum_{j=1}^{N_P}v(z_j)$ for any $v\in H^2(P)$. The definition of $G$ in \eqref{pid2} and \eqref{a0} imply $\overline{Gv_h-Jv_h}=\overline{Gv_h-v_h}+\overline{v_h-Jv_h}=0$. The split $\nabla Jv_h:=(Jv_h)_{\bn} \bn+(Jv_h)_{\bt}\bt$ and an integration of the tangential component show for $k=1,\dots,N_P$ that \begin{align}\int_{E(k)}\nabla Jv_h\,ds&=\int_{E(k)} (Jv_h)_{\bn} \bn_{E(k)}\,ds+(Jv_h(z_{k+1})-Jv_h(z_k)))\bt_{E(k)}\nonumber\\&=\int_{E(k)} (v_h)_{\bn} \bn_{E(k)}\,ds+(v_h(z_{k+1})-v_h(z_k)))\bt_{E(k)}=\int_{E(k)}\nabla v_h\,ds\label{jv}\end{align}
	 with \eqref{a0}-\eqref{a1} in the second step and $\nabla v_h:=(v_h)_{\bn} \bn+(v_h)_{\bt}\bt$ in the last step. 
	 This and \eqref{pid2} lead to $\int_{\partial P}\nabla(Gv_h-Jv_h)\,ds=\int_{\partial P}\nabla(Gv_h-v_h)\,ds+\int_{\partial P}\nabla(v_h-Jv_h)\,ds=0$. Hence the Poincar\'e-Friedrich inequality in Theorem~\ref{PF}.b implies \begin{align}\sum_{m=0}^2|h_\T^{m-2}(Gv_h-Jv_h)|_{m,\pw}\leq (1+C_\PF)|Gv_h-Jv_h|_{2,\pw}.\label{4.9}\end{align}
	 It remains to control $|Gv_h-Jv_h|_{2,\pw}$. 	There exists  a positive constant $C_b$ in the inverse estimates 
	 \begin{align}
	 C_b^{-1}\|\chi\|^2_{L^2(P)}\leq &(b_P,\chi^2)_{L^2(P)}\leq C_b\|\chi\|^2_{L^2(P)},\label{b1}\\
	 C_b^{-1}\|\chi\|_{L^2(P)}\leq&\sum_{m=0}^2 h_P^m|b_P\chi|_{m,P}\leq C_b\|\chi\|_{L^2(P)}\quad \text{for any}\; \chi\in \p_2(P).\label{b2}
	 \end{align}
	 The inequalities \eqref{b1}-\eqref{b2} are standard inverse estimates for a (shape-regular) triangle (replace $P$ by $T\in\tT(P)$ therein) and \eqref{bubble} reveals \eqref{b1}-\eqref{b2} as a sum of those.  Hence $C_b$ depends exclusively on $\rho$. The definition of $J$ in \eqref{J1} shows $|J'v_\M-Jv_h|_{2,P}=|b_Pv_P|_{2,P}$ and the inverse inequality \eqref{b2} proves $ |b_Pv_P|_{2,P}	\leq C_bh_P^{-2}\|v_P\|_{L^2(P)}.$ The first inequality in \eqref{b1} and the definition of $v_P\in\p_2(P)$ in \eqref{*} lead to
	 \begin{align*}
	 C_b^{-1}\|v_P\|^2_{L^2(P)}\leq (b_Pv_P,v_P)_{L^2(P)}=(Gv_h-J' v_\M,v_P)_{L^2(P)}\leq\|Gv_h-J' v_\M\|_{L^2(P)}\|v_P\|_{L^2(P)}
	 \end{align*}
	 with  a Cauchy-Schwarz inequality in the last step. Hence $\|v_P\|_{L^2(P)}\leq C_b\|Gv_h-J' v_\M\|_{L^2(P)}$ and the combination  with the above   estimates verifies \[|J'v_\M-Jv_h|_{2,P}=|b_Pv_P|_{2,P}\leq C_b^2h_P^{-2}\|(Gv_h-J' v_\M)\|_{L^2(P)}.\] Triangle inequalities and Lemma~\ref{J'}.d lead, for any $v\in V$, to
	 \begin{align}
	 C_b^{-2}|J'v_\M-Jv_h|_{2,\pw}
	 &\leq
	 \|h_\T^{-2}(Gv_h-I_\M v_h)\|_{L^2(\Omega)}+\|h_\T^{-2}(I_\M v_h- v_\M)\|_{L^2(\Omega)}\nonumber\\&\quad+C_{J'}(|v_\M-I_\M v_h|_{2,\pw}+|I_\M v_h-v_h|_{2,\pw}+|v_h-v|_{2,\pw}).\label{4.11a}
	 \end{align} 
	 The estimation of the upper bound in \eqref{4.11a} involves a few arguments like the triangle inequality
	 $\|h_\T^{-2}(Gv_h-I_\M v_h)\|_{L^2(\Omega)}\leq \|h_\T^{-2}(Gv_h- v_h)\|_{L^2(\Omega)}+\|h_\T^{-2}(v_h-I_\M v_h)\|_{L^2(\Omega)}$, $ \|h_\T^{-2}(Gv_h- v_h)\|_{L^2(\Omega)}\leq C_\PF|v_h-Gv_h|_{2,\pw}$ from Lemma~\ref{lem2.1}, $\|h_\T^{-2}(v_h-I_\M v_h)\|_{L^2(\Omega)}+|v_h-I_\M v_h|_{2,\pw}\leq 2 |v_h-Gv_h|_{2,\pw}$ from \eqref{2.20}, and $|v_\M-I_\M v_h|\leq C_\M |v_h-Gv_h|_{2,\pw}$ from Lemma~\ref{M}. This and the fact that $v\in V$ is arbitrary eventually result in
	 \begin{align}
	 |J'v_\M-Jv_h|_{2,\pw}\leq C_b^2(2+C_\PF+C_\M+C_{J'}(2+C_\M))\Big(|v_h-Gv_h|_{2,\pw}+\min_{v\in V}|v_h-v|_{2,\pw}\Big)\label{3.8a}
	 \end{align}

	  \noindent as well as $|v_\M-v_h|_{2,\pw}\leq |v_\M-I_\M v_h|_{2,\pw}+|I_\M v_h-v_h|_{2,\pw}\leq (2+C_\M)|v_h-Gv_h|_{2,\pw}$. The latter estimate, 
	  Lemma~\ref{J'}.d,  and a triangle inequality lead, for any $v\in V$, to
	  \begin{align}
	  C_{J'}^{-1}|v_\M-J'v_\M|_{2,\pw}\leq |v_\M-v|_{2,\pw}\leq (2+C_\M)|v_h-Gv_h|_{2,\pw}+|v_h-v|_{2,\pw}.\label{17a}
	  \end{align}
	  Recall $|I_\M v_h-v_\M|_{2,\pw}\leq C_\M|v_h-Gv_h|_{2,\pw}$ from Lemma~\ref{M} and deduce $|Gv_h-I_\M v_h|_{2,\pw}\leq |v_h-Gv_h|_{2,\pw}$ from \eqref{m1}. The aforementioned estimates and \eqref{3.8a}-\eqref{17a} allow the estimation of all terms in the upper bound of the triangle inequality
	 \begin{align*}|Gv_h-Jv_h|_{2,\pw}&\leq|Gv_h-I_\M v_h|_{2,\pw}+|I_\M v_h-v_\M|_{2,\pw}+|v_\M-J'v_\M|_{2,\pw}+|J'v_\M-Jv_h|_{2,\Omega}\\&\leq C_{JP}\big(|v_h-Gv_h|_{2,\pw}+\min_{v\in V}|v_h-v|_{2,\pw}\big) \end{align*}
	  with $C_{JP}:=1+C_\M+C_{J'}(2+C_\M)+C_b^2(2+C_\PF+C_\M+C_{J'}(2+C_\M))$. 
	  This and \eqref{4.9} show
	  \begin{align}
	  \sum_{m=0}^2|h_\T^{m-2}(Gv_h-Jv_h)|_{m,\pw}\leq (1+C_\PF)C_{JP}\big(|v_h-Gv_h|_{2,\pw}+\min_{v\in V}|v_h-v|_{2,\pw}\big).\label{4.8}
	  \end{align}
	Lemma~\ref{lem2.1}, \eqref{4.8}, and   triangle inequalities \[\sum_{m=0}^2|h_\T^{m-2}(v_h-Jv_h)|_{m,\pw}\leq\sum_{m=0}^2|h_\T^{m-2}(v_h-Gv_h)|_{m,\pw}+\sum_{m=0}^2|h_\T^{m-2}(Gv_h-Jv_h)|_{m,\pw}\] conclude the proof of $(c)$ with $C_J:=(1+C_\PF)(1+C_{JP})$.
\end{proof}
\noindent\textit{Proof of $I_hJ =\mathrm{id}$ in $V_h$}.
	Recall $N= |\mathcal{V}|+|\mathcal{E}|$ from Subsection~2.1. Definition~\ref{gd} and \ref{H1}-\ref{H2} imply the existence of a (global)  nodal basis $(\psi_1,\dots,\psi_N)$  of $V_h$ with dof$_j(\psi_k)=\delta_{jk}$ for $j,k=1,\dots,N$. Definition~\ref{gi} of $I_h$  and $\text{dof}_j(Jv_h)= \text{dof}_j(v_h)$ from \eqref{a0}-\eqref{a1} show that
	\[
	I_h(Jv_h) = \sum_{j=1}^{N}\text{dof}_j(Jv_h)\psi_j =\sum_{j=1}^{N}\text{dof}_j(v_h)\psi_j = v_h.\qedhere
	\]	
\end{proof}

\begin{remark}[another companion]
	The properties $(a)$-$(c)$ in Theorem~\ref{J} can also be satisfied by other smoothers, e.g., by $J'I_\M$. The latter is not immediately computable; whence  the new definition of $J$  is adopted in this paper to be used as a smoother $Q$ in Section~5.
\end{remark}
\begin{remark}[generalizations]
	Any  linear operator $J:V_h\to V$ with $(a)$-$(c)$ allow the \textit{a priori} and \textit{a posteriori} error estimates in Section~5-6.   The analog design of such a conforming companion is  possible in 3D with the Morley companion operator $J'$ from \cite{carstensen2018prove} in 3D.
\end{remark}

\section{Discrete problem and a priori error analysis}
\subsection{Stabilization}
The  discrete VE functions $v_h\in V_h$ will not be computed explicitly, but  $Gv_h$ will. The resulting discrete counterpart $a_h(\cdot,\cdot)$ in $V_h$ of the scalar product $a(\cdot,\cdot)$  in \eqref{1.2} requires a stabilization. Recall the semi-scalar product $a^P$ from Subsection~2.3,  the number $N_P$ of vertices  for  $P\in\T$,   $\text{dof}_1,\dots,\text{dof}_{2N_P}$ from \eqref{dof}, and define the semi-scalar product $S^P:V_h(P)\times V_h(P)\to\mathbb{R}$, 
\begin{align}
S^P(v_h,w_h) = h_P^{-2}\sum_{j=1}^{2 N_P}\text{dof}_j(v_h)\text{dof}_j(w_h)\quad\text{for}\;v_h,w_h\in V_h(P).\label{ex}
\end{align}
\begin{lemma}\label{lems}
	There exists a positive constant $C_s$ (that exclusively depends on $\rho$) such that
	\begin{align}
	C_s^{-1}a^P(w_h,w_h)\leq S^P(w_h,w_h)\leq C_sa^P(w_h,w_h)\quad\text{for all}\;w_h\in (1-G)V_h(P).\label{stab}
	\end{align}
\end{lemma}
\begin{proof}
	For any $w_h=\sum_{j=1}^{2N_P}\text{dof}_j(w_h)\psi_j\in V_h(P)$ with the nodal basis functions $\psi_1,\dots,\psi_{2N_P}$ of $V_h(P)$ from \ref{H1}-\ref{H2}, a Cauchy-Schwarz inequality and \ref{H2} show that
	\begin{align*}
	|w_h|_{2,P}\leq |\text{Dof}(w_h)|_{\ell^2}\Big(\sum_{j=1}^{2N_P}|\psi_j|_{2,P}^2\Big)^{1/2}\leq C_{\text{stab}}h_P^{-1}|\text{Dof}(w_h)|_{\ell^2}.
	\end{align*}
	The sum over all squared estimates for $P\in\T$ proves the first inequality in \eqref{stab} for $C_s=C_{\text{stab}}^2$. Note that $Gw_h=0$ for each $w_h\in (1-G)V_h(P)$ and Lemma~\ref{lem2.1}-\ref{lem} result in 
	\[h_P^{-1}|\text{Dof}(w_h)|_{\ell^2}=h_P^{-1}|\text{Dof}(w_h-Gw_h)|_{\ell^2}\leq C_d|w_h-Gw_h|_{2,P}\leq C_d|w_h|_{2,P}.\]
	The sum over all squared estimates for $P\in\T$ proves the second inequality in \eqref{stab}. This concludes the proof with $C_s:=\max\{C_{\text{stab}}^2,C_d^2\}$.
\end{proof}
\begin{remark}[generalization]
	The \textit{a priori} and \textit{a posteriori} error analysis in Section~5-6 hold for any semi-scalar product $S^P:V_h(P)\times V_h(P)\to \mathbb{R}$ with \eqref{stab}. 
\end{remark}

\subsection{Discrete problem}
 Recall  $a_\pw$  from Subsection~2.3 and $(S^P : P\in\T)$  from \eqref{ex}. Define the discrete semi-scalar products, for $v_h,w_h\in V_h$, by
 
\begin{align}
a_h(v_h,w_h)&:=\sum_{P\in\T}a_h^P(v_h,w_h):=a_\pw(Gv_h,Gw_h)+s_h(v_h,w_h),\label{4.4}\\ s_h(v_h,w_h)&:=\sum_{P\in\T}S^P((1-G)v_h,(1-G)w_h).\label{4.5}
\end{align}
  Recall that $|\cdot|_{2,\pw}$ defines a norm in $V_h$ (cf. Lemma~\ref{eq}), so $a_h(\cdot,\cdot)$ is a scalar product in $V_h$.

\begin{lemma}[boundedness and ellipticity of $a_h$]\label{lem2.5}
	Any  $v_h, w_h\in V_h$ satisfy
	\begin{align*}
	a_h(v_h,w_h)\leq (1+C_s)|v_h|_{2,\pw}|w_h|_{2,\pw}
	\quad\text{and}\quad
	C_s^{-1}|v_h|^2_{2,\pw} \leq a_h(v_h,v_h).
	\end{align*}
\end{lemma}
\begin{proof}
	The boundedness follows from the definition of $a_h$ and  \eqref{stab}. The ellipticity follows from  \eqref{stab},\eqref{pid1}, and the Pythagoras identity $|v_h|^2_{2,P}=|(1-G)v_h|_{2,P}^2+|Gv_h|_{2,P}^2$.
\end{proof}
\noindent The Riesz representation theorem guarantees the unique existence of a  solution  $u_h\in V_h$ to
\begin{align}
a_h(u_h,v_h) = F_h(v_h)\quad\text{for all}\;v_h\in V_h\label{dp}
\end{align}
for the right-hand side  $F_h(v_h):=F(Qv_h)$ with $Q=G$ (standard VEM)  or  $Q=J$ (smoother).
\subsection{A priori error estimates}
This subsection establishes an error estimate with respect to the  norms $|\cdot|_{2,\pw}$ and   $|\cdot|_{1,\pw}$. Recall elliptic regularity and $0<\sigma\leq 1$ from \eqref{2.14} for the weak solution $u\in H^2_0(\Omega)\cap H^{2+\sigma}(\Omega)$ to \eqref{3} provided $f\in L^2(\Omega)$ with $F(\cdot):=(f,\cdot)_{L^2(\Omega)}$ in $V$. Let $u_h\in V_h$ solve  \eqref{dp} and recall the maximal mesh-size $h_{\text{max}}$. The second part of the assertion implies \eqref{quasi}.  
\begin{theorem}[error estimates]\label{err}
	There exist positive constants $C_1$ and $C_2$ (that exclusively depend on $\rho$) such that  $f\in L^2(\Omega)$ and $Q=G$ or $Q=J$ imply
	\begin{align*}
	&h^{-\sigma}_{\text{max}}(|u-u_h|_{1,\pw}+|u-Gu_h|_{1,\pw})+|u-u_h|_{2,\pw}+|u-Gu_h|_{2,\pw}\\&\qquad\leq C_1(|u-\pid u|_{2,\pw}+\mathrm{osc}_2(f,\T))\leq C_2 h^{\sigma}_{\text{max}}\|f\|_{L^2(\Omega)}.
	\end{align*}
	For $F\in V^*=H^{-2}(\Omega)$ and solely for $Q=J$, it holds (even without extra regularity of $u\in V$)
	\begin{align*}
	h^{-\sigma}_{\text{max}}(|u-u_h|_{1,\pw}+|u-Gu_h|_{1,\pw})+|u-u_h|_{2,\pw}+|u-Gu_h|_{2,\pw}\leq C_1|u-\pid u|_{2,\pw}.
	\end{align*}
\end{theorem}
\begin{proof}
	\textit{Key identity}. 
	Let $e_h:=I_hu-u_h\in V_h$ and $v_h\in V_h$ with the interpolation $I_h$ from Definition~\ref{gi} and $GI_hu=Gu$ from Theorem~\ref{2.10}.b. The discrete problem \eqref{dp} and \eqref{4.4} imply
	\begin{align}
	a_h(e_h,v_h)=a_\pw(Gu,Gv_h)+s_h(I_hu,v_h)-F(Qv_h)=a_\pw(Gu,Jv_h)+s_h(I_hu,v_h)-F(Qv_h)\nonumber
	\end{align}
	with $a_\pw(Gu,G v_h)=a_\pw(Gu, v_h)=a_\pw(Gu,Jv_h)$ from  \eqref{pid1} and  Theorem~\ref{J}.a in the last step.  The continuous problem \eqref{3} with a test function $v=Jv_h$  reveals
	\begin{align}
	a_h(e_h,v_h)=a_\pw(Gu-u,Jv_h)+s_h(I_hu,v_h)+F(Jv_h)-F(Qv_h).\label{4.7}
	\end{align}
	\textit{Estimate of $a_\pw(Gu-u,Jv_h)$}.
	 A triangle inequality, Theorem~\ref{J}.c, and $|v_h-Gv_h|_{2,\pw}\leq |v_h|_{2,\pw}$ (cf. Lemma~\ref{lem2.1}) show \[|Jv_h|_{2,\Omega}\leq|v_h-Jv_h|_{2,\pw}+|v_h|_{2,\pw}\leq C_J(|v_h-Gv_h|_{2,\pw}+|v_h|_{2,\pw})+|v_h|_{2,\pw}\leq (1+2C_J)|v_h|_{2,\pw}.\] This and a  Cauchy-Schwarz inequality provide
	\begin{align}
	a_\pw(Gu-u,Jv_h)\leq(1+2C_J)|u-\pid u|_{2,\pw}|v_h|_{2,\pw}.\label{5.6a}
	\end{align}
	\textit{Estimates of $s_h(I_hu,v_h)$}. 
	A Cauchy-Schwarz inequality for the semi-scalar product $s_h(\cdot,\cdot)$ and \eqref{stab} provide
	$C_s^{-1}s_h(I_hu,v_h)\leq |(1-G)I_hu|_{2,\pw}|(1-G)v_h|_{2,\pw}.
	$
	A triangle equality, Lemma~\ref{lem2.1}, and Theorem~\ref{2.10}.b-d  in the end  result in 
	\begin{align}
	C_s^{-1}s_h(I_hu,v_h)\leq (|u-I_hu|_{2,\pw}+|u-Gu|_{2,\pw})|v_h|_{2,\pw}\leq (1+C_{\text{I}})|u-\pid u|_{2,\pw}|v_h|_{2,\pw}.\label{5.6b}
	\end{align}
	\textit{Estimate of $F(Jv_h)-F(Qv_h)$}. The term $F(Jv_h)-F(Qv_h)$ vanishes for $Q=J$, so let $Q=G$. The orthogonality $Jv_h-Gv_h\perp \p_2(\T)$ in $L^2(\Omega)$ from Theorem~\ref{J}.b  result in 
 \begin{align}
 F(Jv_h)-F(Qv_h)=(h_\T^{2}(f-\Pi_2f),h_\T^{-2}(Jv_h-Gv_h))_{L^2(\Omega)}\leq 2C_{J}\mathrm{osc}_2(f,\T)|v_h|_{2,\pw}. \label{3.2}
 \end{align}
 The last step  follows from  a Cauchy-Schwarz inequality,  Theorem~\ref{J}.c, and Lemma~\ref{lem2.1}.\\
	 \textit{Estimate of $|u-u_h|_{2,\pw}$}. The key identity \eqref{4.7} with $v_h=e_h$, the coercivity of $a_h$ from Lemma~\ref{lem2.5}, and \eqref{5.6a}-\eqref{5.6b} lead to $C_3:=1+2C_J+C_s(1+C_{\text{I}}))$ in
	 \begin{align}
	 |e_h|_{2,\pw}\leq C_3(|u-\pid u|_{2,\pw}+\mathrm{osc}_2(f,\T)).\label{eh}
	 \end{align}
	 A triangle inequality for  $u-u_h=(u-I_hu)+e_h$, Theorem~\ref{2.10}.d, and \eqref{eh} show that
	 \begin{align}
	 |u-u_h|_{2,\pw}\leq (C_{\text{I}}+C_3)(|u-\pid u|_{2,\pw}+\mathrm{osc}_2(f,\T)).\label{h2}
	 \end{align}
	 Recall that $\mathrm{osc}_2(f,\T)$ can be omitted in \eqref{3.2}-\eqref{h2} (and below in \eqref{4.18}-\eqref{h1}) if $Q=J$.\\
	 \textit{Duality solution and its regularity}. Let $z\in H^2_0(\Omega)\cap H^{2+\sigma}(\Omega)$ be the weak solution to $\Delta^2z = -\Delta Je_h\in L^2(\Omega)$ and recall  $\|z\|_{2+\sigma,\Omega}\leq C_{\mathrm{reg}} |\Delta Je_h|_{-1,\Omega}\leq C_{\mathrm{reg}}|Je_h|_{1,\Omega}$ from \eqref{2.14}.  The weak formulation of $\Delta^2z = -\Delta Je_h\in L^2(\Omega)$  leads to
	 $
	 |Je_h|^2_{1,\Omega}=a(Je_h,z).
	 $\\
	 \textit{Reduction to the key term $a_h(e_h,I_hz)$}. Elementary algebra reveals (with $GI_hz=Gz$)
	 \begin{align}
	 |Je_h|^2_{1,\Omega}=a_\pw(Je_h-e_h,z)+a_\pw(e_h,z-Gz)+a_\pw(e_h,Gz). \label{4.12}
	 \end{align}
	 Theorem~\ref{J}.a implies $a_\pw(Je_h-e_h,z)=a_\pw(Je_h-e_h,z-Gz)$. Hence a Cauchy-Schwarz inequality and Theorem~\ref{J}.c imply that
	 \begin{align}
	 a_\pw(Je_h-e_h,z)&\leq (|e_h-Ge_h|_{2,\pw}|+|e_h|_{2,\pw})|z-Gz|_{2,\pw}\nonumber\\&\leq 2C_{\mathrm{apx}}C_Jh^\sigma_{\text{max}}|e_h|_{2,\pw}|z|_{2+\sigma,\Omega}\label{3.5a}
	 \end{align}
	 with Lemma~\ref{lem2.1}  in the end. A Cauchy-Schwarz inequality and Lemma~\ref{lem2.1} show 
	 \begin{align}
	 a_\pw(e_h,z-Gz)\leq |e_h|_{2,\pw}|z-\pid z|_{2,\pw}\leq C_{\text{apx}}h^\sigma_{\text{max}}|e_h|_{2,\pw}|z|_{2+\sigma,\Omega}.\label{4.14}
	 \end{align}
	 \textit{Key identity revisited}. The definition \eqref{4.4} and the identity   \eqref{4.7}    lead, for $v_h=I_hz$, to 
	 \begin{align*}
	 a_\pw(e_h,Gz)=a_h(e_h,I_hz)-s_h(e_h,I_hz)
	 =a_\pw(Gu-u,JI_hz)+F(JI_hz)-F(QI_hz).
	 \end{align*}
	 \textit{Estimate of $a_\pw(Gu-u,JI_hz)$}.
	 The definition of G in \eqref{pid1}  shows $a_\pw(Gu-u,JI_hz)=a_\pw(\pid u-u,JI_hz-GI_hz)$. Theorem~\ref{J}.c (with $v_h=I_hz$ and $v=z$) and Theorem~\ref{2.10}.b imply $|JI_hz-GI_hz|_{2,\pw}\leq C_{J}(|I_hz-Gz|_{2,\pw}+|I_hz-z|_{2,\pw})$. A triangle inequality reveals
	 \begin{align}
	 C_{J}^{-1}|JI_hz-GI_hz|_{2,\pw}\leq (2|z-I_hz|+|z-Gz|_{2,\pw})\leq (2C_{\text{I}}+C_{\text{apx}})h^\sigma_{\text{max}}|z|_{2+\sigma,\Omega}\label{5.20}
	 \end{align}
	 with Theorem~\ref{2.10}.d and Lemma~\ref{lem2.1}   in the last step.    This and a Cauchy-Schwarz inequality prove
	 \begin{align}
	 a_\pw(Gu-u,JI_hz)\leq C_{J}(2C_{\text{I}}+C_{\text{apx}})h^\sigma_{\text{max}}|u-\pid u|_{2,\pw}|z|_{2+\sigma,\Omega}.\label{4.16}
	 \end{align}
	 \textit{Estimate of $F(JI_hz)-F(QI_hz)$}. The term $F(JI_hz)-F(QI_hz)$ vanishes for $Q=J$, so let $Q=G$. The estimate  \eqref{3.2} for $v_h=I_hz$ provides $F(JI_hz)-F(QI_hz)\leq \mathrm{osc}_2(f,\T)\|h_\T^{-2}(JI_hz-GI_hz)\|_{L^2(\Omega)}.$ Theorem~\ref{J}.c (with $v=z$), Theorem~\ref{2.10}.b, and \eqref{5.20} show 
	 \begin{align}
	 F(JI_hz)-F(QI_hz)&\leq C_{J}\mathrm{osc}_2(f,\T)(|I_hz-Gz|_{2,\pw}+|I_hz-z|_{2,\pw})\nonumber\\&\leq C_{J}
	 (2C_{\text{I}}+C_{\text{apx}})h_{\text{max}}^\sigma\mathrm{osc}_2(f,\T)|z|_{2+\sigma,\Omega}.\label{4.17}
	 \end{align}
	 \textit{Estimate of $|u-u_h|_{1,\pw}$}. The combination of \eqref{4.12}-\eqref{4.14} and \eqref{4.16}-\eqref{4.17} imply
	 \[|Je_h|^2_{1,\Omega}\leq C_4h_{\text{max}}^\sigma(|u-\pid u|_{2,\pw}+\mathrm{osc}_2(f,\T))|z|_{2+\sigma,\Omega}
	 \]
	 with $C_4:=C_{\text{apx}}(1+2C_{J})+C_J(2C_{\text{I}}+C_{\text{apx}}))$. This and the aforementioned regularity   show 
	 \begin{align}|Je_h|_{1,\Omega}\leq C_{\text{reg}}C_4h_{\text{max}}^\sigma(|u-\pid u|_{2,\pw}+\mathrm{osc}_2(f,\T)).\label{4.18}\end{align}
	 A triangle inequality for $u-u_h=(u-I_hu)+(e_h-Je_h)+Je_h$,
	 Theorem~\ref{2.10}.d for the first term, Theorem~\ref{J}.c and \eqref{eh} for the second term, and \eqref{4.18} for the last term  result  in
	 \begin{align}
	 |u-u_h|_{1,\pw}&\leq C_{\text{I}}h_{\text{max}}|u-\pid u|_{2,\pw}+2C_Jh_{\text{max}}|e_h|_{2,\pw}+|Je_h|_{1,\Omega}\nonumber\\&\leq (C_{\text{I}}+2C_JC_3+C_4C_{\text{reg}}) h^\sigma_{\text{max}}(|u-\pid u|_{2,\pw}+\mathrm{osc}_2(f,\T)).\label{h1}
	 \end{align}
	  \textit{Estimate of $|u-Gu_h|_{2,\pw}$}. 
	 Recall that $|\cdot|_s:=s_h(\cdot,\cdot)^{1/2}$ for $s_h(\cdot,\cdot)$ from Subsection~4.1  defines a seminorm in $V_h$ that is equivalent to $|(1-G)\cdot|_{2,\pw}$ owing to \eqref{stab}. It follows that \begin{align}
	 |u_h-Gu_h|_{2,P}\leq C_s^{1/2}S^P((1-G)u_h,(1-G)u_h)^{1/2}.\label{sb}
	 \end{align} 
	 This,  triangle inequalities, and $|(1-G)(u_h-I_hu)|_{2,\pw}\leq |u_h-I_hu|_{2,\pw}$ from Lemma~\ref{lem2.1} with Theorem~\ref{2.10}.b show
	 \begin{align}
	 |u_h|_{s}\leq |u_h-I_hu|_{s}+|I_hu|_{s}&\leq C_s^{1/2}
	 (|u_h-I_hu|_{2,\pw}+|(1-G)I_hu|_{2,\pw})\nonumber\\&\leq C_s^{1/2}(|u_h-I_hu|_{2,\pw}+|u-I_hu|_{2,\pw}+|u-\pid u|_{2,\pw}).\label{3.14}
	 \end{align}
	 The combination of \eqref{eh}-\eqref{h2} and \eqref{sb}-\eqref{3.14} with the triangle inequality $
	 |u-Gu_h|_{2,\pw}\leq |u-u_h|_{2,\pw}+|u_h-Gu_h|_{2,\pw}$  results in
	 \[|u-Gu_h|_{2,\pw}\leq (C_s+(1+C_s)(C_{\text{I}}+C_3))(|u-\pid u|_{2,\pw}+\mathrm{osc}_2(f,\T)).\]
	 \textit{Estimate of $|u-Gu_h|_{1,\pw}$}. A triangle inequality and  Lemma~\ref{lem2.1} provide
	 \begin{align*}
	 |u-Gu_h|_{1,\pw}\leq |u-u_h|_{1,\pw}+C_\PF h_{\text{max}}|u_h-Gu_h|_{2,\pw}.
	 \end{align*}
	 This and the combination of \eqref{h1}-\eqref{3.14}  conclude the proof.
\end{proof}
\section{A posteriori error analysis}
This section establishes a reliable and (up to data oscillations) efficient  explicit residual-based \textit{a posteriori} error estimator for a source term  $f\in L^2(\Omega)$ for both cases $Q=G$ and $Q=J$. Given any  polygon  $P\in\T$ and the discrete solution $u_h\in V_h$ to \eqref{dp}, define the computable terms
\begin{center}
	\begin{tabular}{lr}
		$\eta_P^2:=h_P^4\|f\|_{L^2(P)}^2$&(volume residual),\\
		$\zeta_P^2:=S^P((1-G)u_h,(1-G)u_h)$&(stabilization),\\
		$\Xi_{P}^2:=\sum_{E\in \e(P)}\Big(h_E^{-3}\|[Gu_h]_E\|_{L^2(E)}^2+h_E^{-1}\|[(Gu_h)_\bn]_E\|_{L^2(E)}^2\Big)$&(nonconformity),\\
		$\mu_P^2:=\eta_P^2+\zeta_P^2+\Xi_P^2$&(error estimator).
	\end{tabular}
\end{center}
Those local quantities $\bullet|_P$ form a family ($\bullet|_P:P\in\T$) over the index set $\T$ and their Euclid vector norms $\bullet|_\T$ enter the upper error bounds \begin{align*}
\eta_{\T}:=(\sum_{P\in\T}\eta_P^2)^{1/2}, \quad \zeta_\T:=(\sum_{P\in\T}\zeta_P^2)^{1/2},\quad  \Xi_\T:=(\sum_{P\in\T}\Xi_P^2)^{1/2},\quad\mu_\T:=(\sum_{P\in\T}\mu_P^2)^{1/2}.
\end{align*} 
\begin{theorem}[reliability]\label{rel}
	There exist  positive constants $C_{\mathrm{r1}}$ and $C_{\mathrm{r2}}$ (that exclusively depend on $\rho$), such that,  for $m=1,2$,
	\begin{align}
	C_{\mathrm{rm}}^{-2}(|u-u_h|_{m,\pw}^2+|u-Gu_h|_{m,\pw}^2)&\leq  \sum_{P\in\T}h_P^{2\sigma(2-m)}\mu_P^2. \label{5.21}
	\end{align}
\end{theorem}
\noindent The proof of Theorem~\ref{rel} uses an enrichment operator $E_h:\p_2(\tT)\to H^2_0(\Omega)$ from \cite{georgoulis2011posteriori}.

\begin{lemma}\label{en}
	There exists a positive constant $C_a$ (that exclusively depends on $\rho$) such that any $v_2\in\p_2(\tT)$ satisfies
	\begin{align}
	|v_2-E_hv_2|^2_{2,\pw}\leq C_a^{2}\sum_{E\in\e}\Big(h_E^{-3}\|[v_2]_E\|_{L^2(E)}^2+h_E^{-1}\|[(v_2)_{\bn}]_E\|_{L^2(E)}^2\Big).\label{P}
	\end{align}
\end{lemma}
\begin{proof}
	There exists a positive constant $C_a$ (that exclusively depends on $\rho$) such that any $v_2\in\p_2(\tT)$ and its enrichment $E_hv_2\in H^1_0(\Omega)$ from  \cite[Lemma~3.1]{georgoulis2011posteriori} satisfy
	\begin{align*}
	|v_2-E_hv_2|^2_{2,\tT}\leq C_a^{2}\sum_{E\in\eE}\Big(h_E^{-3}\|[v_2]_E\|_{L^2(E)}^2+h_E^{-1}\|[(v_2)_{\bn}]_E\|_{L^2(E)}^2\Big).
	\end{align*}
	The constant $C_a$ depends on the shape-regularity of the sub-triangulation $\tT$ from Subsection~2.1 and so depends on $\rho$. Since any edge $E\in \e$ is unrefined in the sub-triangulation $\tT$,  the above inequality reduces to (\ref{P}) for any $v_{2|P}\in H^2(P)$ and $P\in\T$. This concludes  the proof.
\end{proof}
\noindent The composition $E_h\circ G:V_h\to V$ connects a given function $v_h\in V_h$ to a conforming function,
\tikzstyle{line} = [draw, -latex']
\begin{figure}[H]
	\begin{center}
		\begin{tikzpicture}[node distance = 1.5cm, auto]
		\node(step1){$V_h$};
		\node[right of =step1,node distance=4cm](step2){$\p_2(\T)\hookrightarrow\p_2(\tT)$};
		\node[right of = step2,node distance=4cm](step3){$V$};
		\path[line](step2)--node[above,text centered]{$E_h$}(step3);
		\path[line](step1)--node[above,text centered]{$G$}(step2);
		\end{tikzpicture}
	\end{center}
\end{figure}
\vspace{-1cm}

\begin{proof}[Proof of  Theorem~\ref{rel} for $m=2$]
	Let $e:=u-E_h Gu_h \in V=H^2_0(\Omega)$. The scalar product $a(\cdot,\cdot)$  and the continuous problem \eqref{3} lead to
	$
	|u-E_hGu_h|^2_{2,\Omega}= a(u-E_hGu_h,e)=F(e)-a(E_hGu_h,e).
	$
	Recall $a_\pw(Gu_h,GI_he)=a_\pw(Gu_h,I_he)=a_\pw(Gu_h,e)$ from \eqref{pid1} and Theorem~\ref{2.10}.a. This and the discrete problem \eqref{dp} result in
	\begin{align}
	|u-E_hGu_h|^2_{2,\Omega}=F(e)-F(QI_he)+a_\pw(Gu_h-E_hGu_h,e)+s_h(u_h,I_he).\label{3.12}
	\end{align}
	\textit{Estimate of $F(e)-F(QI_he)$}. \textit{Case 1 ($Q=G$)}. Theorem~\ref{2.10}.b and Lemma~\ref{lem2.1} show   \[\|e-GI_he\|_{L^2(P)}\leq C_{\PF}h_P^2|e-Ge|_{2,P}\leq C_\PF h_P^2|e|_{2,P}.\] 
	\textit{Case 2 ($Q=J$)}.  A triangle inequality and Theorem~\ref{J}.c lead to  \[\|e-JI_he\|_{L^2(P)}\leq \|e-I_he\|_{L^2(P)}+C_Jh_P^2(|I_he-GI_he|_{2,P}+|I_he|_{2,P}). \] Theorem~\ref{2.10}.c and Lemma~\ref{lem2.1} show $|I_he-GI_he|_{2,P}+|I_he|_{2,P}\leq 2C_{\text{Ib}}|e|_{2,P}$. The previous  estimates and Theorem~\ref{2.10}.d result in $\|e-JI_he\|_{L^2(P)}\leq (C_{\text{I}}+2C_{J}C_{\text{Ib}})h_P^2|e|_{2,P}.$  A Cauchy-Schwarz inequality in the right-hand side of $F(e)-F(QI_he)=(f,e-QI_he)_{L^2(\Omega)}$ and the above estimates in case $Q=G$ or $Q=J$  provide an estimate for  each $P\in\T$. Their sum reads
	\begin{align}	F(e)-F(QI_he)\leq (C_{\PF}+C_{\text{I}}+2C_{J}C_{\text{Ib}})\eta_\T|e|_{2,\Omega}.
	\end{align} 
	\textit{Estimate of $a_\pw(Gu_h-E_hGu_h,e)$}. A Cauchy-Schwarz inequality and Lemma~\ref{en} lead to
    \begin{align}
	a_\pw(Gu_h-E_hGu_h,e)\leq C_{a}\Xi_\T|e|_{2,\Omega}.
	\end{align}
	 \textit{Estimate of $s_h(u_h,I_he)$}. A Cauchy-Schwarz inequality and \eqref{stab}  result in
	\begin{align}
	s_h(u_h,I_he)\leq C_s^{1/2}s_h^{1/2}(u_h,u_h)|(1-G)I_he|_{2,\pw}\leq C_s^{1/2}C_{\text{Ib}}s_h^{1/2}(u_h,u_h)|e|_{2,\Omega}\label{3.18s}
	\end{align}
	with $|(1-G)I_he|_{2,\pw}\leq |I_he|_{2,\pw}\leq C_{\text{Ib}}|e|_{2,\Omega}$ from Lemma~\ref{lem2.1} and Theorem~\ref{2.10}.c in the last step.\\
	\textit{Estimate of $|u-Gu_h|_{2,\pw}$}. The combination of \eqref{3.12}-\eqref{3.18s} shows  $C_5:=C_{\text{Ip}}+C_{\text{I}}+2C_{J}C_{\text{Ib}}+C_a+C_s^{1/2}C_{\text{Ib}}$ in
	$
	|u-E_hGu_h|_{2,\Omega}\leq C_5 \mu_\T.
	$
	This, Lemma~\ref{en}, and a triangle inequality for the error $u-Gu_h=(u-E_hGu_h)+(E_hGu_h-Gu_h)$   reveal
	\begin{align}|u-Gu_h|_{2,\pw}\leq (C_5+C_a)\mu_\T.\label{4.8a}\end{align}
	\textit{Estimate of $|u-u_h|_{2,\pw}$}. Recall  \eqref{sb} in the form $C_s^{1/2}|u_h-Gu_h|_{2,\pw}\leq \zeta_\T.$ This, \eqref{4.8a}, and a triangle inequality lead to \[|u-u_h|_{2,\pw}\leq |u-Gu_h|_{2,\pw}+|u_h-Gu_h|_{2,\pw}\leq (C_5+C_a+C_s^{1/2})\mu_\T.\] This and \eqref{4.8a} verify \eqref{5.21} for $m=2$ with $C_{\mathrm{r2}}:=\sqrt{3}(2(C_5+C_a)+C_s^{1/2})$.
\end{proof}
\begin{proof}[Proof of  Theorem~\ref{rel} for $m=1$]
 Let $\Psi\in V=H^2_0(\Omega)$ solve the dual problem $a(v,\Psi) = (\Delta(u-J'I_\M u_h),v)_{L^2(\Omega)}$ for all $v\in V$. Elliptic regularity \eqref{2.14} provides $\Psi\in H^{2+\sigma}(\Omega)$ and the estimate 
	\begin{align}
	\|\Psi\|_{2+\sigma,\Omega} \leq C_{\text{reg}}|u-J'I_\M u_h|_{1,\Omega}.\label{3.18a}
	\end{align}
	The test function $v=u-J'I_\M u_h\in V$ in the dual problem shows
	\begin{align}
	|u-J'I_\M u_h|_{1,\Omega}^2 =a(u-J'I_\M u_h,\Psi)=F(\Psi)-F(QI_h\Psi)+a_h(u_h,I_h\Psi)-a(J'I_\M u_h,\Psi)\label{6.11}
	\end{align}
	with the continuous problem \eqref{3} and the discrete problem \eqref{dp} in the last step. Theorem~\ref{2.10}.b provides $GI_h\Psi=G\Psi$ and  \eqref{pid1} shows $a_\pw(Gu_h,GI_h\Psi)=a_\pw(Gu_h,\Psi)$. Notice that $a_\pw(Gu_h-J'I_\M u_h,G\Psi)=a_\pw(Gu_h-u_h,G\Psi)+a_\pw(u_h-I_\M u_h,G\Psi)+a_\pw(I_\M u_h-J'I_\M u_h,G\Psi)=0 $ follows from \eqref{pid1}, \eqref{m1}, and Lemma~\ref{J'}.c. This and \eqref{6.11} result in
	\begin{align}
	|u-J'I_\M u_h|_{1,\Omega}^2&=
	F(\Psi)-F(QI_h\Psi)+a_\pw(Gu_h-J'I_\M u_h,\Psi-G\Psi)+s_h(u_h,I_h\Psi).\label{6.13}
	\end{align}
	\textit{Estimate of $F(\Psi)-F(QI_h\Psi)$}. \textit{Case1 ($Q=G$)}. Theorem~\ref{2.10}.b and Lemma~\ref{lem2.1}  show  that \[\|\Psi-GI_h\Psi\|_{L^2(P)}\leq C_{\PF}C_{\text{apx}}h_P^{2+\sigma}|\Psi|_{2+\sigma,P}.\]
	\noindent\textit{Case 2 ($Q=J$)}. A triangle inequality and Theorem~\ref{J}.c (with $v=\Psi$) reveal  that \begin{align*}\|\Psi-JI_h\Psi\|_{L^2(P)}&\leq \|\Psi-I_h\Psi\|_{L^2(P)}+C_Jh_P^2(|I_h\Psi-GI_h\Psi|_{2,P}+|I_h\Psi-\Psi|_{2,P})\\&\leq (C_{\text{I}}+C_J(2C_{\text{I}}+C_{\text{apx}}))h_{P}^{2+\sigma}|\Psi|_{2+\sigma,P}
	\end{align*} with Theorem~\ref{2.10}.d and  \eqref{5.20} in the last step. The two estimates for $\|\Psi-QI_h\Psi\|_{L^2(P)}$ (for $Q=G$ and $Q=J$)  and the Cauchy-Schwarz inequality $\int_P f(\Psi-QI_h\Psi)\,dx\leq \|f\|_{L^2(P)}\|\Psi-QI_h\Psi\|_{L^2(P)}$  prove an estimate for each $P\in\T$. The sum of all those provides
	\begin{align}F(\Psi)-F(QI_h\Psi)\leq (C_{\PF}C_{\text{apx}}+C_{\text{I}}+C_J(2C_{\text{I}}+C_{\text{apx}})) |\Psi|_{2+\sigma,\Omega}\sum_{P\in\T}h_P^{\sigma}\eta_P\label{3.19}.\end{align}
	\textit{Estimate of $a_\pw(Gu_h-J'I_\M u_h,\Psi-G\Psi)$}.
	A triangle inequality and   Lemma~\ref{J'}.d (with $v=E_hGu_h$) imply $|Gu_h-J'I_\M u_h|_{2,\pw}\leq |Gu_h-I_\M u_h|_{2,\pw}+C_{J'} |I_\M u_h-E_hGu_h|_{2,\pw}$. This and $|Gu_h-I_\M u_h|_{2,\pw}\leq |u_h-Gu_h|_{2,\pw}$ from \eqref{m1} lead to
	\begin{align}
	|Gu_h-J'I_\M u_h|_{2,\pw}&\leq(1+C_{J'})|u_h-Gu_h|_{2,\pw}+ C_{J'}|Gu_h-E_hGu_h|_{2,\pw}\nonumber\\&\leq(1+C_{J'})C_s^{1/2}\zeta_\T +C_{J'}C_a\Xi_\T\label{3.19b}
	\end{align}
	with  \eqref{sb} and Lemma~\ref{en} in the last step. A Cauchy-Schwarz inequality, \eqref{3.19b}, and Lemma~\ref{lem2.1} prove with $C_6:=C_{\text{apx}}((1+C_{J'})C_s^{1/2}+C_{J'}C_a)$ that
	\begin{align}
	a_\pw(Gu_h-J'I_\M u_h,\Psi-G\Psi)\leq C_6|\Psi|_{2+\sigma,\Omega}\sum_{P\in\T}h_P^\sigma(\zeta_P+\Xi_P).\label{4.13}
	\end{align}
	\textit{Estimate of $s_h(u_h,I_h\Psi)$}. Argue as in \eqref{3.18s} for the stability term and proceed with $|(1-G)I_h\Psi|_{2,\pw}\leq |I_h\Psi-\pid\Psi|_{2,\pw}$ from Theorem~\ref{2.10}.b to deduce that
	\begin{align}
	s_h(u_h,I_h\Psi)\leq C_s^{1/2} \zeta_\T |I_h\Psi-\pid\Psi|_{2,\pw}\leq C_s^{1/2}(C_{\text{I}}+C_{\text{apx}})|\Psi|_{2+\sigma,\Omega}\sum_{P\in\T}h_P^\sigma\zeta_P\label{3.19c}
	\end{align}
	with a triangle inequality, Theorem~\ref{2.10}.d, and Lemma~\ref{lem2.1} in the last step.\\
	\textit{Estimate of $|u-Gu_h|_{1,\pw}$}. The combination of \eqref{6.13}-\eqref{3.19} and \eqref{4.13}-\eqref{3.19c}  provides 
	\begin{align}
	|u-J'I_\M u_h|_{1,\Omega}^2\leq C_7|\Psi|_{2+\sigma,\Omega}\sum_{P\in\T}h_P^{\sigma}\mu_P\label{3.22}
	\end{align}
	with  $C_7:=C_{\PF}C_{\text{apx}}+C_{\text{I}}+C_J(2C_{\text{I}}+C_{\text{apx}})+C_6+C_s^{1/2}(C_{\text{I}}+C_{\text{apx}})$. Argue as in \eqref{jv} to prove $\int_{\partial P}\nabla (J'I_\M u_h-u_h)\,ds=0$ so that the Poincar\'e-Friedrichs inequality from Theorem~\ref{PF}.a applies. This and  the triangle inequality  $|u-Gu_h|_{1,\pw}\leq |u-J'I_\M u_h|_{1,\Omega}+|J'I_\M u_h-Gu_h|_{1,\pw}$ result in
	\[
	|u-Gu_h|_{1,\pw}\leq |u-J'I_\M u_h|_{1,\Omega}+C_\PF|h_\T(J'I_\M u_h-Gu_h)|_{2,\pw}.
	\]
	  The regularity estimate \eqref{3.18a},  \eqref{3.22}, and \eqref{3.19b}  lead  in the previous displayed estimate to
	\begin{align*}
	|u-Gu_h|_{1,\pw}\leq C_{\text{reg}}(C_7+C_\PF((1+C_{J'})C_s^{1/2}+C_{J'}C_a))\sum_{P\in\T}h_P^\sigma\mu_P.
	\end{align*}
	\textit{Estimate of $|u-u_h|_{1,\pw}$}. The triangle inequality $|u-u_h|_{1,\pw}\leq |u-Gu_h|_{1,\pw}+|Gu_h-u_h|_{1,\pw}$, Lemma~\ref{lem2.1}, and \eqref{sb}  conclude the proof of \eqref{5.21} for $m=1$  with $C_{\mathrm{r1}}:=\sqrt{3}(2C_{\text{reg}}(C_7+C_\PF((1+C_{J'})C_s^{1/2}+C_{J'}C_a)+C_s^{1/2})$.
\end{proof}
Let $z\in{\mathcal{V}}$ be a vertex in $\T$ with the neighbouring polygons $\T(z):=\{P'\in\T: z\in P'\}$ and define the vertex patch $\omega_z := \text{int}(\cup\T(z))$  and the larger neighbourhood $\Omega(P):= \cup_{z\in{\mathcal{V}}(P)}\omega_z$. The edge patch $\T(E):=\{P'\in\T:E\subset\partial P'\}$  consists of  one or  two neighbouring polygons  that share an edge $E\in\e$ and this defines $\omega(E):=\text{int}(\cup\T(E))$.
\begin{theorem}[local efficiency up to data  oscillations]\label{efficiency}
	For any $P\in\T$ it holds
	\begin{align}
	\zeta_P^2&\lesssim |u-u_h|^2_{2,P}+|u-Gu_h|^2_{2,P},\label{lb}\\
	\eta_P^2&\lesssim |u-Gu_h|^2_{2,P}+\mathrm{osc}_2^2(f,P),\label{lb1}\\
	\Xi_P^2&\lesssim \sum_{E\in\e(P)}\sum_{P'\in\Omega(\omega(E))}( |u-u_h|^2_{2,P'}+|u-Gu_h|^2_{2,P'}).\label{lb2}
	\end{align}
\end{theorem}
\begin{proof}[Proof of \eqref{lb}] The upper bound in \eqref{stab}  and a triangle inequality lead to 
	\[
	\zeta_P^2\leq C_s|(1-G)u_h|^2_{2,P}\leq 2C_s(|u-u_h|^2_{2,P}+|u-Gu_h|^2_{2,P}). \qedhere
	\]
\end{proof}	
\begin{proof}[Proof of \eqref{lb1}] Abbreviate  $\theta_P:=(f-\Pi_2f)|_P$, recall the bubble-function $b_P$  from \eqref{bubble}, and substitute $v=b_P\Pi_2f\in V$ in \eqref{3} to obtain
	\begin{align*}
	(\Pi_2f,v)_{L^2(P)}=a^P(u,v)-(\theta_P,v)_{L^2(P)}=a^P(u-Gu_h,v)-(\theta_P,v)_{L^2(P)}
	\end{align*}
	with $a^P(Gu_h,v)=a^P(Gu_h,Gv)=0$ from \eqref{pid1} and $Gv=0$ (from Lemma~\ref{lem2.2} for $\text{Dof}(v)=0$)  in the last step. This and  Cauchy-Schwarz inequalities  show
	\begin{align}
	(\Pi_2f,v)_{L^2(P)}\leq |u-Gu_h|_{2,P}|v|_{2,P}+\|\theta_P\|_{L^2(P)}\|v\|_{L^2(P)}.\label{3.25}
	\end{align}
	The first inequality in \eqref{b1} shows $C_b^{-1}\|\Pi_2f\|_{L^2(P)}^2\leq (\Pi_2f,v)_{L^2(P)}$ and the second inequality in \eqref{b2} verifies $\sum_{m=0}^2h_P^m|v|_{2,P}\leq C_b\|\Pi_2f\|_{L^2(P)}$. Those estimates  prove in \eqref{3.25} that
	\begin{align*}
	C_b^{-2}h_P^2\|\Pi_2f\|_{L^2(P)}\leq |u-Gu_h|_{2,P}+\mathrm{osc}_2(f,P).
	\end{align*}
	This and the triangle inequality  $\|\eta_P^2\|_{L^2(P)}\leq\|h_P^2(f-\Pi_2f)\|_{L^2(P)}+\|h_P^2\Pi_2f\|_{L^2(P)}$ conclude the proof of \eqref{lb1}.
\end{proof}
\begin{proof}[Proof of \eqref{lb2}] Since $[Gu_h]_E = [Gu_h-J'I_\M u_h]_E$, the trace inequality   leads to
	\begin{align*}
	C_T^{-1}\|[Gu_h]_E\|_{L^2(E)}\leq h_E^{-1/2}\|Gu_h-J'I_\M u_h\|_{L^2(\omega(E))}+h_E^{1/2}\|\nabla_\pw(Gu_h-J'I_\M  u_h)\|_{L^2(\omega(E))}.
	\end{align*}
	Rewrite $Gu_h-J'I_\M u_h= (Gu_h-u_h)+(u_h-I_\M u_h)+(I_\M u_h-J'I_\M u_h)$. Abbreviate $|\cdot|_{2,\tT(P')}:=\sum_{T\in\tT(P')}|\cdot|_{2,T}$ for  $P'\in\T$.  Lemma~\ref{lem2.1}, \eqref{2.20}, and Lemma~\ref{J'}.d for $P'\in\T(E)$  lead to
	\begin{align}
	C_T^{-1}h_E^{-3/2}\|[Gu_h]_E\|_{L^2(E)}\leq (2+C_\PF) \sum_{P'\in\T(E)}(|u_h-Gu_h|_{2,P'} + |(1- J')I_\M u_h|_{2,\tT(P')}).\label{3.28}
	\end{align}
	 There exists  a local version of Lemma~\ref{J'}.d established in \cite[Lemma~5.1]{carstensen2018prove} with a positive constant $C_8$  (that exclusively depends on the shape regularity of $\tT$) such that
	\begin{align*}
	|(1-J')I_\M u_h|_{2,\tT(P')}\leq C_8\min_{v\in V}\|D^2_\pw(I_\M u_h-v)\|_{L^2(\Omega(P'))}\leq C_8 \|D^2_\pw(I_\M u_h-u)\|_{L^2(\Omega(P'))}.	
	\end{align*}
	A triangle inequality shows \[|(1-J')I_\M u_h|_{2,\tT(P')}\leq C_8(\|D^2_\pw(I_\M u_h-u_h)\|_{L^2(\Omega(P'))}+\|D^2_\pw(u_h-u)\|_{L^2(\Omega(P'))}).\]  This estimate for each $P'\in\T(E)$, $|u_h-I_\M u_h|_{2,\pw}\leq |u_h-Gu_h|_{2,\pw}$ from\eqref{m1},  \eqref{3.28}, and a triangle inequality imply
	\begin{align}
	h_E^{-3/2}\|[Gu_h]_E\|_{L^2(E)}\leq C_T(2+C_\PF)(1+C_8) \sum_{P'\in \Omega(\T(E))}(|u-u_h|_{2,P'}+|u-Gu_h|_{2,P'}).\label{p1}
	\end{align}
	It remains to control the term $h_E^{-1/2}\|[(Gu_h)_{\bn}]_E\|_{L^2(E)}$ for each $E\in\e(P)$.  Since  $u-u_h\in V_\nc$, $\alpha_E:=\dashint_E (u-u_h)_{\bn}\,ds\in\mathbb{R}$ is uniquely defined. Rewrite $[(Gu_h)_{\bn}]_E = [(Gu_h-u)_{\bn}+\alpha_E]_E$ for $E\in\e(P)$. The  triangle inequality $\|[(Gu_h-u)_{\bn}+\alpha_E]_E\|_{L^2(E)}\leq\|[(Gu_h-u_h)_\bn]_E\|_{L^2(E)}+\|[(u_h-u)_{\bn}+\alpha_E]_E\|_{L^2(E)}$ and the trace inequality  lead to
	\begin{align}
	h_E^{-1/2}\|[(Gu_h)_{\bn}]_E\|_{L^2(E)}&\lesssim (h_E^{-1}\|\nabla_\pw(Gu_h-u_h)\|_{L^2(\omega(E))}+\|D^2_\pw(Gu_h-u_h)\|_{L^2(\omega(E))})\nonumber\\&\quad+(h_E^{-1}\|(u_h-u)_{\bn}+\alpha_E\|_{L^2(\omega(E))}+\|D^2_\pw(u-u_h)\|_{L^2(\omega(E))}).\label{3.17}
	\end{align}
	    Since $\int_E((u_h-u)_{\bn}+\alpha_E)\,ds=0$, the Poincar\'e-Friedrichs inequality in Theorem~\ref{PF}.a applies to $\bn_E\cdot\nabla(u_h-u)+\alpha_E=f$ in  each $P'\in\T(E)$ and asserts $\|(u_h-u)_{\bn}+\alpha_E\|_{L^2(\omega(E))}\leq C_\PF h_P\|D^2_\pw(u_h-u)\|_{L^2(\omega(E))}$. This, $h_E\geq \rho h_{P'}$ from (M2), and Lemma~\ref{lem2.1} show 
	    \begin{align*}
	   h_E^{-1/2}\|[(Gu_h)_{\bn}]_E\|_{L^2(E)}\lesssim \|D^2_\pw(Gu_h-u_h)\|_{L^2(\omega(E))}+\|D^2_\pw(u-u_h)\|_{L^2(\omega(E))}.
	    \end{align*}
	     This and a triangle inequality result in
	\[
	h_E^{-1/2}\|[(Gu_h)_{\bn}]_E\|_{L^2(E)}\lesssim \sum_{P'\in\T(E)}(|u-u_h|_{2,P'}+|u-Gu_h|_{2,P'}).
	\qedhere
	\]
\end{proof}
\begin{remark}[efficiency of $H^1$ error control]
	The upper bounds in \eqref{lb}-\eqref{lb2} with a multiplication factor $h_P^{2\sigma}$ for $0<\sigma\leq 1$ in front of  $\mu_P^2$  show that the error estimators in \eqref{5.21} for $m=1$ converge (at least) with the expected convergence rate of the piecewise $H^1$ error. 
\end{remark}
%
\begin{remark}[higher-order ncVEM and 3D]
	The upper bound \eqref{5.21} for the error $|u-u_h|_{2,\pw}$  can be generalized to  ncVEM  of higher order $r\geq 3$ (see \cite{antonietti2018fully} for the discrete setting). The enrichment operator $E_h$ can be defined from the piecewise polynomial space $\p_r(\tT)$  to $H^2_0(\Omega)$ \cite{georgoulis2011posteriori} and the arguments  in this section hold in the three-dimensional case as well.
\end{remark}
\begin{remark}[conforming VEM]
	There are papers on the \textit{a priori} error estimates for the conforming VEM,  but there is no work on the \textit{a posteriori} VE analysis for  the biharmonic problem in the current literature. The analysis in Section~6 applies to the conforming case with $E_h\circ G =1$ and  $J'\circ I_\M =1$ in the proof of \eqref{5.21}  for $m=2$ and $m=1$. This establishes  the reliable and efficient \textit{a posteriori} error estimator $\eta_\T+\zeta_\T$ for the conforming VEM. 
\end{remark}
\begin{remark}[extensions]
	The source term $F$ is assumed to be an $L^2$ function in the main parts of this paper for simplicity and brevity. A class of more general sources $F\in H^{-2}(\Omega)$ is discussed in Theorem~\ref{err} only for the smoother $Q=J$ in the discrete problem and then avoids the data oscillations. More  examples on a class of right-hand sides $F$ are discussed in \cite{carstensen2021} and, in particular,  the \textit{a posteriori} error estimates can be generalized for this class of source terms as well. The arguments of \cite{carstensen2021} apply here as well and further details are omitted for brevity.   
\end{remark}
\section{Numerical results}
This section  discusses two numerical experiments with  uniform  and  adaptive mesh-refinement.   
\subsection{Adaptive algorithm}
A standard adaptive algorithm with the loop $\text{Solve}\; \to\;\text{Estimate}\;\to\;\text{Mark}\; \to\text{Refine}$
from \cite[Sec.~6.1]{carstensen2021priori}  is performed in two computational benchmarks.\\
\\
Step 1 (SOLVE).
Find the solution $u_h$ to \eqref{dp} for $Q=G$  in the right-hand side and compute the errors $H1e$ and $H2e$, $Hme:=|u-G u_h|_{m,\pw}$ for $m=1,2$, using polygauss quadrature rule \cite{product} for the input parameter $n=10$. \\
\\
Step 2 (ESTIMATE). Compute the local residuals  in Theorem~\ref{rel} and collect all these contributions for  $P\in\T$ to obtain the upper bound $H1\mu$ and $H2\mu$, $Hm\mu^2:= \sum_{P\in\T} h_P^{2\sigma(2-m)}\mu_P^2$ for the piecewise   $H^m$ error for $m=1,2$. Abbreviate the number of degrees of freedom by ndof. \\
\\
Step 3 (MARK). The D$\ddot{o}$rfler marking strategy \cite{16} detemines $\mathcal{D}_m\subset\T$ for $m=1,2$ with
\begin{align*}
Hm\mu^2\leq 0.5\sum_{P\in{\cal D}_m}h_P^{2\sigma(2-m)}\mu_P^2.
\end{align*}
\\
Step 4 (REFINE). A refinement strategy in VEM  divides the marked polygonal domains by connecting the mid-points of the edges to the centroid and allow  at most one hanging node per edge; cf. \cite{yu2021implementation} for further details  on a MATLAB implementation.

\subsection{Numerical example in L-shaped domain}
This subsection considers an L-shaped domain of Figure~7.1 with the exact solution of the model problem in polar co-ordinates $(r,\theta)$ 
\begin{align*}
u(r,\theta) = r^{5/3}\sin\Big(\frac{5\theta}{3}\Big)\quad\text{in}\;\Omega=(-1,1)^2\setminus [0,1)\times(-1,0].
\end{align*}
In this example, both $u$ and $u_{\bn}$ are not zero along the boundary $\partial\Omega$ and $f=0$. The upper bound for inhomogeneous boundary data can be established with minor modifications: The term $\Xi_P$ in the error estimator for $P\in\T$, which share a boundary edge, changes to
\begin{align*}
\Xi_P^2&=\sum_{E\in \e(P)\cap\e(\Omega)}\Big(\frac{1}{h_E^3}\|[Gu_h]_E\|_{L^2(E)}^2+\frac{1}{h_E}\|[(Gu_h)_\bn]_E\|_{L^2(E)}^2\Big)\\&\quad+\sum_{E\in \e(P)\cap\e(\partial\Omega)}\Big(\frac{1}{h_E^3}\|[Gu_h-u]_E\|_{L^2(E)}^2+\frac{1}{h_E}\|[(Gu_h-u)_\bn]_E\|_{L^2(E)}^2\Big).
\end{align*} Figure~7.1 displays strong local mesh-refinement at the re-entry corner in the adaptive mesh-refining. Figure~7.2 shows that uniform refinement yields the sub-optimal convergence rate, whereas adaptive refinements recover the optimal convergence rate.
\begin{figure}[H]
	\centering
	\begin{subfigure}{.33\textwidth}
		\centering
		\includegraphics[width=0.8\linewidth]{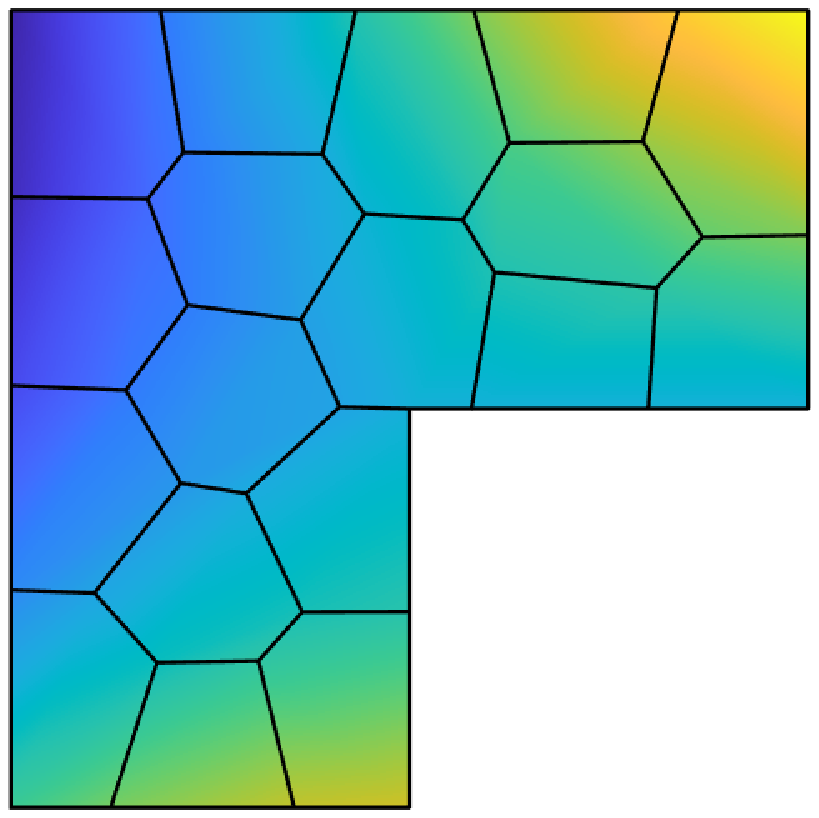}
	\end{subfigure}%
	\begin{subfigure}{.33\textwidth}
		\centering
		\includegraphics[width=0.8\linewidth]{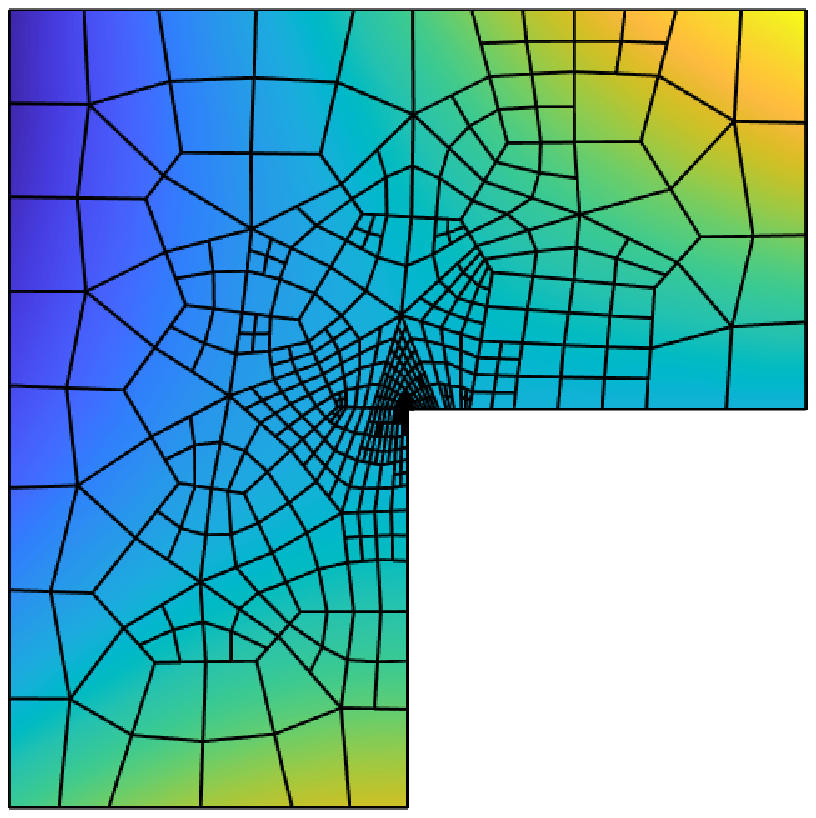}
	\end{subfigure}%
	\begin{subfigure}{.33\textwidth}
		\centering
		\includegraphics[width=0.8\linewidth]{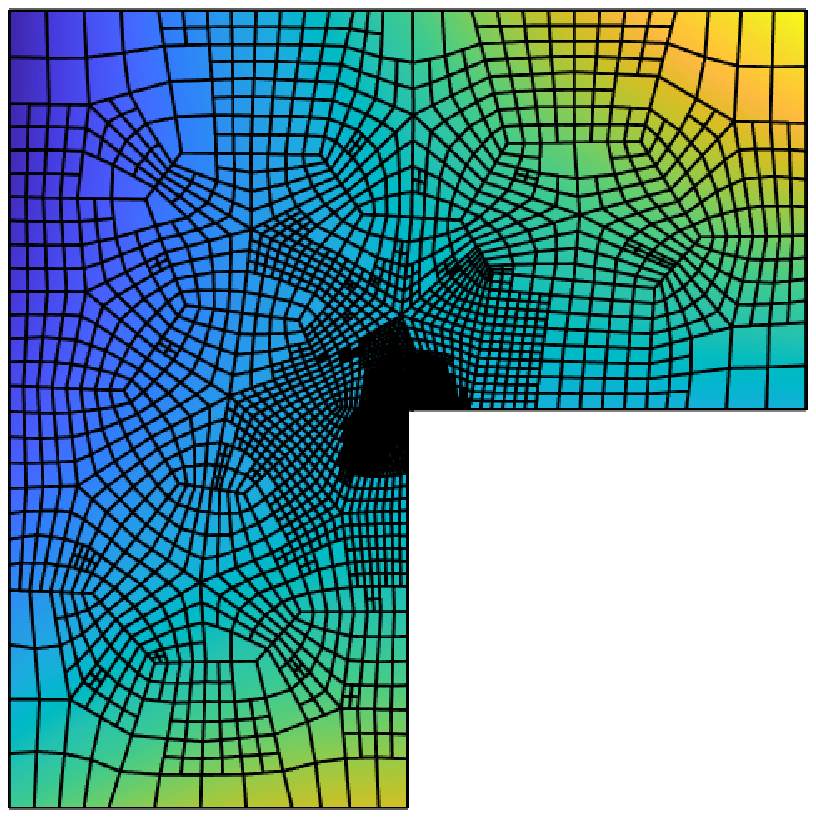}
	\end{subfigure}
	\caption{Output $\T_1, \T_{10}, \T_{15}$ of the adaptive algorithm in Subsection~7.2.}
	\label{fig4.2}
\end{figure}
\begin{figure}[H]
	\centering
	\begin{subfigure}{.5\textwidth}
		\centering
		\includegraphics[width=0.9\linewidth]{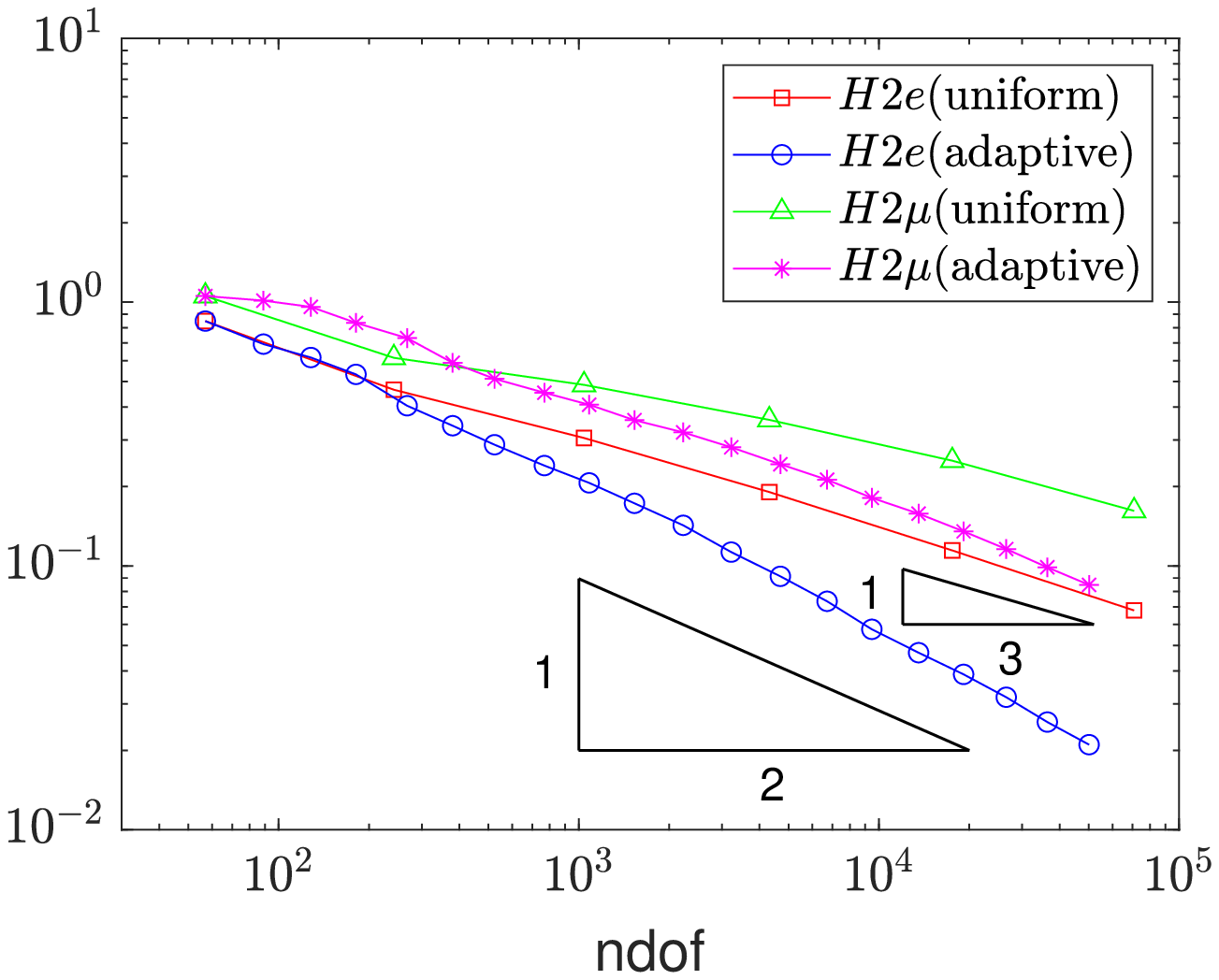}
	\end{subfigure}%
	\begin{subfigure}{.5\textwidth}
		\centering
		\includegraphics[width=0.9\linewidth]{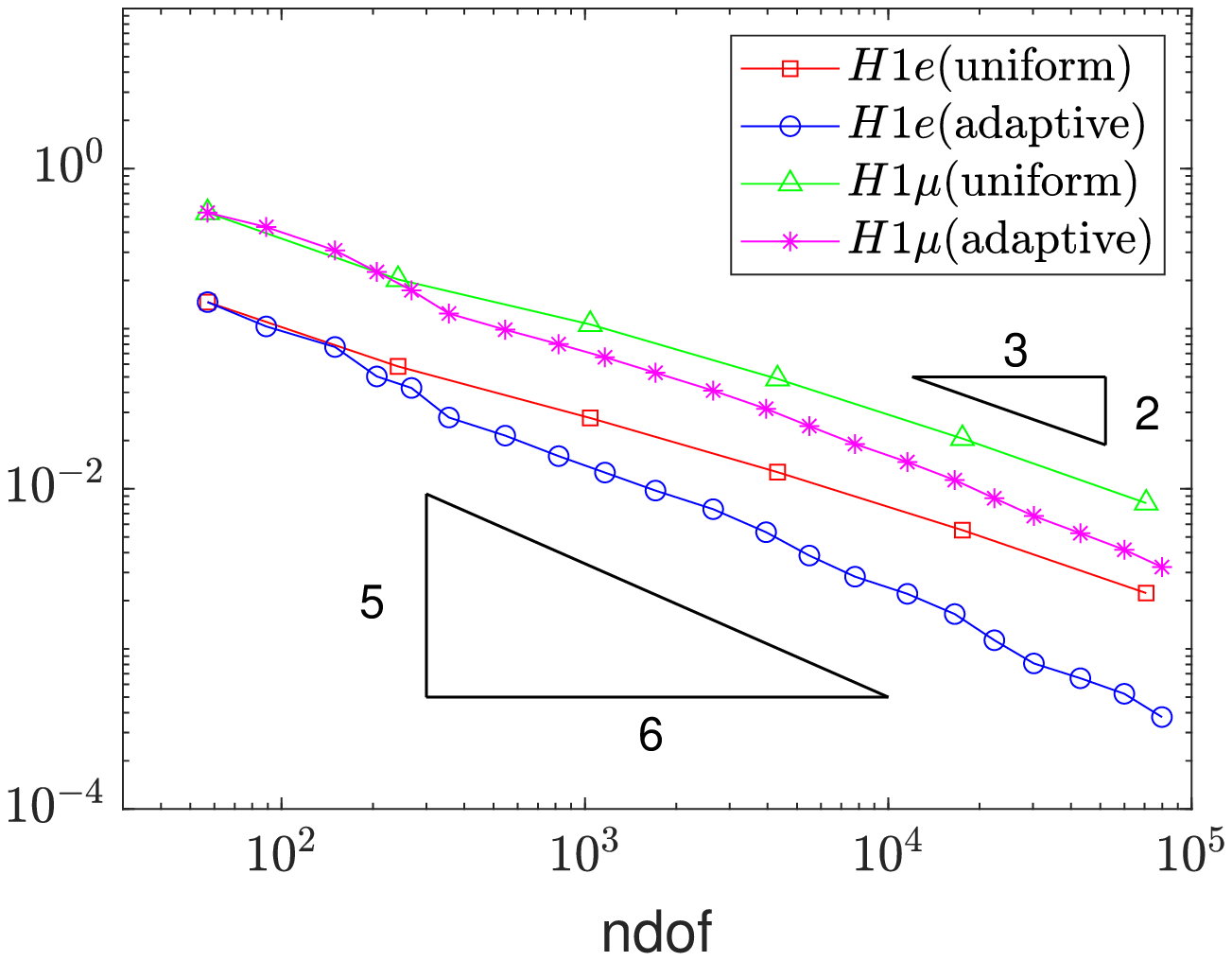}
	\end{subfigure}
	\caption{Convergence history plot of the errors resp. error estimators $H2e$ resp. $H2\mu$ (left) and $H1e$ resp. $H1\mu$ (right) vs  ndof for the L-shaped domain in Subsection~7.2.}
	\label{fig5.7}
\end{figure}
\subsection{Numerical example in Z-shaped domain}
The subsection considers the  polygonal domain $\Omega$ with the vertices $(0,0), (1,0), (1,1),(-1,1),$\\$(-1,-1),(1,-1)$  of Figure~7.3. Define the right-hand side function $f$  in the polar co-ordinates $(r,\theta)$ with the exact solution
\begin{align*}
u(r,\theta) = (1-r^2\cos^2(\theta))^2(1-r^2\sin^2(\theta))^2r^{(1+z)}g(\theta).
\end{align*}
Here $z = 0.505009698896589$ is a noncharacteristic root of $\sin^2(z\omega)=z^2\sin^2(\omega), \omega=7 \pi/4$ and $g(\theta)$ is as given in \cite[p.~107]{grisvard1992elliptic}. Figure~7.3 and 7.4  display the numerical results.

\begin{figure}[H]
	\centering
	\begin{subfigure}{.33\textwidth}
		\centering
		\includegraphics[width=0.73\linewidth]{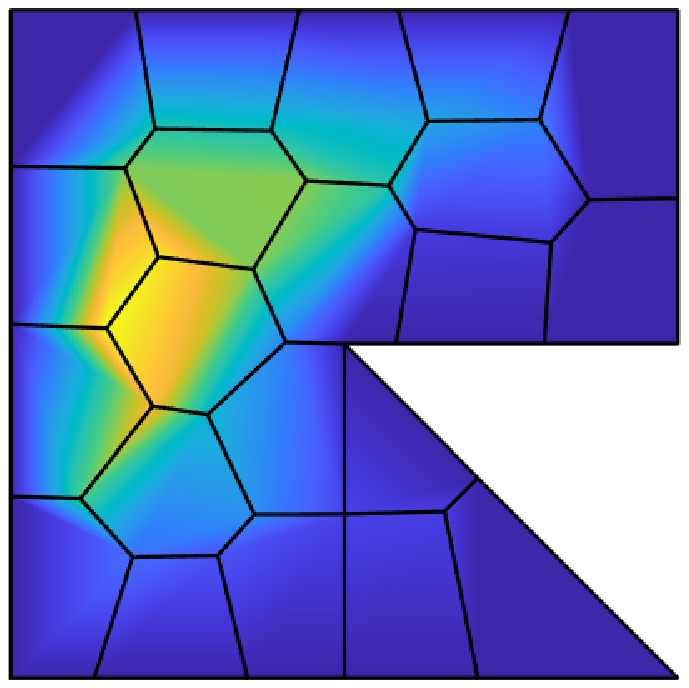}
	\end{subfigure}%
	\begin{subfigure}{.33\textwidth}
		\centering
		\includegraphics[width=0.8\linewidth]{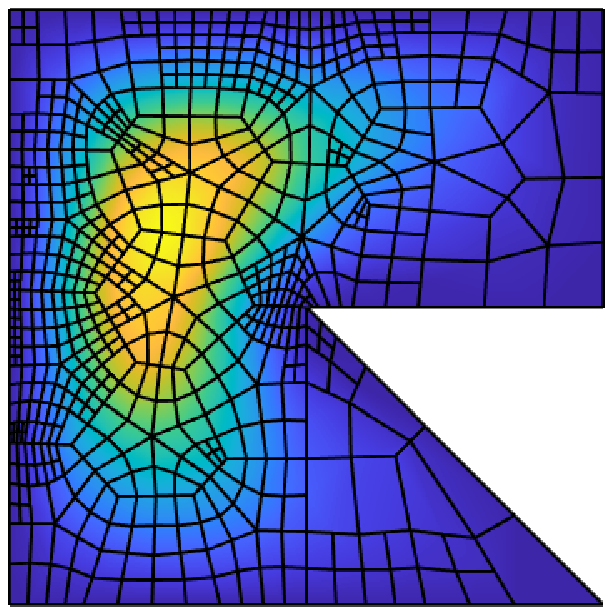}
	\end{subfigure}%
	\begin{subfigure}{.33\textwidth}
		\centering
		\includegraphics[width=0.8\linewidth]{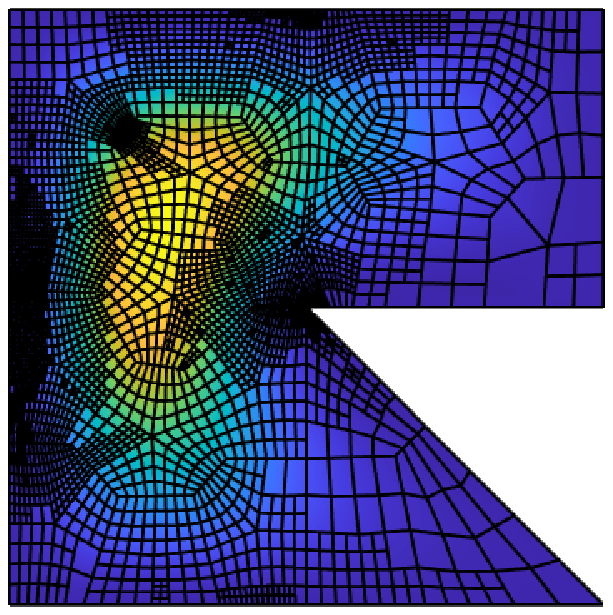}
	\end{subfigure}
	\caption{Output $\T_1, \T_{10}, \T_{15}$ of the adaptive algorithm in Subsection~7.3.}
	\label{fig4.3}
\end{figure}
\begin{figure}[H]
	\centering
	\begin{subfigure}{.5\textwidth}
		\centering
		\includegraphics[width=0.9\linewidth]{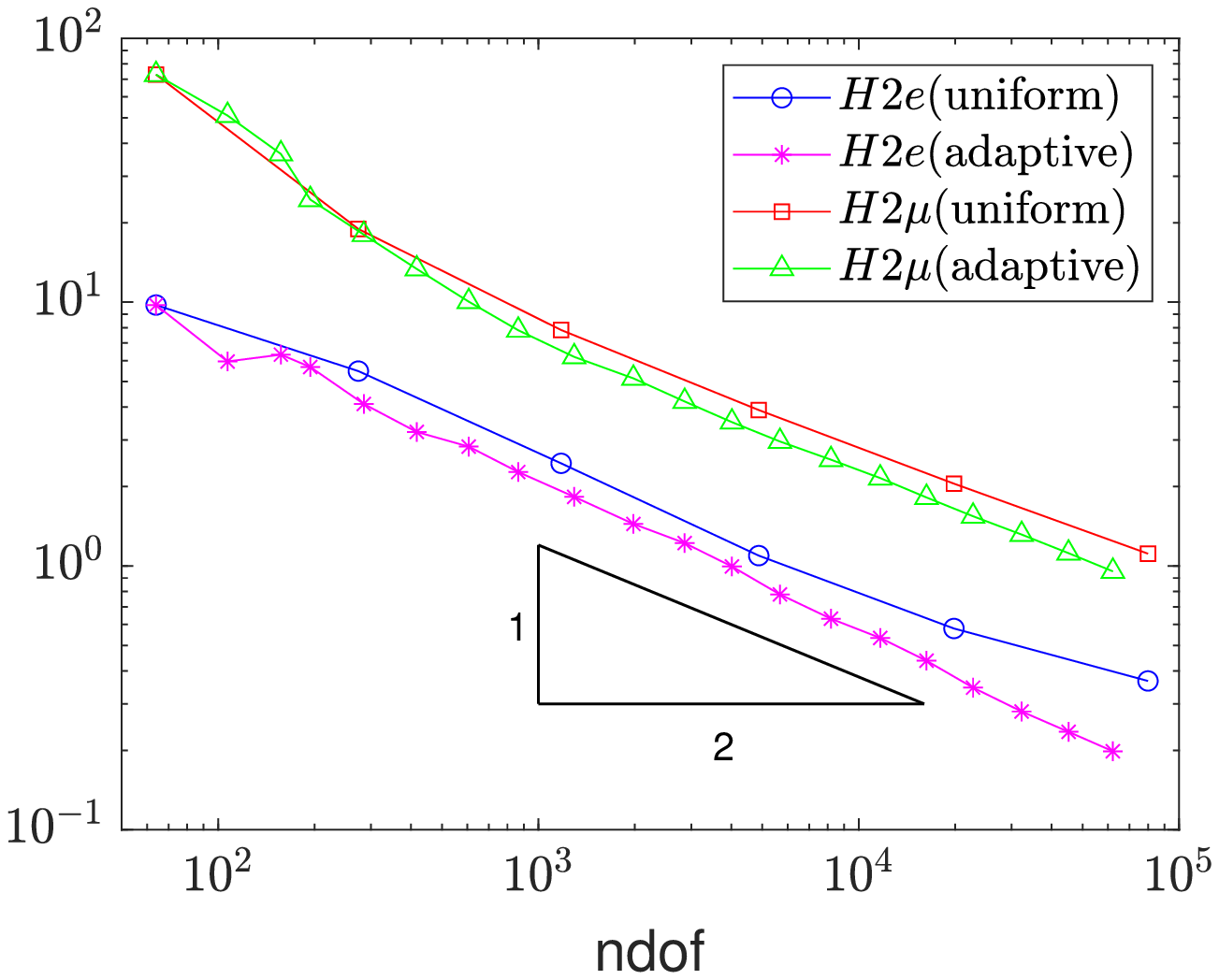}
	\end{subfigure}%
	\begin{subfigure}{.5\textwidth}
		\centering
		\includegraphics[width=0.9\linewidth]{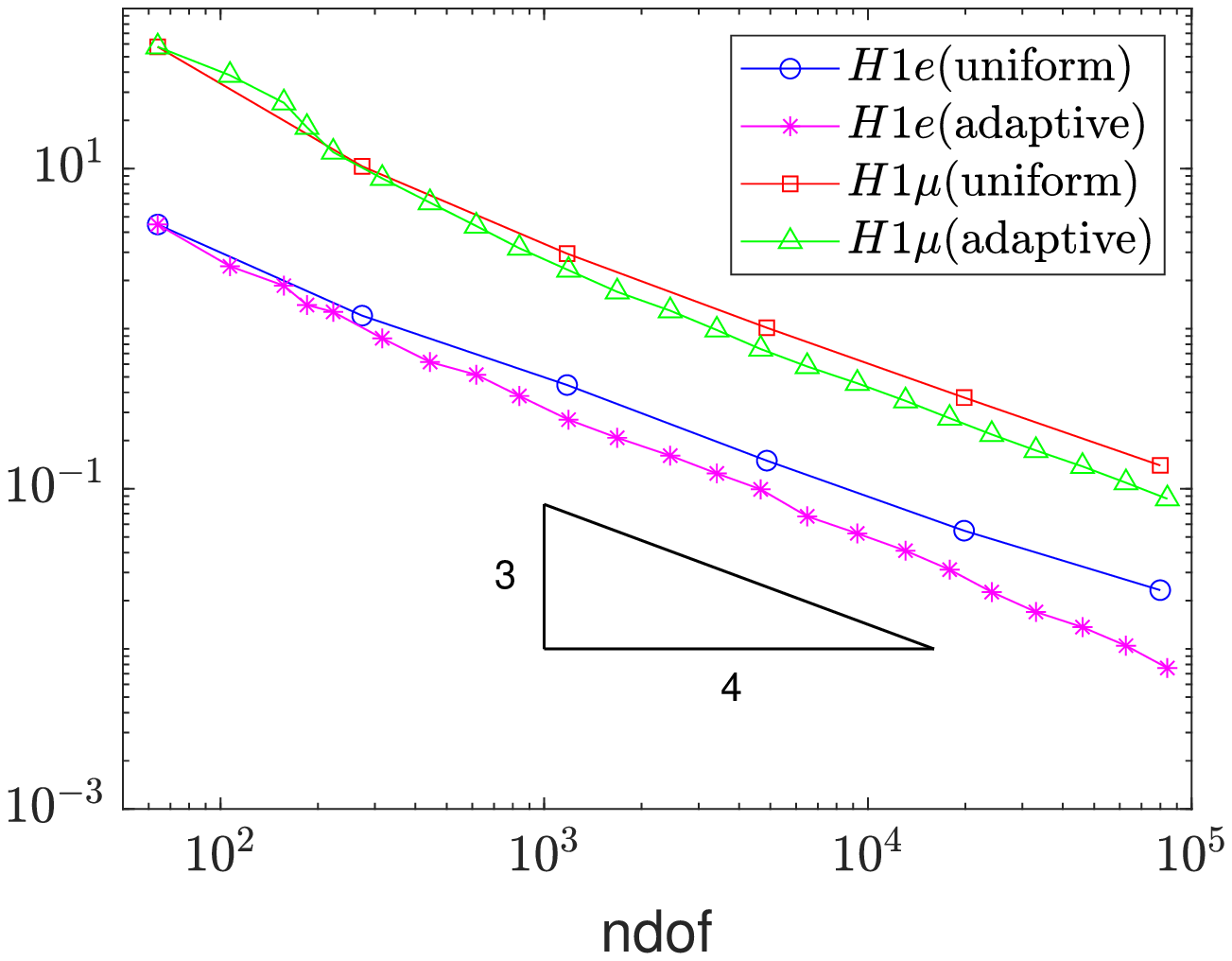}
	\end{subfigure}
	\caption{Convergence history plot of the errors resp. error estimators $H2e$ resp. $H2\mu$ (left) and $H1e$ resp. $H1\mu$ (right) vs  ndof for the Z-shaped domain in Subsection~7.3.}
	\label{fig4.6}
\end{figure}
\subsection{Evaluation}
\textit{Empirical convergence rates}. 
 The two domains from Subsection 7.2 resp. 7.3 have weak solutions $u\in H^{2+\sigma-\epsilon}(\Omega)$ for any $\epsilon>0$ with the typical corner singularity  for $\sigma=2/3$ resp. $\sigma=z=0.505$ and hence we expect and observe the empirical convergence rates in the $H^2$ norm (resp. $H^1$ norm)  $\sigma/2$ $(\text{resp.}\; \sigma)$ in terms of $\text{ndof}^{-1/2}$  for uniform mesh-refinements. The adaptive mesh-refining improves the empirical convergence rates to the optimal values $1/2$ (resp. $(1+\sigma)/2$).

\smallskip
\noindent\textit{Efficiency indices}. The ratio of the error estimator and the total error (effectivity index) in the piecewise $H^2$ and $H^1$ norm  remains bounded:  $3\leq H2\mu/H2e\leq 5$  and  $5.5\leq H1\mu/H1e\leq 7$ in both examples. This confirms empirically that the error  estimator mimics the behaviour of the total error and also validates  Theorem~\ref{rel}.

\smallskip
\noindent\textit{Dominant error contributions}. Figure~7.5 displays the individual components $\eta_\ell,\zeta_\ell,\Xi^1_\ell,\Xi^2_\ell$ in the error estimator, which abbreviate  $\eta_{\T_\ell},\zeta_{\T_\ell},\Xi^1_{\T_\ell},\Xi^2_{\T_\ell}$ and shows the dominance of $\Xi_\ell^1$. 
The remaining part $\Xi_\ell^2$ of $\Xi_\ell:=\Xi_\ell^1+\Xi_\ell^2$ and the volume residual $\eta_\ell$ converge more rapidly. The error estimator  components for the $H^1$ error in the adaptive refinement behave similar.
\begin{figure}[H]
	\begin{subfigure}{.5\textwidth}
		\centering
		\includegraphics[width=0.9\linewidth]{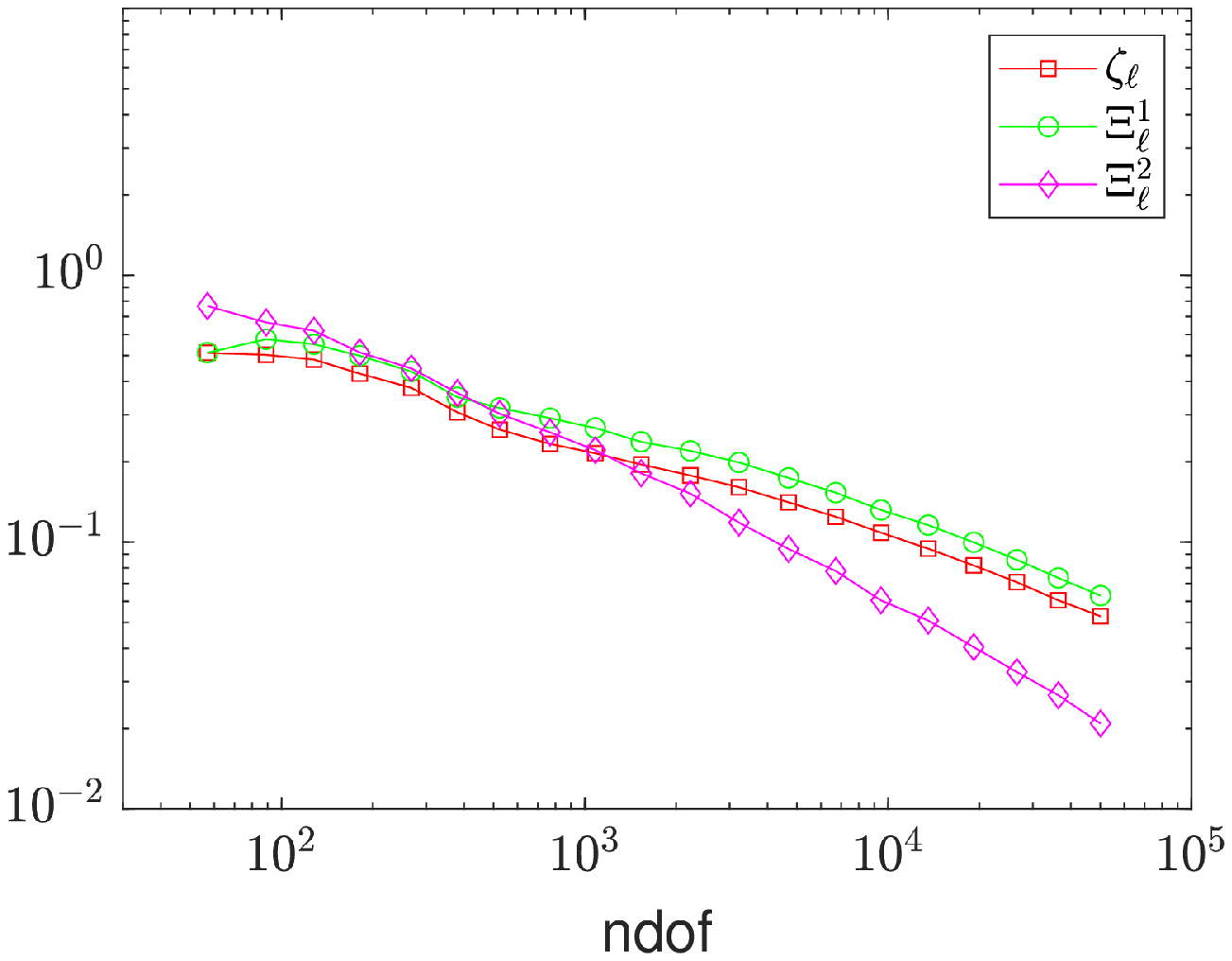}
	\end{subfigure}%
	\begin{subfigure}{.5\textwidth}
		\centering
		\includegraphics[width=0.9\linewidth]{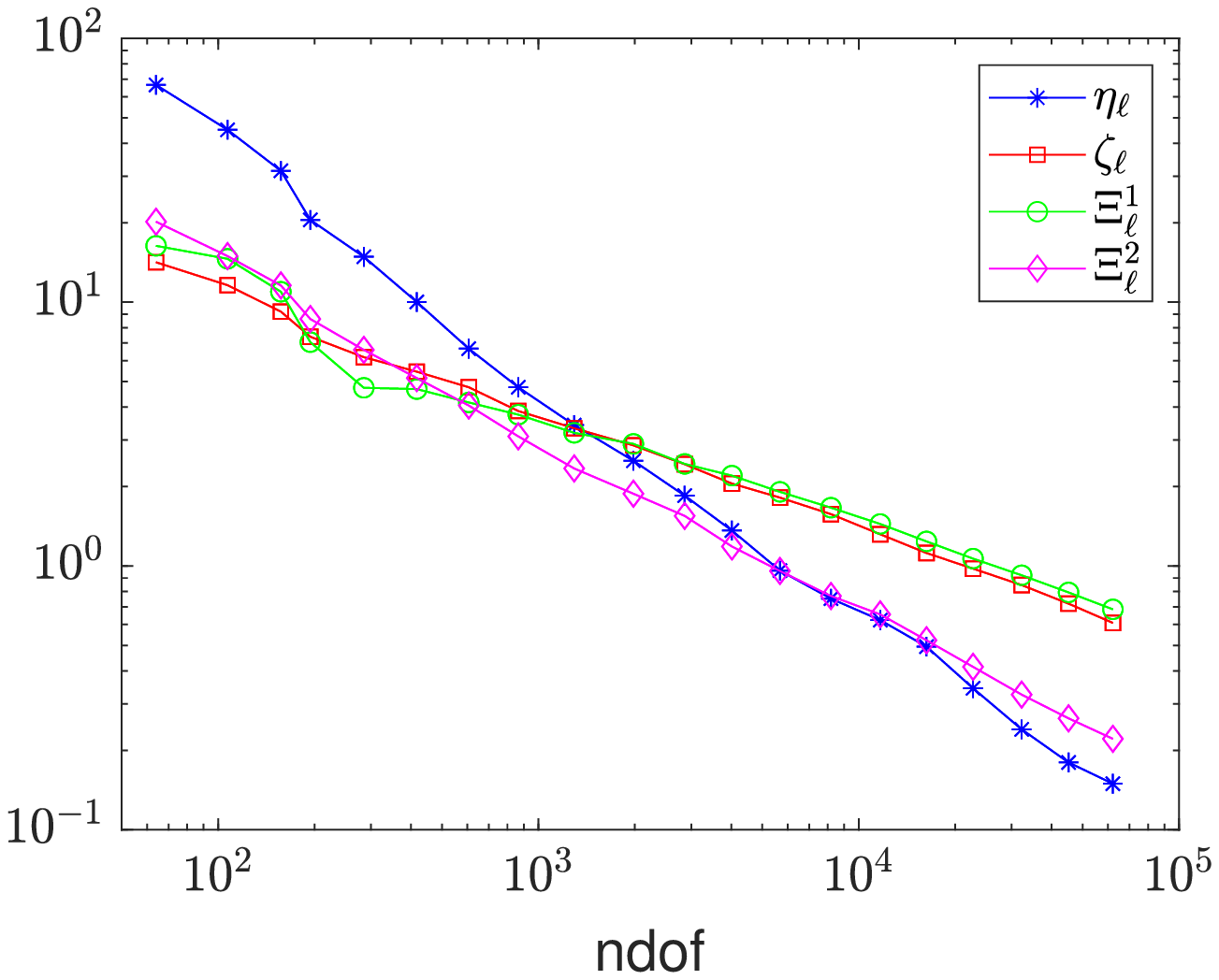}
	\end{subfigure}
	\caption{Convergence history plot of the error estimator components corresponding to $H2e$ for adaptive mesh-refinement vs ndof in Subsection~7.2 (left) and 7.3 (right).}
	\label{fig4.7}
\end{figure}

\noindent\textbf{Acknowledments}. The research of the first author has been supported by the German research foundation
in the Priority Program 1748  \textit{foundation and application of
generalized mixed FEM towards nonlinear problems in solid mechanics} (CA 151/22-2) and SPARC project (id 235)  \textit{the
mathematics and computation of plates} and SERB POWER Fellowship SPF/2020/000019. The second author acknowledges the financial support of the University Grants Commission from the Government of India.

\bibliographystyle{siamplain}
\bibliography{references}
\appendix
\newpage
\section{Proof of Lemma~2.6}
The integral  $a^P(v,w)=\int_P D^2v:D^2w\,dx$ for $v\in H^4(P)$ and $w \in H^2(P)$ allows an integration by parts formula \cite[Sec.~2.1]{antonietti2018fully} with the boundary terms $\mn(v):=v_{\bn \bn},\; T(v):= (\Delta v)_{\bn}+v_{\bn\bt\bt},\; \mt(v):=v_{\bn\bt}$ from \cite{antonietti2018fully} and the   abbreviations $[M_{\bn\bt}(v)]_{z_j}:=M_{\bn\bt}(v)|_{E(j-1)}(z_j)-M_{\bn\bt}(v)|_{E(j)}(z_j)$ with $E(0):=E(N_P)$ for a cyclic notation along $\partial P$. Those boundary terms are well-defined as traces   of a smooth function $v\in H^4(P)$ and then the formula reads
\begin{align}
a^P(v,w)= (\Delta^2v,w)_{L^2(P)}+(M_{\bn\bn}(v),w_\bn)_{L^2(\partial P)}-(T(v),w)_{L^2(\partial P)}+\sum_{j=1}^{N_P}[M_{\bn\bt}(v)]_{z_j}w(z_j).\label{2.6a}
\end{align}
The boundary terms require a smooth function $v$ like the quadratic polynomial $\chi\in\p_2(P)$ with $\Delta^2\chi=0=T(\chi)$ and with piecewise constants $M_{\bn\bn}(\chi)|_{E(k)}=\chi_{\bn\bn}|_{E(k)}$ and $M_{\bn\bt}(\chi)|_{E(k)}=\chi_{\bn\bt}|_{E(k)}$ for $k=1,\dots,N_P$, which are computable in terms of  $D^2\chi\in\mathbb{S}$ and of the geometry of $P$ from Figure~2.1.a. This and the definition of $G$ from  (2.15) lead,  for $\chi\in\p_2(P)$  and $w\in H^2(P)$ in \eqref{2.6a}, to 
\begin{align}
a^P(Gw,\chi)=a^P(\chi,w)= (M_{\bn\bn}(\chi),w_\bn)_{L^2(\partial P)}+\sum_{j=1}^{N_P}[M_{\bn\bt}(\chi)]_{z_j}w(z_j).\label{A1}
\end{align}
The dofs of $w$ from (2.1) allow for a re-writing of  the right-hand side of \eqref{A1}, namely \begin{align}a^P(Gw,\chi)=\sum_{k=1}^{N_P}\chi_{\bn \bn}|_{E(k)} \text{dof}_{N_P+k}(w)+\sum_{j=1}^{N_P}(\chi_{\bn\bt}|_{E(j-1)}-\chi_{\bn\bt}|_{E(j)})\text{dof}_j(w).\label{2.12a}\end{align}
This defines $Gw\in\p_2(P)$ up to an affine contribution fixed in (2.16). The first condition in (2.16)  reads  \begin{align}N_P^{-1}\sum_{j=1}^{N_P}Gw(z_j)=N_P^{-1}\sum_{j=1}^{N_P}w(z_j)=N_P^{-1}\sum_{j=1}^{N_P}\text{dof}_j(w).\label{A4}\end{align} For the second condition in (2.16),  the identity $\int_{E(k)}\nabla w\,ds=\text{dof}_{N_P+k}(w)\bn_{E(k)}+(\bt_{E(k-1)}-\bt_{E(k)})\text{dof}_k(w)$ for $k=1,\dots,N_P$ from (2.11) shows  
\begin{align}
\int_{\partial P}\nabla w\,ds=\sum_{k=1}^{N_P}\Big(\int_{E(k)}\nabla w\,ds\Big)=\sum_{k=1}^{N_P}(\text{dof}_{N_P+k}(w)\bn_{E(k)}+(\bt_{E(k-1)}-\bt_{E(k)})\text{dof}_k(w)).\label{A5}
\end{align}
The equations \eqref{2.12a}-\eqref{A5} form a linear system of 6 equations for $Gw\in\p_2(P)$ and the right-hand sides in \eqref{2.12a}-\eqref{A5}  are computable in terms of  $\text{dof}_1(w),\dots,\text{dof}_{2N_P}(w)$ for any $w\in H^2(P)$. Hence its solution $Gw$ is computable in terms of the dofs of $w$. It is elementary to see that $Gw$ is uniquely determined and the $6\times 6$ coefficient matrix in the resulting linear system of equations is regular. (Another proof for this follows from the estimates in the second part of the proof below.) This proves the first part of the lemma.

\medskip
\noindent The  second part part of the proof estimates $|Gw|_{m,P}$ for $m=0,1,2$ in terms of $\text{Dof}(w)=(\text{dof}_1(w),\dots,\text{dof}_{2N_P}(w))$.
The definition of $G$  with $\chi=Gw\in\p_2(P)$ in (2.15) and \eqref{2.12a} imply
\begin{align}
|Gw|^2_{2,P}=a^P(w,Gw)&\leq \sum_{k=1}^{N_P}|\chi_{\bn\bn}|_{E(k)}||\text{dof}_{N_P+k}(w)|+\sum_{j=1}^{N_P}(|\chi_{\bn\bt}|_{E(j-1)}-\chi_{\bn\bt}|_{E(j)}|)|\text{dof}_j(w)|.\nonumber
\end{align} 
Cauchy-Schwarz inequalities and $|\chi_{\bn\bn}|_{E(k)}|, |\chi_{\bn\bt}|_{E(k)}|\leq |D^2\chi|$ reveal
\begin{align*}
|Gw|^2_{2,P}\leq 2|D^2\chi|\sum_{j=1}^{2N_P}|\text{dof}_j(w)|\leq \sqrt{8N_P}|P|^{-1/2}|Gw|_{2,P}|\text{Dof}(w)|_{\ell^2}.
\end{align*}  The bounds $|P|^{-1/2}\leq \pi^{-1/2}C_{\text{sr}}^{-2} h_P^{-1}$ (cf. \cite[Chap. 1]{da2014mimetic}) with a shape regularity constant $C_{\text{sr}}$ of $\tT(P)$ (that exclusively depends on $\rho$)  and $N_P\leq M(\rho)$ from Subsection~2.1 show
\begin{align}
|Gw|_{2,P}\leq \sqrt{8\pi^{-1}C_{\text{sr}}^{-4}M(\rho)}h_P^{-1}|\text{Dof}(w)|_{\ell^2}.\label{2.22}
\end{align}
Define $a:=N_P^{-1}\sum_{j=1}^{N_P}w(z_j)\in\mathbb{R}$,  $B:=|\partial P|^{-1}\int_{\partial P}\nabla w\,ds\in\mathbb{R}^2$, and
the affine function $g(x):=a+B(x-N_P^{-1}\sum_{j=1}^{N_P}z_j)$ to control the lower-order terms $\sum_{m=0}^{1}h_P^{m-1}|Gw|_{m,P}$. The definition (2.16) of $G$  provides \begin{align*}\sum_{j=1}^{N_P}(Gw-g)(z_j)=\Big(\sum_{j=1}^{N_P}w(z_j)\Big)-aN_P=0,\quad\int_{\partial P}\nabla(Gw-g)\,ds=\int_{\partial P}\nabla w\,ds- |\partial P|B=0.
\end{align*}
The Poincar\'e-Friedrichs inequality from Theorem~2.2.b, therefore, shows  for $Gw-g$  that \begin{align}h_P^{-2}\|Gw-g\|_{L^2(P)}+h_P^{-1}|Gw-g|_{1,P}\leq C_\PF |Gw|_{2,P}.\label{2.23}\end{align}The Cauchy-Schwarz inequality implies $|a|\leq N_P^{-1/2} |\text{Dof}(w)|_{\ell^2}$. Since $|\bn_{E(k)}|=1=|\bt_{E(k)}|$,  \eqref{A5} shows that $|\int_{\partial P}\nabla w\,ds|\leq 2\sum_{j=1}^{2N_P}|\text{dof}_j(w)|$. A Cauchy-Schwarz inequality and $|\partial P|^{-1}\leq  N_P^{-1}\rho^{-1}h_P^{-1}$ from (M2) imply that $|B|\leq 2\sqrt{2}N_P^{-1/2}\rho^{-1}h_P^{-1}|\text{Dof}(w)|_{\ell^2}$.  The definition of $g$ and the previous two estimates for $|a|$ and $|B|$ result in
\begin{align}
\|g\|_{L^2(P)}\leq  |P|^{1/2}(|a|+h_P|B|)\leq  \big(\frac{\pi}{N_P}\big)^{1/2}(1+2\sqrt{2}\rho^{-1})h_P |\text{Dof}(w)|_{\ell^2}\label{2.24}
\end{align}
with the coarse bound $|P|\leq \pi h_P^2$ in the last step. This and the inverse estimate for $|g|_{1,P}$ (collected from an inverse estimate in the triangle $T\in\tT(P)$ and so with a well-established bound for $C_{\text{inv}}$) show
\begin{align}
|g|_{1,P}\leq C_{\text{inv}}h_P^{-1}\|g\|_{L^2(P)}\leq \big(\frac{\pi}{N_P}\big)^{1/2}(1+2\sqrt{2}\rho^{-1})C_{\text{inv}} |\text{Dof}(w)|_{\ell^2}.\label{A9}
\end{align}
Triangle inequalities $\sum_{m=0}^2h_P^{m-2}|Gw|_{m,P}\leq \sum_{m=0}^2h_P^{m-2}(|Gw-g|_{m,P}+|g|_{m,P})$,  \eqref{2.23}-\eqref{A9}, and  the abbreviation $C_{9}:=1+C_\PF+(1+C_{\text{inv}})\pi^{1/2}(1+2\sqrt{2}\rho^{-1})$ prove
\begin{align*}
\sum_{m=0}^2h_P^{m-1}|Gw|_{m,P}\leq C_{9}(h_P|Gw|_{2,P}+ |\text{Dof}(w)|_{\ell^2})\leq C_g|\text{Dof}(w)|_{\ell^2}
\end{align*}
with \eqref{2.22}  and $C_g:=C_{9}(1+\sqrt{8\pi^{-1}C_{\text{sr}}^{-4}M(\rho)})$ in the last step. This concludes the proof. (Notice that, in particular, $\text{Dof}(w)=0$ implies $Gw=0$ and this proves that the linear system of equations \eqref{2.12a}-\eqref{A5} involves a regular coefficient matrix as announced.)\qed
\section{Proof of Proposition~3.1}\label{A}
\textit{Step 1 defines an HCT finite element space}.  Recall the sub-triangulation $\tT(P)$ from Subsection~2.1  and decompose any triangle $T\in\tT(P)$ further into three sub-triangles $\mathcal{K}(T)$  depicted in Figure~B.1.c. Then the Hsieh-Clough-Tocher (HCT) finite element space \cite{ciarlet2002finite} reads\[
\h(\tT(P)):=\{\hat{v}\in H^2(P):\forall\; T\in\tT\quad\hat{v}|_T\in\p_3(\mathcal{K}(T))\}.\]
\begin{figure}[H]
	\centering
	\begin{subfigure}{.33\textwidth}
		\centering
		\includegraphics[width=0.4\linewidth]{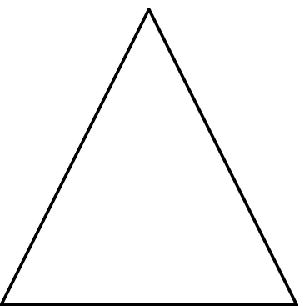}        
		\caption{}
		\vspace{1cm}
	\end{subfigure}%
	\begin{subfigure}{.33\textwidth}
		\centering
		\includegraphics[width=0.4\linewidth]{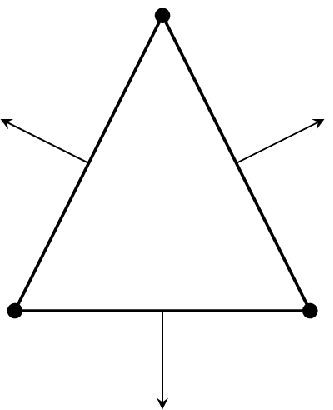}               
		\caption{}
		\vspace{1cm}
	\end{subfigure}%
	\begin{subfigure}{.33\textwidth}
		\centering
		\includegraphics[width=0.4\linewidth]{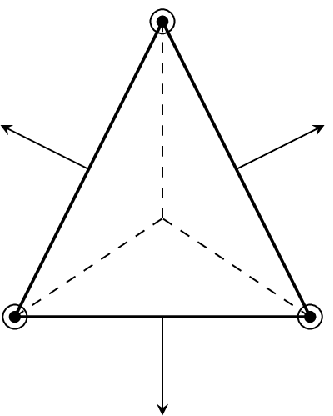}
		\caption{}
		\vspace{1cm}
	\end{subfigure}
	\caption{(a) Triangle $T$,  (b) Morley, (c) HCT.}
	\label{tz}
\end{figure}
\noindent  The standard degrees of freedom (dofs) in the HCT finite element (cf. \cite[Chap.~6]{ciarlet2002finite} or \cite[Sec.~2.3]{gallistl2015morley})  are the nodal values of the function and
its first-order derivatives at each vertex and the mid-point values of the normal derivatives along  each edge of a triangle as depicted in Figure~B.1.c. This paper   utilizes the integral means instead of the  mid-point values  of normal derivatives  along edges  and all other dofs (i.e., nodal values) are unchanged. 
Let $\psi_1^\h,\dots,\psi_{2N_P}^\h$ be the $2N_P$ nodal basis functions  in $\h(\tT(P))$  with $\nabla\psi_j^\h(z_\ell)=0$ for all  $j=1,\dots,2N_P$ and, for all $k,\ell=1,\dots,N_P$,  \begin{align}
\psi_k^\h(z_\ell)=\delta_{k\ell},\; \dashint_{E(\ell)}(\psi_k^\h)_\bn\,ds=0,\;\text{and}\;\psi_{k+N_P}^\h(z_\ell)=0,\; \dashint_{E(\ell)}(\psi_{k+N_P}^\h)_\bn\,ds=\delta_{k\ell}.\label{B}
\end{align} 
In contrast to this, let $\widetilde{\psi}_{\ell+N_P}^\h\in\h(\tT(P))$  be a  nodal basis function in the standard  HCT finite element with $\widetilde{\psi}_{\ell+N_P}^\h(\text{mid}(E(\ell)))=1$ for $\ell= 1,\dots,N_P$, while all other dofs vanish.  Then $4\int_0^1 s(1-s)\,ds=2/3$ leads to $\psi_{\ell+N_P}^\h=\frac{3}{2}\widetilde{\psi}_{\ell+N_P}^\h$  for $\dashint_{E(\ell)}(\psi_{\ell+N_P}^\h)_\bn\,ds=1$ and $\ell=1,\dots,N_P$. This observation, the scaling of the standard HCT basis functions  from  \cite[Prop. 2.5]{gallistl2015morley}, and the bound $h_T^{-1}\leq \rho^{-1}h_P^{-1}$ from (M2) for all $T\in\tT(P)$ provide a positive constant $C_\h$ (that exclusively depends on $\rho$) in \begin{align}\max_{k=1}^{N_P}h_P|\psi_k^\h|_{2,P}+\max_{\ell=1}^{N_P}|\psi_{\ell+N_P}^\h|_{2,P}\leq C_\h.\label{scale}\end{align}  

\bigskip
\noindent\textit{Step 2 constructs an HCT interpolation}.   
Recall the nodal basis functions $\psi_1^\h,\dots,\psi_{2N_P}^\h$ of $\h(\tT(P))$ selected in Step 1. The HCT interpolation $w_\h\in \h(\tT(P))$ of a given $w\in H^2(P)$  reads
\begin{align}
w_\h:=\sum_{k=1}^{N_P} w(z_k)\psi_k^\h+\sum_{\ell=1}^{N_P}\left(\dashint_{E(\ell)}w_\bn\,ds\right) \psi_{\ell+N_P}^\h.\label{I}
\end{align}
The  duality relations \eqref{B} imply for $m=1,\dots,N_P$ in \eqref{I} that
$
w_\h(z_m)=w(z_m)$ and $\int_{E(m)}(w_\h)_\bn\,ds=\int_{E(m)}w_\bn\,ds.
$
In other words, $\text{Dof}(w_\h)=\text{Dof}(w)$ for the vector Dof with components from (2.1).

\bigskip 
\noindent\textit{Step 3 defines the Hilbert space $(V_0,a^P)$}.   The kernel of the linear map $\text{Dof}:H^2(P)\to \mathbb{R}^{2N_P}$ is the closed subspace
\[V_0:=\left\{v\in H^2(P): v(z_j)=0=\int_{E(j)}v_\bn\,ds\quad\text{for}\;j=1,\dots,N_P\right\}\]
of the Hilbert space $H^2(P)$ and so complete.  Hence $(V_0,a^P)$ is a Hilbert space. Notice that $w-w_\h\in V_0$ for any $w\in H^2(P)$ from Step 2 and  $\pid w=0$ for $w\in V_0$ (from Lemma~2.6 with $\text{Dof}(w)=0$ as explained in Appendix~A).

\bigskip
\noindent\textit{Step 4 proves that $V_0^\perp:=\{v\in H^2(P): a^P(v,w)=0\quad\text{for all}\;w\in V_0\}\subset\widehat{V}_h(P).$} Given any $v\in V_0^\perp$ and $w\in H^2(P)$, define the HCT interpolation $w_\h$ of $w$ from Step 2. Then $w-w_\h\in V_0$ from Step 3 implies that
\begin{align}
a^P(v,w)=a^P(v,w_\h)&=\sum_{k=1}^{N_P}a^P(v,\psi_k^\h)w(z_k)+\sum_{\ell=1}^{N_P}a^P(v,\psi_{\ell+N_P}^\h)h_{E(\ell)}^{-1}\int_{E(\ell)} w_\bn\,ds.\label{A2}
\end{align}
Define  $f:=0\in\p_{-1}(P)$, $a_\ell:=a^P(v,\psi_\ell^\h)$, and $g\in\p_0(\e(P))$ by $g|_{E(\ell)}:=a^P(v,\psi_{\ell+N_P}^\h)h_{E(\ell)}^{-1}$ for $\ell=1,\dots,N_P$ to rewrite the right-hand side in \eqref{A2} as
\[a^P(v,w)=(g, w_\bn)_{L^2(\partial P)}+\sum_{\ell=1}^{N_P}a_\ell w(z_\ell) \quad\text{for all}\;w\in H^2(P).\]
This implies $v\in\widehat{V}_h(P)$ for $r=-1$ and concludes the proof of  Step~4. \qed

\bigskip
\noindent\textit{Step 5  proves that $\mathrm{Dof}:\widehat{V}_h(P)\to\mathbb{R}^{2N_P}$ is surjective.} Given any $x=(x_1,\dots,x_{2N_P})\in\mathbb{R}^{2N_P}$ and the  functions $\psi_1^\h,\dots,\psi_{2N_P}^\h$ selected in Step 1, define 
$u_{BC}:=\sum_{k=1}^{2N_P}x_k\psi_k^\h\in H^2(P)$ with $\text{Dof}(u_{BC})=x$ from Step 2.
Let $u_0\in V_0$ denote the Riesz representation of $a^P(u_{BC},\cdot)$ in $(V_0,a^P)$, i.e., $a^P(u_0,\cdot)=a^P(u_{BC},\cdot)$ in $V_0$. Since $a^P(u_0-u_{BC},\cdot)=0$ in $V_0$, Step 4 shows that $\widehat{u}_h:=u_{BC}-u_0\in\widehat{V}_h(P)$. Recall $\text{Dof}(u_{BC})=x$ and $\text{Dof}(u_0)=0$ to deduce $\text{Dof}(\widehat{u}_h)=x$. \qed

\bigskip
\noindent\textit{Step 6 establishes an inclusion in $\widehat{V}_h(P)$ with a non-zero $f$}.  Given $f\in \p_r(P)$, the Riesz-representation theorem guarantees the unique existence of the weak solution $u(f)\in V_0$ to
\begin{align}a^P(u(f),v)=(f,v)_{L^2(P)}\quad\text{for all}\;v\in V_0.\label{2.8a}\end{align} 
Recall that $r$ is a fixed parameter in $\{-1,0,1,2\}$ and $r=-1$ is trivial in this step. Since \eqref{2.8a} implies $\Delta^2u(f)=f$ in $P$, it remains to prove that $u(f)\in\widehat{V}_h(P)$. 
Given any $w\in H^2(P)$ with $w_\h$ from \eqref{I}, $w-w_\h\in V_0$ leads in \eqref{2.8a} to  $a^P(u(f),w-w_\h)=(f, w-w_\h)_{L^2(P)}$. Hence
\begin{align}
a^P(u(f),w)=(f,w-w_\h)_{L^2(P)}+a^P(u(f),w_\h)=(f,w)_{L^2(P)}+\Lambda(w_\h)\label{2.12}
\end{align} 
with the linear functional  $\Lambda(v):=a^P(u(f),v)-(f,v)_{L^2(P)}$ for any $v\in H^2(P)$. The representation \eqref{I} of $w_\h$ shows 
\begin{align}
a^P(u(f),w)&=(f,w)_{L^2(P)}+\sum_{k=1}^{N_P} w(z_k)\Lambda(\psi_k^\h)+\sum_{\ell=1}^{N_P}\left(\dashint_{E(\ell)}w_\bn\,ds\right) \Lambda(\psi_{\ell+N_P}^\h)\nonumber\\&=(f,w)_{L^2(P)}+(g, w_\bn)_{L^2(\partial P)}+\sum_{k=1}^{N_P} a_k w(z_k)\quad\text{for all}\;w\in H^2(P)\label{2.19}
\end{align} 
with the  definition of $a_\ell:=\Lambda(\psi_\ell^\h)$ and $g\in\p_0(\e(P))$ by $g|_{E(\ell)}:=|E(\ell)|^{-1}\Lambda(\psi_{\ell+N_P}^\h)$ for $\ell=1,\dots,N_P$. This implies that $u(f)\in\widehat{V}_h(P)$. Notice that $u(f)$  depends linearly on $f\in\p_r(P)$ and so $\p_r(P)\to \widehat{V}_h(P), f\mapsto u(f)$ defines a linear map.

\bigskip
\noindent\textit{Step 7 proves that $\mathcal{L}:\p_r(P)\to\p_r(P), f\mapsto \Pi_ru(f)$  is an isomorphism}. For any $f\in\p_r(P)$ with $\mathcal{L}f=0$, the orthogonality $(1-\Pi_r)u(f)\perp\p_r(P)$ in $L^2(P)$ shows $0=\int_P(\mathcal{L}f)f\,dx=\int_P u(f)f\,dx$. This and $v=u(f)$ in \eqref{2.8a} result in
$
0=a^P(u(f),u(f))=|u(f)|_{2,P}.
$
Consequently $u(f)\in\p_1(P)\cap V_0$ and so $u(f)=0$. Thus  $f=\Delta^2u(f)=0$ and $\mathcal{L}$ is injective; whence  bijective. \qed

\bigskip
\noindent\textit{Step 8 proves that $\mathrm{Dof}:V_h(P)\to \mathbb{R}^{2N_P}$ is an isomorphism}.

\medskip
\noindent\textit{Proof of surjectivity}.  Given any $x\in\mathbb{R}^{2N_P}$, there exists $\widehat{u}_h\in \widehat{V}_h(P)$ with $\text{Dof}(\widehat{u}_h)=x$  from Step 4. This leads to $g:=\Pi_r\pid\widehat{u}_h-\Pi_r\widehat{u}_h\in\p_r(P)$. Since $\mathcal{L}$ is bijective in $\p_r(P)$ (from Step 7), there exists $f\in\p_r(P)$ with $\Pi_ru(f)=g$. Recall that $u(f)\in V_0$ implies $\mathrm{Dof}(u(f))=0$ and $\pid u(f)=0$. Altogether, $u_h^P:=u(f)+\widehat{u}_h\in\widehat{V}_h(P)$ satisfies $\mathrm{Dof}(u^P_h)=x$  and $\Pi_r\pid u^P_h=\Pi_r\pid\widehat{u}_h=g+\Pi_r\widehat{u}_h=\Pi_ru(f)+\Pi_r\widehat{u}_h=\Pi_ru_h^P$. Hence $u^P_h\in V_h(P)$.\qed

\medskip
\noindent\textit{Proof of injectivity}. 
Suppose $v_h\in V_h(P)$ satisfies $\mathrm{Dof}(v_h)=0$. Recall  $v_h\in V_h(P)\cap V_0$ and $\pid v_h=0$ from Step 3.  The definition (3.1) (for $v=v_h$ with $\text{Dof}(v)=0$) leads for some $f\in\p_r(P)$ such that $a^P(v_h,v_h)=(f, v_h)_{L^2(P)}$. This and (3.2) reveal
\[|v_h|_{2,P}^2=a^P(v_h,v_h)=(f, v_h)_{L^2(P)}=(f,\Pi_rv_h)_{L^2(P)}=(f,\Pi_rGv_h)_{L^2(P)}=0.\]
Hence $v_h\in V_0\cap\p_1(P)$ and so $v_h=0$.\qed

\bigskip
\noindent\textit{Proof of \textbf{(H1)}}. The key observation from Step 8 is that $V_h(P)$ has the dimension $2N_P$ and $\text{dof}_1,\dots,\text{dof}_{2N_P}$ from (2.1) are  linear independent. Consequently $(P,V_h(P),(\text{dof}_1,\dots,\text{dof}_{2N_P}))$ is a finite element in the sense of Ciarlet.\qed

\bigskip
\noindent\textit{Step 9 provides the scaling of an HCT interpolation}. Let $\psi_h\equiv \psi_p$ be a nodal basis function of the finite element  $(P,V_h(P),\text{Dof})$ for some $p\in\{1,\dots,2N_P\}$ and let $\psi_\h$ be its HCT interpolation from Step 2, namely
\begin{align}
\psi_\h=\sum_{k=1}^{N_P} \psi_h(z_k)\psi_k^\h+\sum_{\ell=1}^{N_P}\left(\dashint_{E(\ell)}(\psi_h)_\bn\,ds\right) \psi_{\ell+N_P}^\h.\label{B9}
\end{align}
The definition of a nodal basis function $\psi_h$ shows that $\psi_h(z_k)$ and $ \int_{E(\ell)}(\psi_h)_\bn\,ds$ are zero or one for $k,\ell=1,\dots,N_P$.  The scaling of $\psi_1^\h,\dots,\psi_{2N_P}^\h$ from \eqref{scale}, and the bound $ h_{E}^{-1}\leq \rho^{-1}h_P^{-1}$ for all $E\in\e(P)$ from (M2) lead to  \begin{align}|\psi_\h|_{2,P}\leq C_\h(1+\rho^{-1}) h_P^{-1}.\label{B10}\end{align}

\bigskip
\noindent\textit{Step 10 controls a nodal basis function in $V_h(P)$ by its HCT interpolation}. For a given nodal basis function $\psi_h\in V_h(P)$, its HCT interpolation $\psi_\h$ from \eqref{B9},  $\text{Dof}(\psi_h-\psi_\h)=0$ from Step 3, and the test function $w=\psi_h-\psi_\h$  lead in (3.1) to \begin{align}a^P(\psi_h,\psi_h-\psi_\h)=(f,\psi_h-\psi_\h)_{L^2(P)}\label{B7}\end{align} for some $f\in\p_r(P)$.  The definition of the $L^2$ projection $\Pi_r$  and the relation $\Pi_rG\psi_h=\Pi_r\psi_h$ from (3.2) for $r=-1,0,1,2$ show that  \[(f,\psi_h)_{L^2(P)}=(f,\Pi_rG\psi_h)_{L^2(P)}=(f,\Pi_rG\psi_\h)_{L^2(P)}\] with $G\psi_h=G\psi_\h$ from $\text{Dof}(\psi_h)=\text{Dof}(\psi_\h)$ and Lemma~2.6 in the last step. This, \eqref{B7}, and a Cauchy-Schwarz inequality imply  that
\begin{align}
|\psi_h|^2_{2,P}&\leq |\psi_h|_{2,P}|\psi_\h|_{2,P}+\|f\|_{L^2(P)}\|\Pi_rG\psi_\h-\psi_\h\|_{L^2(P)}.\label{B1}
\end{align}
Recall $f=0$ for $r=-1$ so suppose $r=0,1,2$ for the time being. Since $\int_P(1-\Pi_r)G\psi_\h\,dx=0$ from the definition of the $L^2$ projection $\Pi_r$ for $r=0,1,2$, the Poincar\'e inequality from \cite[Subsec.~2.1.5]{14} shows  that  $\|(1-\Pi_r)G\psi_\h\|_{L^2(P)}\leq C_{\text{P}}h_P|G\psi_\h|_{1,P}$ with a positive constant $C_{\text{P}}$ (that exclusively depends on $\rho$).  
Since   $\psi_h$ is a nodal basis function, $\int_E(\psi_\h)_\bn\,ds=\int_E(\psi_h)_\bn\,ds$ for all $E\in\e(P)$ imply that $\int_E(\psi_\h)_\bn\,ds=0$ for all but at most one $E\in\e(P)$ and  the Poincar\'e-Friedrichs inequality  from Theorem~2.2.a  shows $|\psi_\h|_{1,P}\leq C_\PF h_P|\psi_\h|_{2,P}$. This,  Lemma~2.4, and   a triangle inequality result in $|G\psi_\h|_{1,P}\leq |G\psi_\h-\psi_\h|_{1,P}+|\psi_\h|_{1,P}\leq 2C_\PF h_P|\psi_\h|_{2,P}$. Hence the previous estimates lead to
\begin{align*}
\|(1-\Pi_r)G\psi_\h\|_{L^2(P)}\leq 2C_{\text{P}}C_\PF h_P^2|\psi_\h|_{2,P}.
\end{align*} This, Lemma~2.4, and 
$\|(1-\Pi_rG)\psi_\h\|_{L^2(P)}\leq \|(1-\Pi_r)G\psi_\h\|_{L^2(P)}+\|(1-G)\psi_\h\|_{L^2(P)}$ result in \begin{align}\|(1-\Pi_rG)\psi_\h\|_{L^2(P)}\leq C_\PF(1+2C_{\text{P}}) h_P^2|\psi_\h|_{2,P}.\label{B1n}\end{align}

\bigskip
\noindent\textit{ Step 11 bounds the term $\|f\|_{L^2(P)}$}. Recall the bubble-function $b_P$ from the proof of Theorem~4.3. The substitution of $\chi=f\in\p_r(P)$ in the first  estimate of (4.15) proves that \[C_b^{-1}\|f\|^2_{L^2(P)}\leq (f,b_Pf)_{L^2(P)}=a^P(\psi_h,b_Pf)\] with $w=b_Pf\in H^2_0(P)$ in (3.1)   in the last step. A Cauchy-Schwarz inequality and the second inverse estimate of (4.16) imply 
\[C_b^{-1}\|f\|_{L^2(P)}^2\leq |\psi_h|_{2,P}|b_Pf|_{2,P}\leq C_bh_P^{-2}|\psi_h|_{2,P}\|f\|_{L^2(P)}.\] 
Consequently, $\|f\|_{L^2(P)}\leq C_b^2h_P^{-2}|\psi_h|_{2,P}$. 

\bigskip
\noindent\textit{Proof of \textbf{(H2)}}. The last estimate and the combination of \eqref{B1}-\eqref{B1n} result in
\[|\psi_h|_{2,P}\leq (1+C_b^2(C_\PF(1+2C_{\text{P}}))) |\psi_\h|_{2,P}\]
for $r=0,1,2$. For $r=-1$, $f=0$ and \eqref{B7} show $|\psi_h|_{2,P}\leq |\psi_\h|_{2,P}$. The combination with \eqref{B10}   verifies  \textbf{(H2)} with $C_{\text{stab}}:=C_\h (1+\rho^{-1})(1+C_b^2(C_\PF(1+2C_{\text{P}})))$ for a nodal basis function $\psi_h$ in $V_h(P)$ from (3.1)-(3.2) and concludes the proof.\qed
\begin{remark}[trace of $H^2$ functions]\label{rem1}
	The trace operator $\text{tr}:=(\gamma_0,\gamma_1): H^2(P)\to H^{3/2}(\partial P)\times H^{1/2}(\partial P)$ for a polygon $P$ is \textit{not} surjective, i.e., $v|_{\partial P}$ and $v_\bn|_{\partial P}$ are \textit{not} independent of each other \cite{costabel1996invertibility, fuhrer2021trace}. One consequence for weak solutions is that we \textit{cannot} immediately 
	split the boundary conditions in the weak form  \eqref{2.6a}  into the two strong formulations     $M_{\bn\bn}(v_h)|_{\partial P}\in\p_0(\e(P))$ and $T(v_h)|_{\partial P}=0$ for $v_h\in V_h(P)$. 
\end{remark}
\begin{remark}[weak and strong formulations of VE functions]\label{rem2}
	Compared to the current VE literature on fourth-order problems \cite{antonietti2018fully,chen2020nonconforming,dedner2022robust,huang2021medius,zhao2018morley}, the definition of $V_h(P)$ in (3.1)-(3.2) or (3.3)-(3.4) looks different. In Example~1, for instance, the analog to $V_h(P)$ in \cite{antonietti2018fully,chen2020nonconforming} reads 
	\begin{align}
	V_h^s(P):=\begin{rcases}\begin{dcases}v_h\in H^2(P):& \Delta^2v_h\in\p_r(P),\; \mn(v_h)|_{\partial P}\in\p_0(\e(P))\;\text{and}\;\\& T(v_h)|_{\partial P}\in\p_{-1}(\partial P),\; v_h-Gv_h\perp \p_r(P)\;\text{in}\;L^2(P)\end{dcases}\end{rcases}.\label{0}
	\end{align}
	We refer to this as the strong formulation, but utilize the weak form of $V_h(P)$ in (3.1)-(3.2) throughout this paper. The point is the regularity of the weak solution $v\in H^2(P)/\p_1(P)$ to
	\begin{align}
	a^P(v,w) = (f,w)_{L^2(P)}+(g,w_\bn)_{L^2(\partial P)}+\sum_{j=1}^{N_P}a_jw(z_j)\quad\text{for all}\;w\in H^2(P).\label{1}
	\end{align}
	The weak solution $v\in H^2(P)$ is unique up to affine functions $\p_1(P)$ and the right-hand side displays a given $f\in\p_r(P), g\in\p_0(\e(P))$ and $a_1,\dots,a_{N_P}\in\mathbb{R}$.    The integration by parts formula from \eqref{2.6a}  exploits  the formula  \eqref{1} as
	\begin{align*}
	a^P(v,w) &= (\Delta^2v,w)_{L^2(P)}+(M_{\bn\bn}(v),w_\bn)_{L^2(\partial P)}-(T(v),w)_{L^2(\partial P)}+\sum_{j=1}^{N_P}[M_{\bn\bt}(v)]_{z_j}w(z_j)
	\end{align*}
	that holds for smooth $v\in H^4(P)$ and for all $w\in H^2(P)$. The technical issue is that only for smooth $v$, it holds $M_{\bn\bn}(v)|_{\partial P}\in\p_0(\e(P))$ and $T(v)|_{\partial P}\in\p_{-1}(\partial P)$ and, to the best knowledge of the authors, it is unclear whether the weak solution $v\in H^2(P)\cap C^\infty(\text{int}(P))$ to \eqref{1} allows a weak definition of the individual terms $M_{\bn\bn}(v)|_{\partial P}$ and $T(v)|_{\partial P}$ on the boundary $\partial P$ despite the fact that elliptic regularity guarantees that $v\in C^\infty(\text{int}(P))$ is smooth inside the polygonal domain $P$. A routine argument  with a test function $w\in\mathcal{D}(\text{int}(P))$ leads to
	\begin{align}
	(g,w_\bn)_{L^2(\partial P)}+\sum_{j=1}^{N_P}a_jw(z_j)=(M_{\bn\bn}(v),w_\bn)_{L^2(\partial P)}-(T(v),w)_{L^2(\partial P)}+\sum_{j=1}^{N_P}[M_{\bn\bt}(v)]_{z_j}w(z_j)\label{3}
	\end{align}
	provided $M_{\bn\bn}(v)|_{\partial P}$ and $T(v)|_{\partial P}$ can be defined well. Remark~\ref{rem1} already gives a warning for this and the consequence in the literature on fourth-order problems is, cf., e.g. \cite{fuhrer2021trace}, that only the sum of the right-hand side in \eqref{3}  is a well-defined linear functional and it is never split into $M_{\bn\bn}(v)|_{\partial P}$ and $T(v)|_{\partial P}$. Clearly, once we knew that both $M_{\bn\bn}(v)|_{\partial P}$ and $T(v)|_{\partial P}$ are distributions, the identity \eqref{3} might reveal that $g=M_{\bn\bn}(v)|_{\partial P}$ and $T(v)|_{\partial P}=0$, but we do not know that. On the formal level, if we have the information that $v\in V_h(P)$ is smooth up to the boundary, then $v\in V_h^s(P)$ would follow. In this sense,  the notation $V_h^s(P)$ is interpreted as a strong formulation of $V_h(P)$ from (3.1)-(3.2). We understand that, with the substitution of $V_h^s(P)$ from \eqref{0} as $V_h(P)$, the results of \cite{antonietti2018fully,chen2020nonconforming} are well-defined and remain valid. At least the paper \cite{chen2020nonconforming} already  adopts this point of view in \cite[Lemma~3.4]{chen2020nonconforming}.
	
	\medskip
	\noindent Analog remarks apply to \cite{zhao2018morley} and  Example~2 in this paper.
\end{remark}

\section{Proof of Proposition~3.2}\label{SuppB}
\subsection{One-dimensional finite element}
Recall the vertices $z_1,\dots,z_{N_P}$ and the edges $E(1),\dots,E(N_P)$, the corners $\zeta_1,\dots,\zeta_J$ and the sides $\gamma(1),\dots,\gamma(J)$ of a polygonal domain $P\in\T$, and recall the vector space $S(j)$ and a linear functional $\Lambda_j$ for $j=1,\dots,J$  from Subsection~3.2.2. Change the coordinate system to let $\gamma(j)=(0,L)$ belong to  the real axis $\mathbb{R}\times\{0\}$  and identify $ (z_{k(j)},\dots,z_{k(j)+m(j)})\equiv(0,s_1,\dots,s_m)$ with a quasi-uniform partition $0=s_0<s_1<\dots<s_m=L$. Then $S(j)=\p_2(\e(\gamma(j)))\cap C^1[0,L]$ has a known $B_2^2$ spline basis, while we  define $\varphi_\ell\in S(j)$ for $\ell=0,\dots,m$ by 
\[\varphi_\ell(s_\ell)=1, \varphi_\ell'(s_\ell)=0=\varphi_\ell(s_k)\quad\text{for all}\;k=0,\dots,m\;\text{and}\; k\neq \ell.\]

\begin{lemma}[construction of $\varphi_\ell$]\label{lemc}
	The functions $\varphi_0,\dots,\varphi_m$ belong to $S(j)$ and there exists a positive constant $C_{10}$ (that exclusively depends on $\rho$) such that 
	\begin{align*}\|\varphi_\ell\|_{L^\infty(\gamma(j))}\leq C_{10}\quad\text{for}\; \ell=0,\dots,m.\end{align*}
\end{lemma}
\begin{proof} The construction of $\varphi_\ell$ first determines the piecewise quadratic $\varphi_\ell$ uniquely  in the intervals $(s_{\ell-1},s_\ell)$ and $(s_\ell,s_{\ell+1})$. This determines  the derivative  $\varphi'_\ell$  at the end points $s_{\ell-1}$ and $s_{\ell+1}$. The values $\varphi'_\ell(s_{\ell-1})$ and $\varphi_\ell(s_{\ell-2})=0=\varphi_\ell(s_{\ell-1})$ lead to a unique quadratic polynomial  $\varphi_\ell|_{(s_{\ell-2},s_{\ell-1})}$. A successive application of this argument to the remaining intervals leads to a   $\varphi_\ell\in S(j)$. Let $h_\ell:=s_\ell-s_{\ell-1}$ for $\ell=1,\dots,m$. A direct computation of $\varphi_0,\dots,\varphi_m$ displayed in Figure~C.1 and C.2  controls the extrema $\|\varphi_\ell\|_{L^\infty(\gamma(j))}\leq\max\{1,\max_{p,q=0,\dots,m}h_p/(2h_q)\}$. It follows from (M2) that $h_\ell\approx L$ and so $\|\varphi_\ell\|_{L^\infty(\gamma(j))}\leq C_{10}\quad\text{for}\; \ell=0,\dots,m$ with a positive constant $C_{10}$ that exclusively depends on $\rho$.
\end{proof} 
\begin{figure}[H]
	\centering
	\includegraphics[width=0.65\linewidth]{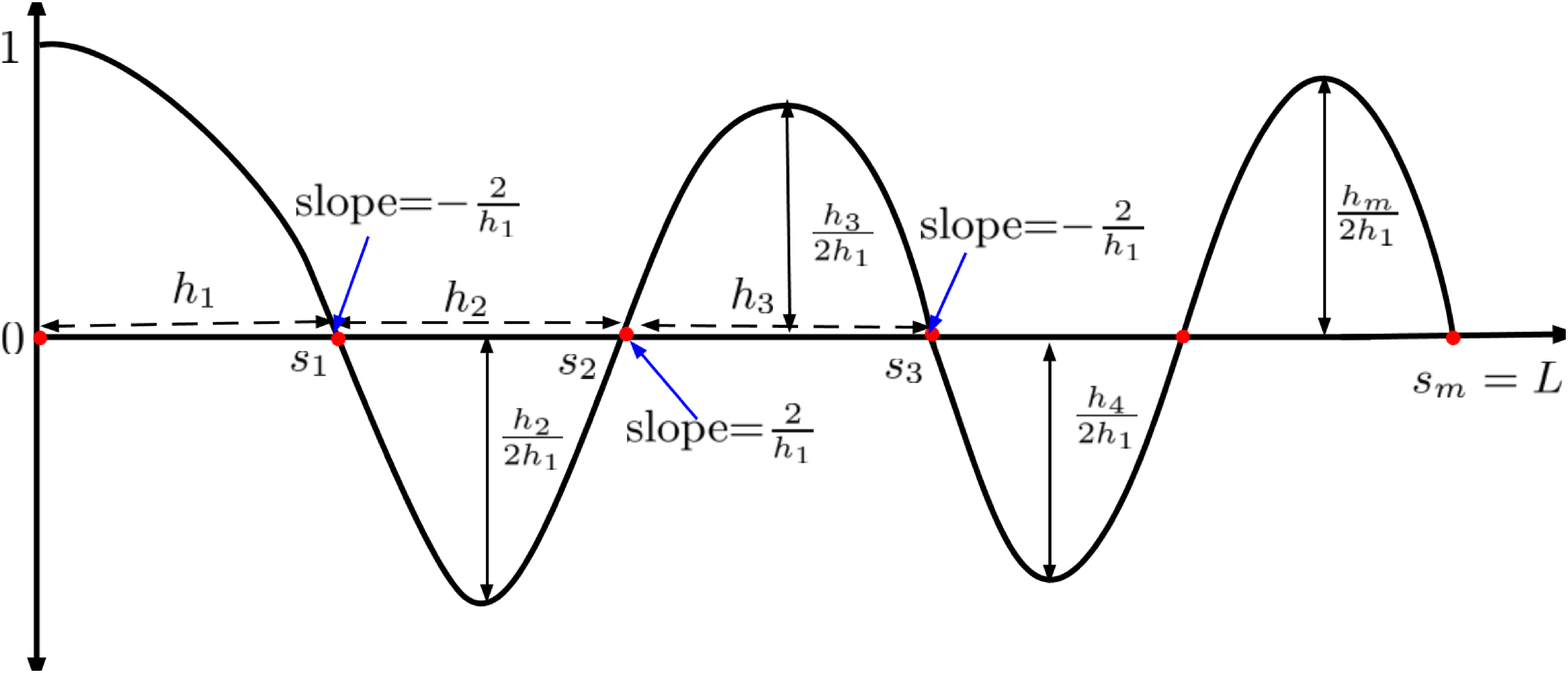}
	\caption{The nodal basis  function $\varphi_0$.}
\end{figure}%
\begin{figure}[H]
	\centering
	\includegraphics[width=0.65\linewidth]{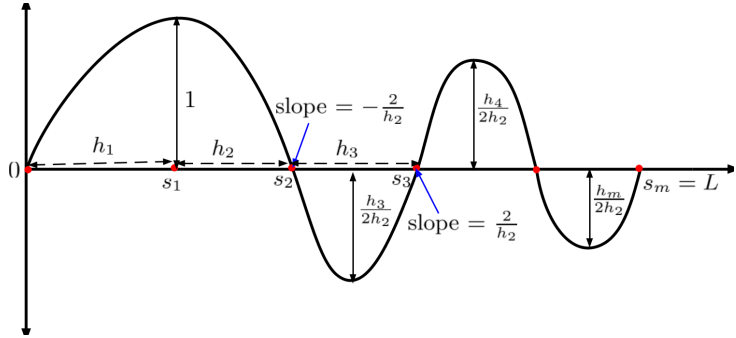}
	\caption{The nodal basis  function $\varphi_1$.}
\end{figure}%
\noindent Define $\widetilde{\psi}$ with $\widetilde{\psi}(0)=0=\widetilde{\psi}(s_1)$ and $\widetilde{\psi}(h_1/2)=1$ in the first interval $(0,s_1)$. Then compute the slope   $-4/h_1$ at $s_1$ and define $\widetilde{\psi}$ uniquely in $(s_1,s_2)$ with $\widetilde{\psi}'(s_1)=-4/h_1$ and $\widetilde{\psi}(s_1)=0=\widetilde{\psi}(s_2)$. Continue the procedure by preassigning the derivative values $(-1)^{\ell-1}4/h_1$ at the left vertex $s(\ell-1)$ and function values $\widetilde{\psi}(s_{\ell-1})=0=\widetilde{\psi}(s_\ell)$ for all $\ell=2,\dots,m$ so that $\widetilde{\psi}\in S(j)$, and verify that extrema in each interval $(s_{\ell-1},s_\ell)$ is $(-1)^{\ell-1}h_\ell/h_1$. Since $\|\widetilde{\psi}\|_{L^{\infty}(\gamma(j))}\geq 1$, rescale $\psi:=\frac{\widetilde{\psi}}{\|\widetilde{\psi}\|_{L^\infty(\gamma(j))}}$ to derive $\|\psi\|_{L^\infty(\gamma(j))}=1$.
\begin{figure}[H]
	\centering
	\includegraphics[width=0.65\linewidth]{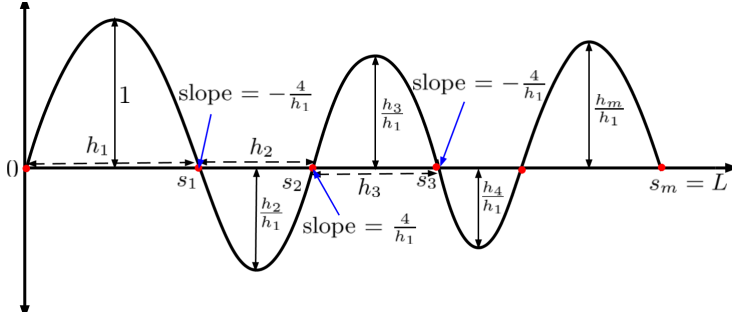}
	\caption{The nodal basis function $\psi$ with $\|\psi\|_{L^\infty(\gamma(j))}= 1$ (here $h_1=\max_{\ell=1}^m h_\ell$).}
\end{figure}

\begin{lemma}[basis of $S(j)$]\label{C1} The functions $\varphi_0,\dots,\varphi_m,\psi$ form a basis of $S(j)$.
\end{lemma}
\begin{proof}
	Let $\alpha,\alpha_0,\dots,\alpha_m\in\mathbb{R}$ satisfy $\sum_{\ell=0}^m\alpha_\ell\varphi_\ell+\alpha\psi=0$. This implies, for any $k=0,\dots,m$, that \[0=\sum_{\ell=0}^m\alpha_\ell\varphi_\ell(s_k)+\alpha\psi(s_k)=\sum_{\ell=0}^m\alpha_\ell\delta_{k\ell }=\alpha_k.\]
	The function $\psi$ attains a positive value at the midpoint $h_1/2$. So $\alpha\psi(h_1/2)=0$ shows that $\alpha=0$. Consequently, $\varphi_0,\dots,\varphi_m,\psi$ are linearly independent. Since $\text{dim}(S(j))=m+2$, they form a basis of $S(j)$.
\end{proof}
\noindent Given the basis functions $\varphi_0,\dots,\varphi_m,\psi$ as in Lemma~\ref{C1}, define 
\begin{align*}
\psi_\ell:=\varphi_\ell-\Lambda_j(\varphi_\ell)\psi\quad\text{for all}\;\ell=0,\dots,m\quad\text{and}\quad \psi_{m+1}:=\psi.
\end{align*}

\begin{lemma}[finite element and 1D stability]\label{1d} The triple $(\gamma(j), S(j), (\text{dof}_{k(j)},\dots,\text{dof}_{k(j)+m(j)},$\\$\Lambda_j))$ forms a finite element in the sense of Ciarlet. The  functions $\psi_0,\dots,\psi_{m+1}$ form a nodal basis of $S(j)$  and
	\begin{align*}1\leq  \|\psi_\ell\|_{L^{\infty}(\gamma(j))}\leq C_{11}\quad\text{for all}\;\ell=0,\dots,m+1\end{align*}
	holds with a positive constant $C_{11}$ (that exclusively depends on $\rho$).
\end{lemma}
\begin{proof} 
	The functions $\psi_0,\dots,\psi_{m+1}$ satisfy, for all $k,\ell=0,\dots,m$, the duality relations \begin{align*}&\text{dof}_k(\psi_\ell)=\psi_\ell(s_k)=\varphi_\ell(s_k)=\delta_{k\ell},\quad \text{dof}_k(\psi_{m+1})=\psi_{m+1}(s_k)=\psi(s_k)=0,\\   &\Lambda_j(\psi_\ell)=\Lambda_j(\phi_\ell)-\Lambda_j(\phi_\ell)\Lambda_j(\psi)=0,\quad \Lambda_j(\psi_{m+1})=\Lambda_j(\psi)=1.
	\end{align*}  
	Hence the  functions $\psi_0,\dots,\psi_{m+1}$ form a nodal basis of $S(j)$. Consequently the triple $(\gamma(j), S(j),$\\$ (\text{dof}_{k(j)},\dots,\text{dof}_{k(j)+m(j)},\Lambda_j)$ forms a finite element in the sense of Ciarlet.	Recall $\|\psi\|_{L^\infty(\gamma(j))}=1$ and notice $\|\psi_{m+1}\|_{L^\infty(\gamma(j))}=\|\psi\|_{L^\infty(\gamma(j))}= 1$. The definitions of $\varphi_\ell$ show, for $\ell=0,\dots,m$, that \[1\leq\|\psi_\ell\|_{L^{\infty}(\gamma(j))}\leq \|\varphi_\ell\|_{L^{\infty}(\gamma(j))}+|\Lambda_j(\varphi_\ell)|\leq (1+C_{\Lambda})\|\varphi_\ell\|_{L^{\infty}(\gamma(j))}\leq (1+C_{\Lambda})C_{10}\]
	with the assumption $\|\Lambda_j\|\leq C_{\Lambda}$ from Subsection~3.2.2  in the second last step and $\|\varphi_\ell\|_{L^{\infty}(\gamma(j))}\leq C_{10}$ from Lemma~\ref{lemc} in the last step.
\end{proof}

\subsection{Proof of \textbf{(H1)}}
Recall $\text{Dof}(\bullet)=(\text{dof}_1(\bullet),\dots,\text{dof}_{2N_P}(\bullet))$ from Subsection~2.1 for the polygon $P$ with $N_P$ edges   and $J$ sides  and recall $\Lambda_1,\dots,\Lambda_J$ from Subsection~3.2.2.  Let $\Lambda(\bullet):=(\Lambda_1(\bullet),\dots,\Lambda_J(\bullet))$ and abbreviate $\text{Dof}\oplus\Lambda:=(\text{dof}_1,\dots,\text{dof}_{2N_P},\Lambda_1,\dots,\Lambda_J):H^2(P)\to \mathbb{R}^{2N_P+J}$.

\bigskip
\noindent\textit{Step 1 designs an HCT interpolation}.    Given  $w_h\in W_h(P)$, its trace $w_h|_{\gamma(j)}\in S(j)\subset C^1(\gamma(j))$ allows for well-defined tangential derivatives at all the vertices, that are utilized to define the HCT interpolation as follows. Define the normal derivative of $w_\h\in\h(\tT(P))$ to be zero  at any vertex $z$, which is not  a corner. This and the tangential derivative uniquely define $\nabla w_h(z)=\nabla w_\h(z)$. There are two linearly independent tangential derivatives at a corner $\zeta_j$ along $\gamma(j)$ and $\gamma(j+1)$ and they uniquely define a vector $\nabla w_h(\zeta_j)=\nabla w_\h(\zeta_j)$ for $j=1,\dots,J$. Those values and the point evaluations $w_h(z)=w_\h(z)$ at $z\in\mathcal{V}(P)$ allow for the  HCT interpolation $w_\h\in\h(\tT(P))$ of $w_h\in W_h(P)$  with \[w_\h(z_k)=w_h(z_k),\quad \nabla w_\h(z_k)=\nabla w_h(z_k),\quad\text{and}\quad \dashint_{E(\ell)}(w_\h)_\bn\,ds=\dashint_{E(\ell)}(w_h)_\bn\,ds\] for all $k,\ell=1,\dots,N_P$, while  all other dofs of $w_\h$ in the HCT finite element space  vanish. Let $\psi_1^\h,\dots,\psi_{4N_P}^\h$ denote the $4N_P$ nodal basis functions in $\h(\tT(P))$ with $\psi_1^\h,\dots,\psi_{2N_P}^\h$  from Step 1 in Appendix~B. For any $k=2N_P+1,\dots,4N_P$, define some remaining nodal basis functions $\psi_k^\h\in\h(\tT(P))$ uniquely by $\psi_k^\h(z_\ell)=0=\dashint_{E(\ell)}(\psi_k^\h)_{\bn}\,ds$ for $\ell=1,\dots,N_P$, and \begin{align*}
&\partial_x(\psi_k^\h)(z_\ell)=\delta_{(k-2N_P)\ell},\quad \partial_y(\psi_k^\h)(z_\ell)=0 \quad\text{if}\;k=2N_P+1,\dots,3N_P,\\&\partial_x(\psi_k^\h)(z_\ell)=0,\quad \partial_y(\psi_k^\h)(z_\ell)=\delta_{(k-3N_P)\ell} \quad\text{if}\;k=3N_P+1,\dots,4N_P.
\end{align*} The scaling of the standard HCT basis functions from \cite[Prop. 2.5]{gallistl2015morley}, the observation  from Step 1 in Appendix~B for the scaling of $\psi_1^\h,\dots,\psi_{4N_P}^\h$, and the bound $h_T^{-1}\leq \rho^{-1}h_P^{-1}$ for all $T\in\tT(P)$ from (M2)  provide a positive constant $C_\h$ (that exclusively depends on $\rho$) in \begin{align}\max_{k=1}^{N_P}h_P|\psi_k^\h|_{2,P}+\max_{\ell=1}^{3N_P}|\psi_{\ell+N_P}^\h|_{2,P}\leq C_\h.\label{scale1}\end{align} (This extends \eqref{scale} with  a possibly different constant $C_\h$ for  additional $2N_P$ nodal basis functions $\psi_{2N_P+1}^\h,\dots,\psi_{4N_P}^\h$.)   Then,  given $w_h\in W_h(P)$, the HCT interpolation $w_\h$ reads
\begin{align*}w_\h&=\sum_{k=1}^{N_P}\text{dof}_k(w_h)\psi_k^{\h}+\sum_{k=N_P+1}^{2N_P}\text{dof}_k(w_h)h_{E(k-N_P)}^{-1}\psi_k^{\h}+\sum_{k=2N_P+1}^{3N_P}\partial_x(w_h)(z_{k-2N_P})\psi_k^\h\\&\quad+\sum_{k=3N_P+1}^{4N_P}\partial_y(w_h)(z_{k-3N_P})\psi_k^\h.
\end{align*}

\bigskip
\noindent\textit{Step 2 defines the Hilbert space $W_0$}. The kernel $\{v\in H^2(P):\forall \ell=1,\dots,N_P\quad\int_{E(\ell)}v_{\bn}\,ds=0\}$ of the linear map $(\text{dof}_{N_P+1},\dots,\text{dof}_{2N_P}): H^2(P)\cap H^1_0(P)\to \mathbb{R}^{N_P}$ is a  Hilbert space $H^2(P)$ and its intersection \[W_0:=\left\{v\in W: \int_{E(\ell)}v_\bn\,ds=0\quad\text{for}\;\ell=1,\dots,N_P\right\}\subset W:=H^1_0(P)\cap H^2(P)\]    is  complete. Therefore $(W_0,a^P)$ is a Hilbert space.

\bigskip
\noindent\textit{Step 3 proves $w_h-w_\h\in W_0$ for any $w_h\in W_h(P)$ and its HCT interpolation $w_\h\in\h(\tT(P))$ from Step 1}. The design  leads to $(w_h-w_\h)|_{\gamma(j)}\in\p_3(\e(\gamma(j)))$. Step 1 shows that $(w_h-w_\h)(z_k)$ and $\nabla(w_h-w_\h)(z_k)$ vanish for all $k=0,\dots,m$. These values and $(w_h-w_\h)|_{\gamma(j)}\in\p_3(\e(\gamma(j)))\cap C^1(\gamma(j))$  uniquely determine $(w_h-w_\h)|_{\gamma(j)}=0$   for  $j=1,\dots,J$ and so $(w_h-w_\h)|_{\partial P}=0$.  Consequently, $w_h-w_\h\in W$. Since $\int_{E(\ell)}(w_\h)_{\bn}\,ds=\int_{E(\ell)}(w_h)_{\bn}\,ds$ for all $\ell=1,\dots,N_P$ (by design of $w_\h$), $w_h-w_\h\in W_0$ follows from the definition of $W_0$ in Step 2. \qed

\bigskip
\noindent\textit{Step 4 establishes a sufficient criterion for the inclusion in $\widehat{W}_h(P)$}: If $w\in H^2(P)$ and $f\in\p_r(P)$ satisfy $a^P(w,\phi_0)=(f,\phi_0)_{L^2(P)}$ for all $\phi_0\in W_0$, then there exists $g\in\p_0(\e(P))$ such that 
\[a^P(w,\phi)=(f,\phi)_{L^2(P)}+(g,\phi_{\bn})_{L^2(\partial P)}\quad\text{for all}\;\phi\in W.\]

\medskip
\noindent\textit{Proof}: Let $\varphi_\ell\in S^1(\tT(P)):=\{v\in C^0(\tT(P)): \forall T\in\tT(P)\quad v|_T\in\p_1(T)\}$  be the nodal basis functions in the Courant FEM ($\p_1$ conforming) associated with  the vertices $z_\ell$  and define the quadratic edge-bubble function $b_{E(\ell)}:=4\varphi_\ell\varphi_{\ell+1}$ with support $T(E(\ell))$ depicted in Figure~2.1.b. Define $\psi_1,\dots,\psi_{N_P}\in W$ by
\begin{align*}
\psi_{\ell}(x):=\frac{8}{15}(x-\text{mid}(E(\ell)))\cdot\bn_{E(\ell)}b^2_{E(\ell)}(x)\quad\text{for all}\;x\in P
\end{align*}
so that  $\dashint_{E(k)}(\psi_\ell)_\bn\,ds=\delta_{k\ell}$ for all $k,\ell=1,\dots,N_P$. Given any $\phi\in W$, define
\[\phi_0:=\phi-\sum_{\ell=1}^{N_P}\Big(\int_{E(\ell)}\phi_{\bn}\,ds\Big)\psi_\ell\in W_0.\]
Then $a^P(w,\phi)=a^P(w,\phi_0)+a^P(w,\phi-\phi_0)$ and the assumption $a^P(w,\phi_0)=(f,\phi_0)_{L^2(P)}$ imply 
\begin{align*}
a^P(w,\phi)=(f,\phi_0)_{L^2(P)}+\sum_{\ell=1}^{N_P}\Big(\int_{E(\ell)}\phi_{\bn}\,ds\Big)a^P(w,\psi_\ell)=(f,\phi)_{L^2(P)}+(g,\phi_{\bn})_{L^2(\partial P)}
\end{align*}
for $g\in\p_0(\e(P))$ with $g|_{E(\ell)}:=a^P(w,\psi_\ell)-(f,\psi_\ell)_{L^2(P)}$ for $\ell=1,\dots,N_P$ in the last step. This concludes the proof of the claim.\qed

\bigskip
\noindent\textit{Step 5 proves that $\text{Dof}\, \oplus\,\Lambda:\widehat{W}_h(P)\to \mathbb{R}^{2N_P+J}$ is surjective.} Given any $x\in\mathbb{R}^{2N_P+J}$, Lemma~\ref{1d} leads to some $w_j\in S(j)$ such that $\text{dof}_{k(j)+\ell}(w_j)=x_{k(j)+\ell}$ for $\ell=0,\dots,m(j)$ and $\Lambda_j(w_j)=x_{2N_P+j}$. This holds for all $j=1,\dots,J$ and   defines a continuous $w\in \p^2(\e(P))\cap C^0(\partial P)$ with $w|_{\gamma(j)}=w_j\in S(j)$ on the boundary $\partial P$. Since $w\in C^0(\partial P)$ satisfies $w|_{\gamma(j)}\in C^1(\gamma(j))$ for all $j=1,\dots,J$, the tangential derivatives of $w$ define $\nabla w(\zeta_j)$ at each corner point $\zeta_j$, while at each other vertex $z\in\mathcal{V}(P)$ (that is not a corner) the tangential derivative $w'(z)=w_{\bt}(z)$ and the vanishing normal derivative determine a unique vector $\nabla w(z)\in\mathbb{R}^2$. Given those values of $w$ and $\nabla w$ at all vertices in $\mathcal{V}(P)$, define an HCT interpolation $u_\h\in\h(\tT(P))$ as in Step 1  with
\[u_\h(z_k)=w(z_k),\quad \nabla u_\h(z_k)=\nabla w(z_k),\quad\text{and}\quad \dashint_{E(\ell)}(u_\h)_\bn\,ds=x_{\ell+N_P}\]
for $k,\ell=1,\dots,N_P$, while all other dofs vanish. This and Step 3 reveal that $u_\h|_{\partial P}=w|_{\partial P}$. Let $u_0\in W_0$ denote the Riesz representation of the linear functional $a^P(u_\h,\cdot)$ in the Hilbert space $(W_0,a^P)$, i.e., $a^P(u_0,\cdot)=a^P(u_\h,\cdot)$ in $W_0$. Let $\widehat{u}_h:=u_\h-u_0$ and deduce $\widehat{u}_h|_{\partial P}=u_\h|_{\partial P}\in S^2(\e(P))$ from $u_0|_{\partial P}=0$. Since $a^P(\widehat{u}_h,\varphi_0)=0$  for all $\varphi_0\in\mathcal{D}(\text{int}(P))$, it follows $\Delta^2\widehat{u}_h=0$ in $P$.  Step 4 with $f=0\in\p_{-1}(P)$ implies the existence of $g\in\p_0(\e(P))$ with
\[a^P(\widehat{u}_h,\phi)=(g,\phi_{\bn})_{L^2(\partial P)}\quad\text{for all}\;\phi\in W\]
and proves that $\widehat{u}_h\in \widehat{W}_h(P)$. Recall $(\text{Dof}\,\oplus\,\Lambda)(\widehat{u}_h)=x$ from the design of $w$ and $w_\h$ in the very beginning of the proof. \qed

\bigskip
\noindent\textit{Step 6 establishes an inclusion in $\widehat{W}_h(P)$ with a non-zero $f$}. Given $f\in \p_r(P)$, the Riesz-representation theorem guarantees the unique existence of the weak solution $u(f)\in W_0$ to
\begin{align}a^P(u(f),v)=(f,v)_{L^2(P)}\quad\text{for all}\;v\in W_0.\label{A6}\end{align} 
Consequently  $\Delta^2u(f)=f$ in $P$.  This and Step 4 show that $u(f)\in\widehat{W}_h(P)$.

\bigskip
\noindent\textit{Step 7 proves that $\mathcal{L}:\p_r(P)\to\p_r(P), f\mapsto \Pi_ru(f)$  is an isomorphism}. For any $f\in\p_r(P)$ with $\mathcal{L}f=0$, the orthogonality $(1-\Pi_r)u(f)\perp\p_r(P)$ in $L^2(P)$ shows $0=\int_P(\mathcal{L}f)f\,dx=\int_P u(f)f\,dx$. This and $v=u(f)$ in \eqref{A6} result in
$
0=a^P(u(f),u(f))=|u(f)|_{2,P}^2.
$
Consequently $u(f)\in\p_1(P)\cap W_0$ and so $u(f)=0$. Thus  $f=\Delta^2u(f)=0$ and $\mathcal{L}$ is injective; whence  bijective.\qed

\bigskip
\noindent\textit{Step 8 proves that $\mathrm{Dof}:W_h(P)\to \mathbb{R}^{2N_P}$ is an isomorphism}.

\medskip
\noindent\textit{Proof of surjectivity}:  
Given   $x\in\mathbb{R}^{2N_P},$ there exists some $v\in H^2(P)$ with $\text{Dof}(v)=x$ (for a proof we may utilize Step 5 for $(x,0,\dots,0)\in\mathbb{R}^{2N_P+J}$ and obtain at least one $v\in \widehat{W}_h(P)\subset H^2(P)$). Given $v\in H^2(P)$, let $\chi:=Gv\in\p_2(P)$ and notice $\chi$ is computable with Lemma~2.6 from $x$ in a unique way.  Set $y_j:=\Lambda_j(\chi|_{\gamma(j)})\in\mathbb{R}$ for all $j=1,\dots,J$. Step 5 proves that given  $y=(x_1,\dots,x_{2N_P},y_1,\dots,y_J)\in\mathbb{R}^{2N_P+J}$,  there exists some $\widehat{u}_h\in\widehat{W}_h(P)$ with $(\text{Dof}\,\oplus\,\Lambda)(\widehat{u}_h)=y$, i.e.,   $\text{Dof}(\widehat{u}_h)=x$ and $\Lambda_j(\widehat{u}_h|_{\gamma(j)})= y_j$ for $j=1,\dots,J$. Since $G\widehat{u}_h=Gv=\chi$ for any $v\in H^2(P)$ with $\text{Dof}(v)=x$ from Lemma~2.6,  $\Lambda_j(\widehat{u}_h|_{\gamma(j)})= y_j=\Lambda_j(\chi|_{\gamma(j)})=\Lambda_j(G\widehat{u}_h|_{\gamma(j)})$.  This leads to $g:=\Pi_rG\widehat{u}_h-\Pi_r\widehat{u}_h\in\p_r(P)$. Since $\mathcal{L}$ is bijective in $\p_r(P)$ (from Step 7), there exists $f\in\p_r(P)$ with $\Pi_ru(f)=g$. Recall $u(f)\in \widehat{W}_h(P)\cap W_0$ from Step 6 and $\text{Dof}(u(f))=0, G(u(f))=0,$ and $\Lambda(u(f))=0$. Altogether, $u_h^P:=u(f)+\widehat{u}_h\in\widehat{W}_h(P)$ satisfies $\text{Dof}(u_h^P)=x$, $\Lambda(u_h^P-Gu_h^P)=0$, and $\Pi_rGu_h^P=\Pi_rG\widehat{u}_h=g+\Pi_r\widehat{u}_h=\Pi_ru_h^P$. This concludes the proof of $u_h^P\in W_h(P)$ with $\text{Dof}(u_h^P)=x$.\qed

\medskip
\noindent\textit{Proof of injectivity}. Suppose $w_h\in W_h(P)$ satisfies $\text{Dof}(w_h)=0$. This and Lemma~2.6 imply  $Gw_h=0$ and hence the condition $\Lambda_j((w_h-Gw_h)|_{\gamma(j)})=0$ in (3.4) shows $\Lambda_j(w_h)=0$. Lemma~\ref{1d} implies that  $\text{Dof}(w_h)$ and $\Lambda_j(w_h)$ uniquely determine  $w_h|_{\gamma(j)}=0$ for all $j=1,\dots,J$; whence $w_h|_{\partial P}=0$. The substitution of $w=\phi=w_h$ in (3.3), the $L^2$ orthogonality of $\Pi_r$, $f\in\p_r(P) $, and $\Pi_rw_h=\Pi_r(Gw_h)=0$ from (3.4) result in
\begin{align*}
|w_h|^2_{2,P}=a^P(w_h,w_h)=(f,w_h)_{L^2(P)}=(f,\Pi_r(Gw_h))_{L^2(P)}=0.
\end{align*}
Consequently $w_h\in\p_1(P)\cap W_0$, whence $w_h=0$.\qed

\bigskip
\noindent\textit{Proof of \textbf{(H1)}}. The key observation from Step 8 is that $W_h(P)$ has the dimension $2N_P$ and $\text{dof}_1,\dots,\text{dof}_{2N_P}$ from (2.1) are linear independent. Consequently $(P,W_h(P),(\text{dof}_1,\dots,\text{dof}_{2N_P}))$ is a finite element in the sense of Ciarlet.\qed


\subsection{Proof of \textbf{(H2)}}
\noindent\textit{Step 1 defines an HCT interpolation of a nodal basis function of $W_h(P)$}.  Let $\psi_h\equiv\psi_p$ be a nodal basis function    of the finite element $(P,W_h(P),\text{Dof})$ for some $p\in\{1,\dots,2N_P\}$ and let $\psi_\h$ be its   HCT interpolation  as in Step 1 from Subsection~C.2, namely
\begin{align*}\psi_\h&=\sum_{k=1}^{N_P}\text{dof}_k(\psi_h)\psi_k^{\h}+\sum_{k=N_P+1}^{2N_P}\text{dof}_k(\psi_h)h_{E(k-N_P)}^{-1}\psi_k^{\h}+\sum_{k=2N_P+1}^{3N_P}\partial_x(\psi_h)(z_{k-2N_P})\psi_k^\h\\&\quad+\sum_{k=3N_P+1}^{4N_P}\partial_y(\psi_h)(z_{k-3N_P})\psi_k^\h.
\end{align*}

\bigskip
\noindent\textit{Step 2 proves that $|\nabla(\psi_h)(z_{k})|\lesssim h_P^{-1}$ for $k=1,\dots,N_P$}. Recall from Step 1 in Subsection C.2 that the two linearly independent tangential derivatives uniquely define $\nabla\psi_h(\zeta)$ at each corner $\zeta$, and the tangential derivative and vanishing normal derivative uniquely define $\nabla\psi_h(z)$ at each vertex $z$, which is not a corner.
Expand $\psi_h|_{\gamma(j)}$ in terms of the finite element  $(\gamma(j),S(j),(\text{dof}_{k(j)},\dots,\text{dof}_{k(j)+m(j)},\Lambda_j))$ from Lemma~\ref{1d}, write $(\cdot)':=(\cdot)_{\bt}$ for the derivative along $\gamma(j)$, and deduce that
\begin{align}(\psi_h)_{\bt}|_{\gamma(j)}=\sum_{\ell=0}^{m}\text{dof}_\ell(\psi_h) \psi_\ell'+\Lambda_j(\psi_h|_{\gamma(j)})\psi_{m+1}'.\label{B3}\end{align}
The definition of $\psi_h$ shows that $|\text{dof}_\ell(\psi_h)|\leq 1$  for $\ell=0,\dots,m$. The definition of $W_h(P)$ from (3.4) and of $\|\Lambda_j\|$ from Subsection~3.2.2 imply \begin{align*}|\Lambda_j(\psi_h|_{\gamma(j)})|=|\Lambda_j(G\psi_h|_{\gamma(j)})|\leq C_\Lambda\|G\psi_h\|_{L^{\infty}(\gamma(j))}\leq C_\Lambda C_{\text{inv}}L^{-1/2}\|G\psi_h\|_{L^2(\gamma(j))}\end{align*} with an inverse estimate for $G\psi_h|_{\gamma(j)}\in\p_2(\gamma(j))$ in the last step. This and  the trace inequality \cite[p.~554]{14} show $|\Lambda_j(\psi_h|_{\gamma(j)})|\leq C_\Lambda C_{\text{inv}}C_T(L^{-1}\|G\psi_h\|_{L^2(P)}+|G\psi_h|_{1,P})$. Consequently, $L^{-1}\leq \rho^{-1}h_P^{-1}$ from (M2),  Lemma~2.6,  and $|\text{Dof}(\psi_h)|_{\ell^2}=1$ result in
\begin{align}
|\Lambda_j(\psi_h|_{\gamma(j)})|\leq C_\Lambda C_{\text{inv}}C_TC_g(1+\rho^{-1})|\text{Dof}(\psi_h)|_{\ell^2}\leq C_\Lambda C_{\text{inv}}C_TC_g(1+\rho^{-1}).\label{B4}
\end{align}
The inverse inequality for the piecewise quadratic polynomial $\psi_\ell\in\p_2(\e(\gamma(j)))$ with $h_E^{-1}\leq \rho^{-1}h_P^{-1}$ and Lemma~\ref{1d} lead, for $\ell=0,\dots,m+1$, to 
\begin{align*}
|\psi_\ell|_{1,\infty,E}\leq C_{\text{inv}}\rho^{-1}h_{P}^{-1}\|\psi_\ell\|_{L^\infty(E)}\leq C_{\text{inv}}\rho^{-1}C_{11}h_{P}^{-1} \quad\text{for all}\;E\in\e(\gamma(j)).
\end{align*}
The above estimate for each $\ell=0,\dots,m$, the combination \eqref{B3}-\eqref{B4}, and $1+m\leq N_P\leq M(\rho)$ from Subsection~2.1 prove  for a positive constant  $C_{12}:=C_{\text{inv}}C_{11}\rho^{-1}M(\rho)(1+C_\Lambda C_{\text{inv}}C_TC_g)$ (that exclusively depends on $\rho$) that
\begin{align}|(\psi_h)_{\bt}|_{\gamma(j)}(s_k)|\leq|\psi_h|_{1,\infty,\gamma(j)}\leq C_{12} h_P^{-1}\quad\text{for}\; k=0,\dots,m.\label{C5}\end{align} Since the normal derivative at a vertex $s_k$ ($= z_{k(j)+k}$), which is not a corner is zero for $k=1,\dots,m-1$, \eqref{C5} shows that $|\nabla(\psi_h)(s_k)|\leq C_{12} h_P^{-1}$ for $k=1,\dots,m-1$. The expansion in \eqref{B3} for $(\psi_h)_{\bt}|_{\gamma(j-1)}$ leads to $|(\psi_h)_{\bt}|_{\gamma(j-1)}(\zeta_j)|\leq C_{12} h_P^{-1}$. This and \eqref{C5} prove that $|\nabla(\psi_h)(\zeta_j)|\leq C(\omega_j)h_P^{-1}$ with a positive constant $C(\omega_j)$ that depends on $C_{12}$ and on the interior angle $\omega_j\neq\pi$ at the corner $\zeta_j$.  This holds for all $j=1,\dots,J$ and concludes the proof.\qed

\bigskip
\noindent\textit{Step 3 provides the scaling of the HCT interpolation}. The definition of $\psi_h$ shows  $\text{dof}_k(\psi_h)=\delta_{kp}$ for $k,p=1,\dots,2N_P$.  The scaling of $\psi_k^{\h}$  from \eqref{scale1} for the first $2N_P$ indices $k=1,\dots,2N_P$  and $h_E^{-1}\leq\rho^{-1}h_P^{-1}$ for $E\in\e(P)$ from (M2) show 
\[\sum_{k=1}^{N_P}|\text{dof}_k(\psi_h)||\psi_k^{\h}|_{2,P}+\sum_{k=N_P+1}^{2N_P}|\text{dof}_k(\psi_h)|h_{E(k-N_P)}^{-1}|\psi_k^{\h}|_{2,P}\leq C_\h(1+\rho^{-1}) h_P^{-1}.\]
The scaling $|\psi_k^\h|_{2,P}\leq C_\h$ from \eqref{scale1}   and from Step 2 for the remaining $2N_P$ indices $k=2N_P+1,\dots,4N_P$ prove with $C_{13}:=C(\omega_1)+\dots+C(\omega_{J})+C_{12}M(\rho)$ that\begin{align*}&\sum_{k=2N_P+1}^{3N_P}|\partial_x(\psi_h)(z_{k-2N_P})||\psi_k^\h|_{2,P}+\sum_{k=3N_P+1}^{4N_P}|\partial_y(\psi_h)(z_{k-3N_P})||\psi_k^\h|_{2,P}\\&\qquad\leq C_\h\sum_{\ell=1}^{N_P}|\nabla\psi_h(z_\ell)|\leq C_\h C_{13}h_P^{-1}.\end{align*} The previous two displayed estimates lead in the representation of $\psi_\h$ from Step 1 to \begin{align}|\psi_\h|_{2,P}\leq C_\h(1+\rho^{-1}+C_{13}) h_P^{-1}.\label{C6}\end{align}

\bigskip
\noindent\textit{Proof of \textbf{(H2)}}.  Step 3 in Subsection~C.2 shows for the  nodal basis function $\psi_h\in W_h(P)$ that $(\psi_h-\psi_\h)|_{\partial P}=0$. Hence the test function $\phi=\psi_h-\psi_\h$  leads in (3.3) to \begin{align}a^P(\psi_h,\psi_h-\psi_\h)=(f,\psi_h-\psi_\h)_{L^2(P)}\label{C7}\end{align} and it remains to control $\|f\|_{L^2(P)}\lesssim h_P^{-2}|\psi_h|_{2,P}$.  The analogous arguments in Step 10-11 from Appendix~B apply to $\psi_h\in W_h(P)$  and  its HCT interpolation $\psi_\h\in\h(\tT(P))$ from Step 1.  This leads here to $(f,\psi_h-\psi_\h)_{L^2(P)}=(f,\Pi_rG\psi_\h-\psi_\h)_{L^2(P)}$. The arguments in Step 10 of Appendix~B provide $\|\Pi_rG\psi_\h-\psi_\h\|_{L^2(P)}\leq C_\PF(1+2C_P)h_P^2|\psi_\h|_{2,P}$. The sole modification in the arguments concerns  the equality $(f,b_Pf)_{L^2(P)}=a^P(\psi_h,f)$, which follows from $\phi=b_Pf\in H^2_0(P)$ in (3.3). The remaining arguments in Step 11 apply here verbatim and lead to $\|f\|_{L^2(P)}\leq C_b^2h_P^{-2}|\psi_h|_{2,P}$.   These estimates,  Cauchy-Schwarz inequalities, and \eqref{C7} result in \[|\psi_h|_{2,P}\leq (1+C_b^2(C_\PF(1+2C_{\text{P}}))) |\psi_\h|_{2,P}.\] The combination with \eqref{C6} shows
$|\psi_h|_{2,P}\leq C_{\text{stab}} h_P^{-1}$ with $C_{\text{stab}}:=C_\h(1+\rho^{-1}+C_{13})(1+C_b^2(C_\PF(1+2C_{\text{P}})))$.
This verifies  \textbf{(H2)} for a nodal basis function $\psi_h$  in $W_h(P)$ from (3.3)-(3.4) and concludes the proof.\qed
\begin{remark}[comparison with \cite{zhao2018morley}]
	The discrete space in \cite{zhao2018morley} reads
	\begin{align}
	V_h^s(P):=\begin{cases}
	\begin{rcases}
	&v_h\in H^2(P): \Delta^2 v_h\in\p_r(P)\;  v_h|_{\partial P}\in\p_2(\e(P))\;\Delta v_h|_{\partial P}\in \p_0(\e(P)),\\&\quad\forall{E\in\e(P)}\quad \int_Ev_h\,ds=\int_EGv_h\,ds,\quad v_h-Gv_h\perp\p_r(P)\;\text{in}\;L^2(P)
	\end{rcases}.
	\end{cases}\label{zhao}
	\end{align}
	Recall the sides $\gamma(j)$ of a polygonal domain $P\in\T$ for $j=1,\dots,J$  from Subsection~3.2.2. First notice that, for $v_h\in V_h^s(P)$ in \eqref{zhao}, $v_h|_{\gamma(j)}$ belongs to $C^1(\gamma(j))$ and so  $V_h^s(P)$ from \eqref{zhao} allows only the polygons without hanging nodes (i.e., all vertices are corner points).   Second, as discussed in Remark~\ref{rem1}-\ref{rem2}, we avoid the strong formulation $\Delta v_h|_{\partial P}\in\p_0(\e(P))$ and solely consider  the weak formulation (3.3)-(3.4). Hanging nodes (i.e., vertices that are not corners on $\partial P$) are important for a more flexible mesh-design to allow obligatory adaptive mesh-refining.
\end{remark}
\begin{remark}[individual parameters]
	The selection of the linear functional $\Lambda_j^P:=\Lambda_j:S(j)\to\mathbb{R}$ resp. of the parameter $r_P=r=-1,0,1,2$ is individually for each polygon $P\in\T$ and may be labelled with an index $P$ to underline this. Given an interior side $\gamma(j)\subset\partial P_+\cap\partial P_-$ shared by two polygons $P_+,P_-\in\T$, $\Lambda_j^{P_+}$ and $\Lambda_j^{P_-}$ resp. $r_{P_+}$ and $r_{P_-}$ could be different in general. A single selection $\Lambda_j^{P_+}=\Lambda_j^{P_-}$ is also possible and could be even more appealing; but it does not imply $C^0$ conformity: The jump $[w]_{\gamma(j)}\in\text{span}\{\psi_j\}$ for $w\in V_\nc$ with $w|_{P_\pm}\in W_h(P)$ cannot be expected to vanish; so the schemes are fully nonconforming.
\end{remark}

\end{document}